\documentclass[11pt,letterpaper]{amsart}

\usepackage{amsmath,amsthm,amssymb,amsfonts,mathrsfs,color,hyperref,xcolor}
\usepackage{bbm}
\usepackage{amsbsy}
\usepackage{latexsym}
\usepackage{faktor}
\usepackage[utf8]{inputenc}
\usepackage{cite}
\usepackage{enumerate}
\usepackage[shortlabels]{enumitem}

\allowdisplaybreaks

\theoremstyle{plain}
\begingroup
\newtheorem{thm}{Theorem}[section]
\newtheorem{lem}[thm]{Lemma}
\newtheorem{prop}[thm]{Proposition}

\endgroup

\theoremstyle{definition}
\begingroup
\newtheorem{defn}[thm]{Definition}
\newtheorem{rem}[thm]{Remark}
\newtheorem{example}[thm]{Example}
\endgroup

\numberwithin{equation}{section}

\newcommand{\res}{\mathop{\hbox{\vrule height 7pt width .5pt depth 0pt
\vrule height .5pt width 6pt depth 0pt}}\nolimits}

\newcommand{\N}{\mathbb N} 
\newcommand{\R}{\mathbb R} 

\newcommand{\wto}{\rightharpoonup}
\newcommand{\wsto}{\stackrel{*}{\rightharpoonup}}
\newcommand{\e}{\varepsilon}

\newcommand{\LL}{{\mathcal L}}
\newcommand{\HH}{{\mathcal H}}
\newcommand{\M}{{\mathcal M}}
\newcommand{\A}{{\Acal}}

\newcommand{\restrict}{\begin{picture}(10,8)\put(2,0){\line(0,1){7}}\put(1.8,0){\line(1,0){7}}\end{picture}}

\newcommand{\Crm}{\mathrm{C}}

\newcommand{\Hrm}{\mathrm{H}}

\newcommand{\Lrm}{\mathrm{L}}

\newcommand{\Acal}{\mathcal{A}}

\newcommand{\Dcal}{\mathcal{D}}

\newcommand{\Lcal}{\mathcal{L}}

\newcommand{\Ocal}{\mathcal{O}}

\newcommand{\Rcal}{\mathcal{R}}

\newcommand{\Abb}{\mathbb{A}}

\newcommand{\Cbb}{\mathbb{C}}

\DeclareMathOperator{\dist}{dist}

\DeclareMathOperator{\supp}{supp}
\DeclareMathOperator{\diag}{diag}

\DeclareMathOperator{\Glim}{\Gamma\mathrm{-lim}}

\newcommand{\ee}{\mathrm{e}}
\newcommand{\ii}{\mathrm{i}}
\newcommand{\set}[2]{\left\{\, #1 \ \ \textup{\textbf{:}}\ \ #2 \,\right\}}

\newcommand{\setb}[2]{\bigl\{\, #1 \ \ \textup{\textbf{:}}\ \ #2 \,\bigr\}}
\newcommand{\setB}[2]{\Bigl\{\, #1 \ \ \textup{\textbf{:}}\ \ #2 \,\Bigr\}}
\newcommand{\setBB}[2]{\biggl\{\, #1 \ \ \textup{\textbf{:}}\ \ #2 \,\biggr\}}

\newcommand{\norm}[1]{\|#1\|}

\newcommand{\normb}[1]{\bigl\|#1\bigr\|}

\newcommand{\abs}[1]{|#1|}

\newcommand{\absBB}[1]{\biggl|#1\biggr|}

\newcommand{\altnorm}[1]{{\left\vert\kern-0.25ex\left\vert\kern-0.25ex\left\vert #1 \right\vert\kern-0.25ex\right\vert\kern-0.25ex\right\vert}}

\newcommand{\dprb}[1]{\bigl\langle #1 \bigr\rangle}

\newcommand{\cl}[1]{\overline{#1}}
\newcommand{\di}{\mathrm{d}}
\newcommand{\dd}{\;\mathrm{d}}

\newcommand{\toup}{\uparrow}
\newcommand{\todown}{\downarrow}

\newcommand{\BigO}{\mathrm{\textup{O}}}
\newcommand{\SmallO}{\mathrm{\textup{o}}}
\newcommand{\sbullet}{\begin{picture}(1,1)(-0.5,-2.5)\circle*{2}\end{picture}}
\newcommand{\frarg}{\,\sbullet\,}

\def\XXint#1#2#3{{\setbox0=\hbox{$#1{#2#3}{\int}$} 
\vcenter{\hbox{$#2#3$}}\kern-.5\wd0}}

\DeclareMathOperator{\tr}{tr}
\DeclareMathOperator{\Id}{Id}
\DeclareMathOperator{\dive}{div}
\DeclareMathOperator{\cof}{cof}
\newcommand{\Mstwo}{\mathbb{M}^{2{\times}2}_{\mathrm{sym}}}
\newcommand{\Msn}{\mathbb{M}^{n{\times}n}_{\mathrm{sym}}}

\begin{document}
 
\title[Shape optimization of light structures]{Shape optimization of light structures and the vanishing mass conjecture}

\author[J.-F. Babadjian]{Jean-Fran\c cois Babadjian}
\address[J.-F. Babadjian]{Universit\'e Paris-Saclay, CNRS, Laboratoire de math\'ematiques d'Orsay, 91405, Orsay, France}
\email{jean-francois.babadjian@universite-paris-saclay.fr}

\author[F. Iurlano]{Flaviana Iurlano}
\address[F. Iurlano]{Sorbonne Universit\'e, CNRS, Universit\'e de Paris, Laboratoire Jacques-Louis Lions, F-75005 Paris, France}
\email{iurlano@ljll.math.upmc.fr}

\author[F. Rindler]{Filip Rindler}
\address[F. Rindler]{Mathematics Institute, University of Warwick, Coventry CV4 7AL, UK.}
\email{F.Rindler@warwick.ac.uk}

%\date{\today}
\subjclass[2010]{}


\begin{abstract}
This work proves rigorous results about the vanishing-mass limit of the classical problem to find a shape with minimal elastic compliance. Contrary to all previous results in the mathematical literature, which utilize a soft mass constraint by introducing a Lagrange multiplier, we here consider the hard mass constraint. Our results are the first to establish the convergence of approximately optimal shapes of (exact) size $\e \todown 0$ to a limit generalized shape represented by a (possibly diffuse) probability measure. This limit generalized shape is a minimizer of the limit compliance, which involves a new integrand, namely the one conjectured by Bouchitt\'e in 2001 and predicted heuristically before in works of Allaire \& Kohn and Kohn \& Strang from the 1980s and 1990s. This integrand gives the energy of the limit generalized shape understood as a fine oscillation of (optimal) lower-dimensional structures. Its appearance is surprising since the integrand in the original compliance is just a quadratic form and the non-convexity of the problem is not immediately obvious. In fact, it is the interaction of the mass constraint with the requirement of attaining the loading (in the form of a divergence-constraint) that gives rise to this new integrand. We also present connections to the theory of Michell trusses, first formulated in 1904, and show how our results can be interpreted as a rigorous justification of that theory on the level of functionals in both two and three dimensions, settling this open problem. Our proofs rest on compensated compactness arguments applied to an explicit family of (symmetric) $\dive$-quasiconvex quadratic forms, computations involving the Hashin--Shtrikman bounds for the Kohn--Strang integrand, and the characterization of limit minimizers due to Bouchitt\'e \& Buttazzo.

\medskip


\medskip

\noindent\textsc{Date:} \today
\end{abstract}

\maketitle

%\tableofcontents

%\newpage

\section{Introduction}

One of the main questions in the theory of shape optimization is the following (see, e.g.,~\cite{Allaire02book,Bouchitte03} for an introduction):
\begin{quote}
Given a bounded Lipschitz domain $\Omega \subset \R^n$ ($n = 2,3$), what is the optimal shape (Lipschitz subdomain) $\omega \subset \Omega$ of prescribed volume $\LL^n(\omega) = \e$ that minimizes a given cost functional?
\end{quote}
Here and in the sequel, a Lipschitz domain is understood to be an open, connected set with Lipschitz-regular boundary. A common choice for the cost functional is the \emph{elastic compliance}~\cite[Section~4.2.1]{Allaire02book}
\begin{equation}\label{eq:hatc}
  \hat{c}(\omega) := -\min_{\hat v \in \Hrm^1(\R^n;\R^n) /\Rcal} \left\{ \int_\omega j(e(\hat v))  \dd x -  \dprb{ f,\hat v } \right\},
  \end{equation}
where $\Hrm^1(\R^n;\R^n) /\Rcal$ denotes the quotient space of $\Hrm^1(\R^n;\R^n)$ by the rigid body motions $\Rcal$, i.e., the affine maps $x \mapsto r(x)=Sx+b$ with $S\in \mathbb M^{n \times n}_{\rm skew}$ a skew-symmetric matrix and $b \in \R^n$. The integrand $j$ is a quadratic form, $j(\xi) = \frac12 \Cbb \xi : \xi$ for symmetric matrices $\xi \in \Msn$, with $\Cbb$ a positive definite and symmetric 4'th-order tensor, which describes the material response as a function of the linearized strain
\[
  e(\hat v):=\frac12(\nabla\hat v+\nabla\hat v^T)
\]
of a displacement $\hat v \in \Hrm^1(\R^n;\R^n)$. Furthermore, $f \in \Hrm^{-1}(\R^n;\R^n)$ is an external loading compactly supported in $\overline \Omega$ and satisfying the balance condition
\[
  \dprb{ f,r }=0  \quad\text{for all $r \in \Rcal$.}
\]
Note that this means that $f$ can itself be written as the divergence of some symmetric matrix field, see Proposition \ref{prop:W1pR*}.

It might happen that $\hat c(\omega)=+\infty$ in~\eqref{eq:hatc} if, e.g., the support of $f$ is not contained in the closure $\overline\omega$ of $\omega$ (so that $\omega$ cannot support the loading $f$). For that reason, as in~\cite[Section 4.2.1]{Allaire02book}, we assume for simplicity that the load
\begin{equation}\label{eq:forces}
f\in\Lrm^2(\partial\Omega;\R^n)
\end{equation}
is concentrated on $\partial\Omega$ and that every admissible shape $\omega$ satisfies
\[
  \partial\Omega\subset\partial\omega.
\]
We thus define for $0 < \e < \LL^n(\Omega)$ the set of \emph{admissible shapes} as follows:
\begin{equation}  \label{eq:Ae}
  \A_\e := \setb{ \omega \subset \Omega }{ \text{$\omega$ Lipschitz domain, $\partial\Omega\subset\partial\omega$, and $\LL^n(\omega) = \e$} }.
\end{equation}
In particular, we require $\omega$ to be connected (which is natural from the point of view of applications). Also observe that competitors in~\eqref{eq:hatc} can then be taken in the space $\Hrm^1(\omega;\R^n)/\Rcal$ instead of $\Hrm^1(\R^n;\R^n)/\Rcal$.

An application of the Direct Method in the Calculus of Variations together with Korn's inequality (see~\cite{Nitsche81,Horgan95,Temam84book}) shows that for each $\omega \in \A_\e$ the minimization problem~\eqref{eq:hatc} has a unique solution $u \in \Hrm^1(\omega;\R^n)/\Rcal$, which satisfies
$$-\dive (\Cbb e(u)\chi_\omega)=f \qquad \text{in }\Dcal'(\R^n;\R^n),$$
where $\chi_\omega$ is the characteristic function of the set $\omega$, i.e.,\ $\chi_\omega(x) = 1$ if $x \in \omega$ and $0$ otherwise, and ``$\dive$'' is the row-wise (distributional) divergence. This means that $u$ is a weak solution of
\[
  \left\{\begin{aligned}
-\dive (\Cbb e(u)) &=0 \quad\text{in }\omega,\\
\Cbb e(u)\nu &=0  \quad\text{on }\partial\omega\setminus\partial\Omega,\\
\Cbb e(u)\nu &=f \quad\text{on }\partial\Omega,
  \end{aligned}\right.
\]
where $\nu$ denotes the unit outer normal to $\partial \omega$.

The shape optimization question corresponds to finding the ``stiffest'' (least compliant) material shape $\omega \in \A_\e$ under the given side conditions. This question is of tremendous practical importance since in a huge number of engineering problems one is tasked with finding an optimal shape of a component. In fact, recent innovations in the engineering space, such as additive manufacturing (e.g., ``3D printing''), have opened up a new level of freedom in shaping materials and thus only increased the importance of this task.

In many applications of shape optimization, the amount of material to be distributed is small relative to the size of the surrounding domain. Thus, immediately the question arises how the optimization behaves in the \emph{vanishing mass limit} $\e \todown 0$. It is clear that in order to investigate this question, one needs to broaden the notion of ``shape'' to also encompass representations of sequences of sets that exhibit fine oscillations or concentrations on (Lebesgue-)nullsets. Then, it is in particular unclear what an optimal ``displacement'' should be. In fact, as is well known in the theory of shape optimization (cf.~\cite[Section~4.2]{Allaire02book}), there may be no limit displacement in the classical sense. In order to arrive at a mathematically tractable formulation, one therefore first dualizes the problem and switches to the stress formulation. Indeed, standard arguments in convex duality theory (see, e.g.,~\cite[Proposition~VI.2.3]{EkelandTemam76book}) allow us to rewrite the minimization in the definition of $\hat c(\omega)$ using the Legendre--Fenchel convex conjugate
\[
  j^*(\tau) := \sup_{\xi \in \Msn} \bigl\{ \xi:\tau - j(\xi) \bigr\},  \qquad \tau \in \Msn,
\]
of $j$ as follows:
\begin{align*}
&\min_{\hat v \in \Hrm^1(\omega;\R^n)/\Rcal} \left\{ \int_\omega j(e(\hat v))  \dd x -  \int_{\partial\Omega}f\cdot \hat v\dd\HH^{n-1}\right\}\\
&\qquad = \min_{\hat v \in \Hrm^1(\omega;\R^n)/\Rcal} \; \max_{\hat \sigma \in \Lrm^2(\omega;\Msn)}\biggl\{\int_\omega \hat \sigma:e(\hat v) \dd x \\
&\qquad\qquad\qquad\qquad\qquad - \int_\omega j^*(\hat \sigma)  \dd x -\int_{\partial\Omega} f\cdot \hat v\dd\HH^{n-1}  \biggr\}\\
&\qquad = \max_{\hat \sigma \in \Lrm^2(\omega;\Msn)} \; \min_{\hat v \in\Hrm^1(\omega;\R^n)/\Rcal} \biggl\{\int_\omega \hat \sigma:e(\hat v) \dd x \\
&\qquad\qquad\qquad\qquad\qquad - \int_\omega j^*(\hat \sigma)  \dd x-\int_{\partial\Omega} f\cdot \hat v\dd\HH^{n-1}  \biggr\} .
\end{align*}
Thus,
\[
\min_{\hat v \in \Hrm^1(\omega;\R^n)/\Rcal} \left\{ \int_\omega j(e(\hat v))  \dd x -  \int_{\partial\Omega}f\cdot \hat v\dd\HH^{n-1}\right\}
=\max_{\hat \sigma \in \Sigma(\omega)} \left\{ - \int_\omega j^*(\hat \sigma)  \dd x\right\},
\]
where $\Sigma(\omega)$ is the set of all \emph{statically admissible stresses}, that is, those $\hat \sigma \in \Lrm^2(\omega;\Msn)$ such that 
\[
\int_\omega \hat \sigma:e(\hat v) \dd x  - \int_{\partial\Omega} f\cdot \hat v\dd\HH^{n-1} =0, \qquad \hat v \in\mathrm{ H}^1(\omega;\R^n)/\Rcal.
\]
Note that since $\omega$ has a Lipschitz boundary, $\hat \sigma \in \Sigma(\omega)$ if and only if (recall~\eqref{eq:forces})
$$  -\dive(\hat \sigma \chi_\omega)=f \quad\text{in $\Dcal'(\R^n;\R^n)$,} $$
that is, in the sense of distributions on $\R^n$. The compliance $\hat{c}(\omega)$ of a shape $\omega \in \A_\e$ can thus equivalently be written as
\[
  \hat{c}(\omega) = \min_{\hat \sigma \in \Lrm^2(\omega;\Msn)} \set{ \int_\omega j^*(\hat \sigma)  \dd x }{ -\dive(\hat \sigma \chi_\omega)=f \text{ in } \Dcal'(\R^n;\R^n) }.
\]
Again, it can be immediately checked via an application of the Direct Method that this minimization problem is well-posed.

In this formulation we can now investigate the vanishing-mass limit $\e \todown 0$. Upon observing that (see~\cite[Theorem~2.3~(i)]{BouchitteButtazzo01})
\[
  \inf \, \setb{ \hat{c}(\omega) }{ \omega \in \A_\e } \sim \frac{1}{\e} \qquad\text{as $\e \todown 0$},
\]
we need to rescale the problem by writing $\hat \sigma=\e^{-1}\sigma$, such that the compliance becomes
\[
  \hat{c}(\omega)=\frac{1}{\e}c_\e(\omega),
\]
where
\[
  c_\e(\omega):=\min_{\sigma \in \Lrm^2(\omega;\Msn)} \setBB{ \int_\Omega j^*(\sigma) \frac{\chi_\omega}{\e}  \dd x }{ -\dive \left(\sigma\frac{\chi_\omega}{\e}\right)=f \text{ in }\Dcal'(\R^n;\R^n) }.
\]
Then, for every probability measure $\mu \in \M^1(\overline\Omega)$, we define the functional
\[
  \mathscr C_\e(\mu):=
\begin{cases}
c_\e(\omega) & \text{if }\mu=\frac{1}{\e}\LL^n\res \omega \text{ for some } \omega \in \A_\e,\\
+\infty & \text{otherwise.}
\end{cases}
\]
Note that if $\mu=\frac{1}{\e}\LL^n \res\omega $ for some $\omega \in \A_\e$ (so, in particular, $\mu$ is a probability measure), then
\[
  \mathscr C_\e(\mu)=\min_{\sigma \in \Lrm^2(\R^n,\mu;\Msn)} \setBB{ \int_\Omega j^*(\sigma)  \dd \mu }{ -\dive \left(\sigma\mu\right)=f \text{ in }\Dcal'(\R^n;\R^n)}.
\]

In this variational framework, the pivotal problem is now to identify the variational limit of $\mathscr C_\e$, that is, to find a ``limit compliance'' $\overline{\mathscr C}$ defined on the space $\M^1(\overline\Omega)$ of probability measures such that
\[
  \inf_{\omega \in \A_\e} \mathscr C_\e \biggl( \frac{\LL^n\res \omega}{\e} \biggr)
  \to \inf_{\mu \in \M^1(\overline\Omega)} \overline{\mathscr C}(\mu)  \qquad\text{as $\e \todown 0$.}
\]
Moreover, we require that approximately minimizing shapes for $\mathscr C_\e$ converge to a limit shape (given by a probability measure) that is minimizing for $\overline{\mathscr C}$ and that all minimizing shapes for $\overline{\mathscr C}$ can be recovered as limits of approximately minimizing shapes of the functionals $\mathscr C_\e$.

It has been conjectured by Bouchitt\'e in~\cite{Bouchitte03} (and implicitly in special cases by several authors before, most notably in Kohn--Strang~\cite{KohnStrang86_1,KohnStrang86_2,KohnStrang86_3} and Allaire--Kohn~\cite{AllaireKohn93}) that the limit compliance $\overline{\mathscr C}$ involves a new integrand in place of $j^*$, even though $j^*$ is convex. Instead, Bouchitt\'e proposed
\begin{equation}\label{eq:minsigma}
\overline{\mathscr C}(\mu) = \min_{\sigma \in \Lrm^2(\R^n,\mu;\Msn)} \setBB{ \int_{\R^n} \bar j^*(\sigma)  \dd \mu }{ -\dive \left(\sigma\mu\right)=f \text{ in }\Dcal'(\R^n ;\R^n) },
\end{equation}
where the \emph{infinitesimal-mass integrand} $\bar j^*$ is defined as the convex conjugate to
\[
  \bar j(\xi) := \sup_{\substack{\tau \in \Msn\\ \det \tau = 0}} \bigl\{ \xi:\tau - j^*(\tau) \bigr\},  \qquad \xi \in \Msn.
\]
The decisive feature here is that in the definition of $\bar j(\xi)$ we only take the supremum over \emph{singular} matrices $\tau$. This should be understood as only considering ``lower-dimensional'' stress tensors, corresponding to $\mu$ being a limit of lower-dimensional structures, which is also to be expected since the (nearly) optimal shapes employed by engineers are usually a combination of sheets and rods, i.e., lower-dimensional structures. We also note that the singular symmetric matrices form the \emph{wave cone} for the divergence operator (first defined in~\cite{Murat78,Murat79,Tartar79}), see Section~\ref{sc:convex}. This fact will become important later on.

More specifically, one is interested in the isotropic and homogeneous situation, where
\[
  j(\xi) := \frac12 |\xi|^2.
\]
In this case, one can compute (see~\cite{AllaireKohn93} and Lemma \ref{lem:jbar} below) that
\[
  \bar j(\xi)=\frac12 (|\xi|^2-\xi_1^2),
\]
where for $\xi \in \Msn$ we let $\xi_1,\ldots,\xi_n$ be the eigenvalues of $\xi$ ordered as singular values, i.e., $|\xi_1|\leq \cdots \leq |\xi_n|$. Another computation shows
\begin{equation} \label{eq:jbs_elast}
  \begin{aligned}
  n=2 &: \quad \bar j^*(\tau)=\frac12 (|\tau_1| + |\tau_2|)^2,
\\
  n=3 &: \quad \bar j^*(\tau)=
\begin{cases}
\frac12 [(|\tau_1|+|\tau_2|)^2+\tau_3^2] & \text{if }|\tau_1|+|\tau_2| \leq |\tau_3|,\\
\frac14 (|\tau_1|+|\tau_2|+|\tau_3|)^2 & \text{if }|\tau_1|+|\tau_2| > |\tau_3|.
\end{cases}
  \end{aligned}
\end{equation}
For additional motivation of this integrand besides the reasoning given for the abstract $\bar j^*$, we refer to Appendix~\ref{ax:ex}, where we will exhibit the approximately optimal shape of an elastic material that leads to this integrand in the vanishing-mass limit.

The above question is also intimately connected to the task of rigorously justifying the Michell truss theory. We make this relationship explicit in Section~\ref{sc:Michell} after we have stated our main result.

There are two primary mathematical issues that make the analysis of Bouchitt\'{e}'s conjecture challenging, namely the hard mass constraint for the admissible shapes, cf.~\eqref{eq:Ae}, and the difficulty in understanding asymptotic concentrations in PDE-constrained sequences of measures. We will comment on both of these issues in turn.

First, we recall that it is common in the theory of shape optimization (and indeed in many related constrained relaxation problems as well) to replace~\eqref{eq:Ae} by a soft constraint involving a Lagrange multiplier, that is, a penalization term $\frac{\kappa}{\e} \int_\Omega \chi_\omega  \dd x$ ($\kappa > 0$) is added to the functional. While on a heuristic level this is a classical approach, see, e.g.,~\cite{Allaire02book,AllaireKohn93} for a pointwise analysis, a full justification in the framework of $\Gamma$-convergence (using a soft mass constraint) has only recently been claimed, see~\cite{Olbermann17,Olbermann20} (but see Remark~\ref{rem:terrible}). With a soft mass constraint the question turns out to be considerably more tractable, since one can now essentially minimize with respect to $\chi=\chi_\omega$ \emph{pointwise} and then compute a suitable relaxation. However, it is not clear how the formulation with the soft constraint is related rigorously to the formulation with the hard constraint. Moreover, the constant $\kappa > 0$ that controls the trade-off between the mass constraint and the elastic energy is not an intrinsic quantity of the original formulation with the hard constraint and the result depends on it. The added term $\kappa \chi/\e$ gives one uniform coercivity (of a linear, not quadratic, nature) to work with, whereas in our case, we get coercivity only on \emph{moving} domains. Some of these issues are also discussed in~\cite{BabadjianIurlanoRindler19?}, where a relaxation analysis in the framework of damage mechanics in terms of the displacement is carried out.

Second, the asymptotic concentrations allowed under the PDE constraint $-\dive (\sigma \mu) = f$ are not currently known explicitly and only partial results exist, see, e.g.,~\cite{BabadjianIurlanoRindler19?,KristensenRindler10,DePhilippisRindler17,ArroyoRabasa19?,KristensenRaita19?} and the references contained therein. It follows from the recent results in~\cite{DePhilippisRindler16,ArroyoRabasaDePhilippisHirschRindler19} that $\sigma \mu$ must be absolutely continuous with respect to the $1$-dimensional Hausdorff measure and $\det \sigma(x) = 0$ for $\mu^s$-almost every $x$, where $\mu^s$ is the singular part of $\mu$. Indeed, if we also knew that $\det \sigma(x) = 0$ for $\LL^n$-almost every $x$ in the support of $\mu$, then we could just replace $j^*$ by $\bar j^*$ since $j^*$ and $\bar j^*$ agree on singular matrices (see Lemma~\ref{lem:jbar}). However, one cannot expect that $\sigma$ is singular $\LL^n$-almost everywhere. We also refer to~\cite[Example~5.1 and Remark~5.4]{BouchitteButtazzo01} for an example of a singular measure that is optimal for the problem with $\bar j^*$, but not for that with $j^*$. Also, as~\cite[Example~2.2]{BouchitteGangboSeppecher08} shows, it could happen that an optimal measure $\mu$ with a nonzero Lebesgue-absolutely continuous part is obtained as a limit of diffuse concentrations (singular measures converging weakly* to an absolutely continuous one). No general compensated compactness theory that could be used to analyze these situations exists at present.

\subsection{Main results}

Let $\Omega \subset \R^n$, $n \in \{2,3\}$, be a bounded Lipschitz domain, i.e., a bounded, open, connected set with Lipschitz boundary. The set $\A_\e$ of admissible shapes is defined in~\eqref{eq:Ae}. We further set
\begin{align*}
  X(\Omega) &:=\setb{(\sigma,\mu) }{ \mu \in \M^1(\overline\Omega) \text{ and } \sigma \in \Lrm^2(\R^n,\mu;\Msn) }, \\
  X_\e(\Omega) &:=\setBB{(\sigma,\mu) \in X(\Omega) }{ \mu=\frac{\LL^n\res \omega}{\e} \text{ for some }\omega \in \A_\e }.
\end{align*}

In this work, we only work with the (dual) integrand $j^*\colon \Msn \to [0,\infty)$,
$$j^*(\tau) := \frac12 |\tau|^2, \qquad \tau \in \Msn,$$
and $\bar j^*$ given in~\eqref{eq:jbs_elast}. The energy functionals $\mathscr E_\e \colon X(\Omega) \to [0,\infty]$ and $\overline{\mathscr E} \colon X(\Omega) \to [0,\infty)$ are defined, for $(\sigma,\mu) \in X(\Omega)$, by
\[
  \mathscr E_\e(\sigma,\mu) := \begin{cases}
    \displaystyle \int_{\R^n} \frac12 |\sigma|^2 \dd \mu   &\text{if } (\sigma,\mu) \in X_\e(\Omega),\\
    +\infty  &\text{otherwise,}
   \end{cases}
\]
and
\[
  \overline{\mathscr E}(\sigma,\mu) := \int_{\R^n} \bar j^*(\sigma) \dd \mu.\quad \label{eq:energies2}
\]
Let $f \in \Lrm^2(\partial\Omega;\R^n)$ be a given external boundary loading satisfying 
\[
  \int_{\partial\Omega} f\cdot r \dd\HH^{n-1}=0, \qquad r \in \Rcal.
\]
In the sequel, we will identify $f$ with an element of $\Hrm^{-1}(\R^n;\R^n) \cap \M(\R^n;\R^n)$. We introduce the compliances $\mathscr C_\e, \overline{\mathscr C} \colon \M^1(\overline\Omega) \to [0,\infty]$ corresponding to the energies $\mathscr E_\e, \overline{\mathscr E}$, respectively, as
\begin{equation}\label{eq:Ceps}
\mathscr C_\e(\mu) :=  \inf \setB{ \mathscr E_\e(\sigma,\mu) }{ \sigma \in \Lrm^2(\R^n,\mu;\Msn),\; -\dive (\sigma\mu)=f }
\end{equation}
and
\begin{equation}\label{eq:barC}
\overline{\mathscr C}(\mu) :=  \inf \setB{ \overline{\mathscr E}(\sigma,\mu) }{ \sigma \in \Lrm^2(\R^n,\mu;\Msn),\; -\dive (\sigma\mu)=f }
\end{equation}
for $\mu \in \M^1(\overline\Omega)$.

The main result of this work concerns the convergence of minimum values and almost-minimizers of the $\e$-compliance $\mathscr C_\e$ to, respectively, the minimum value and a minimizer of the limit compliance $\overline{\mathscr C}$.

\begin{thm} \label{thm:conv-min}
Assume that $\Omega$ is a bounded $\Crm^2$-domain, i.e., $\Omega$ is open, bounded, connected, and the boundary $\partial \Omega$ is of class $\Crm^2$.

\begin{enumerate}[(i)]
\item {\bf Convergence of almost-minimizers.} For $\e>0$ let $\omega_\e \in \A_\e$ be such that 
\[
  \mathscr C_\e \biggl( \frac{\LL^n\res \omega_\e}{\e} \biggr) \leq \inf_{\omega\in\A_\e} \mathscr C_\e \biggl( \frac{\LL^n\res \omega}{\e} \biggr)  + \alpha_\e,  \qquad\text{where}\qquad \alpha_\e \todown 0.
\]
Then, there exists a sequence $\{\e_k\}_{k \in \N}$ with $\e_k \todown 0$ such that the probability measures
\[
  \frac{1}{\e_k}\LL^n\res\omega_{\e_k},  \qquad k \in \N,
\]
converge weakly* in $\M(\R^n)$ to a solution $\bar\mu \in \M^1(\overline\Omega)$ of
\begin{equation}\label{eq:mass-opt}
\min_{\M^1(\overline\Omega)} \overline{\mathscr C},
\end{equation} 
and
\[
  \lim_{k \to \infty} \mathscr C_{\e_k} \biggl( \frac{\LL^n\res \omega_{\e_k}}{\e_k} \biggr)
  = \overline{\mathscr C}(\bar \mu)=\min_{\M^1(\overline\Omega)} \overline{\mathscr C}.
\]

\item {\bf  Convergence of minimum values.}  Let $\mu^* \in \M^1(\overline\Omega)$ be a solution of~\eqref{eq:mass-opt}. Then, for $\e>0$, there exists $\omega^*_\e \in \A_\e$ such that 
$$\frac{1}{\e}\LL^n\res\omega^*_\e \wsto \mu^*\quad\text{ in }\M(\R^n)$$
as $\e\to 0$, and
$$\min_{\M^1(\overline\Omega)} \overline{\mathscr C}= \overline{\mathscr C}(\mu^*)=\lim_{\e \to 0}\mathscr C_\e\left(\frac{1}{\e}\LL^n\res\omega^*_\e\right)= \lim_{\e \to 0} \inf_{\omega \in \A_\e} \mathscr C_\e \biggl( \frac{\LL^n\res \omega}{\e} \biggr).$$
\end{enumerate}
\end{thm}

Usually, the convergence of minimum values and almost-minimizers arises as a consequence of a more general $\Gamma$-convergence analysis~\cite{DalMaso93book}. Indeed, it would be interesting to know if also $\Glim_{\e \todown 0} \mathscr C_\e = \overline{\mathscr C}$ with respect to weak* convergence of measures. In fact, the $\Gamma$-lower bound follows directly from our Proposition~\ref{prop:liminf}. Hence, the question is whether one can find a recovery sequence that admits the mass constraint and loading \emph{exactly}. This seems to be a deeper issue than might appear at first sight since one cannot modify the loading from $f$ to $f'$, say, by solving a PDE of the form $-\dive e(v)=f'-f$ in $\omega$. Unfortunately, the corrector $e(v)$ is not necessarily ``small'' since the Korn--Poincar\'e constant of $\omega$ cannot be assumed to be uniformly bounded (cf.~\cite{LewickaMuller11,LewickaMuller16} and also~\cite{Nitsche81,Horgan95,Temam84book}). In fact, it seems likely that for \emph{any} recovery construction of shapes $\omega_\e$ the Korn--Poincar\'e constant tends to $+\infty$ in the limit because of the requirement that the mass vanishes in the limit. Thus, one can perhaps not expect a ``stability'' with respect to small modifications in $f$ (with respect to some suitable norm). Let us also point out that some ideas about a general upper bound have been given in~\cite[Section 3]{Bouchitte03} using the primal formulation of the compliance in suitable Sobolev spaces with respect to a measure (using techniques from~\cite{BouchitteFragalaSeppecher11}).

Our proof of Theorem~\ref{thm:conv-min} instead proceeds via a different route. For the lower bound, proved in Section~\ref{sc:lower}, we modify the integrand by subtracting a suitable (symmetric) $\dive$-quasiconvex quadratic form, which has a sign in the limit by classical compensated compactness theory (see Section~\ref{sc:CC} for this construction). Pointwise optimizing over a whole family of such (symmetric) $\dive$-quasiconvex quadratic forms, we can then pass to the limit integrand $\bar j^*$ in place of $j^*$ in the vanishing mass limit $\e \todown 0$.

For the convergence of approximate minimizers to a minimizer of the limit problem, we first investigate some properties of the Kohn--Strang functional in Section~\ref{sc:relax}. In particular, we give a proof of (the upper part of) a relaxation result for the formulation with a soft-divergence constraint, Proposition~\ref{prop:Olbermann}, which was first claimed in~\cite{Olbermann17,Olbermann20} (but see Remark~\ref{rem:terrible}). Our proof rests on a suitable smoothing procedure for measures that does not increase the support (adapted from~\cite{ErnGuermond}, see Section~\ref{s:mollification}), for which we require $\Omega$ to be a $\Crm^2$ domain, and some pointwise arguments.

In the final Section~\ref{sc:upper} we then prove Theorem~\ref{thm:conv-min}. For this, we exploit that minimizers of the limit problem have special properties by the results in~\cite{BouchitteButtazzo01}. In particular, the Lagrange multiplier $\kappa = \rho^\circ(\sigma)$ is a \emph{constant} function, which allows one to invoke Proposition~\ref{prop:Olbermann} to obtain a recovery sequence for a limit minimizer.

\subsection{Relation to Michell truss theory} \label{sc:Michell}

Michell trusses, introduced in 1904 in the seminal work~\cite{Michell04}, are a particular class of structures that can be used to approximate optimal shapes. We refer to~\cite{Allaire02book,BouchitteGangboSeppecher08,Olbermann17,Olbermann20} for an overview and a modern mathematical formulation. In this section, we explain how our results can be understood in this context.

In the literature it is often stated, usually heuristically, that the Michell truss theory should be seen as the vanishing-mass limit of the shape optimization problem for elastic materials. One can typically find a formulation that roughly corresponds to the following: Given a loading $f$, one is tasked with the minimization
\begin{equation}\label{eq:min-michell}
\min_{\lambda \in \M(\R^n;\Msn)} \mathscr F(\lambda),
\end{equation}
where the \emph{Michell functional} and the \emph{Michell integrand} are given by
\[
   \mathscr F(\lambda) := 
   \begin{cases}
    \displaystyle \int_{\overline \Omega} \rho^\circ\left(\frac{\di\lambda}{\di|\lambda|}\right)\dd |\lambda|   &\text{if }\supp(\lambda) \subset \overline \Omega, \; -\dive \lambda=f \text{ in }\Dcal'(\R^n;\R^n), \\
    +\infty  &\text{otherwise,}
   \end{cases} 
\]
and
\begin{equation} \label{eq:WMi}
  \begin{aligned}
  n=2 &: \quad \rho^\circ(\tau) := |\tau_1| + |\tau_2|, \\
  n=3 &: \quad \rho^\circ(\tau) :=
\begin{cases}
 \sqrt{(|\tau_1|+|\tau_2|)^2+\tau_3^2} & \text{if }|\tau_1|+|\tau_2| \leq |\tau_3|,\\
\frac{1}{\sqrt{2}} (|\tau_1|+|\tau_2|+|\tau_3|) & \text{if }|\tau_1|+|\tau_2| > |\tau_3|,
\end{cases}
  \end{aligned}
\end{equation}
respectively. Note that from~\eqref{eq:jbs_elast} and~\eqref{eq:WMi} we see that the functions $\rho^\circ$ and $\bar j^*$ are related by
\[
  \bar j^*=\frac12 (\rho^\circ)^2,
\]
also see Section~\ref{sc:convex} for more on the underlying convex analysis involved here. Historically, $\rho^\circ$ is only equal to the predicted Michell integrand in two dimensions, whereas the picture is more complicated in three dimensions by the inclusion of plates besides bars.

The above problem has been rigorously derived in~\cite{Olbermann17,Olbermann20} by considering the Kohn--Strang-type functionals
\[
   \mathscr F_\e(\lambda) := \int_{\Omega} h_\e(\lambda) \dd x,
\]
with
\[
h_\e(\tau) :=
\begin{cases}
\frac{\e}{2} |\tau|^2 + \frac{1}{2\e} &\text{if }\tau \neq 0,\\
0 & \text{if } \tau=0,
\end{cases}
\]
where now $1/(2\e)$ takes the role of a Lagrange multiplier for the vanishing mass constraint acting on the shape $\omega=\{\lambda\neq 0\}$ (as $\e \todown 0$). Then, one takes the (symmetric) $\dive$-quasiconvex envelope (see Section~\ref{sc:CC}) of this integrand and passes to a pointwise limit using the Hashin--Shtrikman bounds from~\cite{AllaireKohn93}, see, e.g.,~\cite[Section 4.2.3]{Allaire02book} for details. 

Note that there might be many decompositions $\lambda=\sigma\mu$ for $\mu \in \M^1(\overline\Omega)$  and $\sigma \in \Lrm^1(\R^n,\mu;\Msn)$ with the property
$$\int_{\overline \Omega} \rho^\circ\left(\frac{\di\lambda}{\di|\lambda|}\right)\dd |\lambda|  =\int_{\overline \Omega} \rho^\circ(\sigma)\dd\mu.$$
However, it is proved in~\cite[Theorem~2.3~(iii)]{BouchitteButtazzo01} that any minimizer $\mu$ for $\overline{\mathscr C}$ together with its unique associated stress $\sigma$ (solving~\eqref{eq:minsigma} for $\mu$) defines a minimizer $\lambda:=\sigma\mu$ of $\mathscr F$. Conversely, every minimizer $\lambda$ of $\mathscr F$ can be decomposed as $\lambda = \sigma\mu$ for some optimal $\sigma$ and $\mu$. In addition, the minimal values are related via\begin{equation}\label{eq:entropy-cond}
\int_{\overline \Omega} \bar j^*(\sigma)\dd\mu= \int_{\overline \Omega} \frac12 \rho^\circ(\sigma)^2\dd\mu=\overline{\mathscr C}(\mu)=\frac{\mathscr F(\lambda)^2}{2}=\frac12\left( \int_{\overline \Omega} \rho^\circ(\sigma)\dd\mu\right)^2.
\end{equation}
Thus, even though $\rho^\circ$ and our $\bar j^*= \frac12(\rho^\circ)^2$ have the same \emph{pointwise} minimizers, the different scalings in $\sigma$ cause a discrepancy in the minimizers of the corresponding compliance functionals. This point is further illustrated in Appendix~\ref{ax:ex}, which contains constructions that show how finer and finer structures indeed lead the integrand $\frac12 |\sigma|^2$ to behave like the infinitesimal-mass integrand $\bar j^*$ in the limit.

Let us illustrate this point through the following example.

\begin{example}
Consider the square $\Omega := (0,1)^2$ and the parametrized set of probability measures
\[
  \mu_\delta := \delta \HH^1 \restrict ((0,1) \times \{0\}) + (1-\delta) \HH^1 \restrict (\{0\} \times (0,1)),  \qquad \delta \in [0,1].
\]
Denoting by $\{\ee_1,\ee_2\}$ the canonical basis of $\R^2$, we set
\[
f := \ee_1 \delta_{\ee_1} + \ee_2 \delta_{\ee_2} -(\ee_1+\ee_2)\delta_0,
\]
or a suitable mollification thereof. Then, we can easily see that
\[
  \sigma_\delta := \begin{cases}
    \frac{1}{\delta} \ee_1 \otimes \ee_1 &\text{on $(0,1) \times \{0\}$,}\\
    \frac{1}{1-\delta} \ee_2 \otimes \ee_2 &\text{on $\{0\} \times (0,1)$}
  \end{cases}
\]
is the optimal stress, i.e., the unique $\sigma_\delta$ satisfying $-\dive (\sigma_\delta \mu_\delta) = f$, which here decomposes into two (distributional) ODEs on the coordinate lines and can thus be solved explicitly. Consequently, setting 
$$\lambda_\delta:=\sigma_\delta \mu_\delta=\ee_1 \otimes \ee_1 \, \HH^1 \restrict ((0,1) \times \{0\}) + \ee_2 \otimes \ee_2 \, \HH^1 \restrict (\{0\} \times (0,1)),$$
which is actually independent of $\delta$, we have
\[
\mathscr F(\lambda_\delta) = 2  \qquad\text{for all $\delta \in [0,1]$.}
\]
This means that any distribution of mass between the horizontal and vertical bars gives a Michell-optimal shape in the sense above. This is, however, clearly not physical: For two orthogonal (non-interacting) elastic bars, one expects to distribute the mass equally between the two bars to minimize the total compliance (note in particular that the endpoints $\delta = 0,1$ should have \emph{infinite} compliance since they cannot carry the loading).

Our main result, Theorem~\ref{thm:conv-min}, involving the infinitesimal-mass integrand $\bar j^*$ instead, gives the limit compliance
\[
  \overline{\mathscr C}(\mu_\delta) = \int_\Omega \frac12 \bigl(|(\sigma_\delta)_1| + |(\sigma_\delta)_2| \bigr)^2 \dd \mu_\delta
  = \frac{1}{2\delta} + \frac{1}{2(1-\delta)},
\]
which attains the minimum value $2$ at the (unique) minimizer $\mu_{1/2}$, as expected. Hence, $\overline{\mathscr C}$, and not $\mathscr F$, selects the physically correct minimizer $\mu_{1/2}$. 

Note that an application of~\cite[Theorem~2.3~(iii)]{BouchitteButtazzo01} yields that the measure
$$\mu:=\frac12 \rho^\circ(\lambda_\delta)$$
coincides with $\mu_{1/2}$ (which here is also obvious), hence it is optimal for $\overline{\mathscr C}$. As a consequence, the measure $\mu_{1/2}$ is precisely the one selected by the condition~\eqref{eq:entropy-cond} which here reads as
$$ \int_{\overline \Omega} \rho^\circ(\sigma)^2 \dd\mu= \left( \int_{\overline \Omega} \rho^\circ(\sigma) \dd\mu\right)^2.$$
This identity acts as a kind of entropy selection criterion for the physically relevant decomposition $\lambda=\sigma\mu$ of the solutions $\lambda$ of~\eqref{eq:min-michell}.
\end{example}

The previous discussion shows that any optimal pair $(\sigma,\mu)$ of our limit compliance problem~\eqref{eq:barC}--\eqref{eq:mass-opt} gives rise to a minimizer $\lambda:=\sigma\mu$ of the Michell functional. We therefore recover the same structure as in~\cite{Olbermann17,Olbermann20}. However, in the two-dimensional case $n=2$ studied in~\cite{Olbermann17}, the problem is reformulated in terms of the Airy function $\phi$, which is related to the measure $\lambda$ by the relation $D^2 \phi=\cof(\lambda)$. In this formulation, a boundary integral penalizes the cost of violating the limit boundary value of the normal derivative $\nabla \phi\cdot \nu$. We refer to~\cite[Lemma 4.5]{Olbermann17} for a discussion on the relation between the boundary conditions for $\phi$ and $\lambda$. In the case $n=2$, the formulation with the Airy function entails that singularities (in the interior) essentially come from discontinuities in the gradient of the Airy function $\phi$, so the structure of singularities can be understood more finely via the theory of maps with second derivative a measure. It is unclear at present whether similar constraints can be proved to hold in dimension $n\geq 3$.

To conclude the introduction, let us mention several interesting and challenging open questions:
\begin{itemize}
\item \textit{Understanding singularities of solutions to the minimization problem~\eqref{eq:mass-opt}:} From~\cite{ArroyoRabasaDePhilippisHirschRindler19} we know that general measures $\lambda=\sigma\mu$ satisfying $-\dive \lambda \in \M(\R^n;\R^n)$ must have a singular support carried by a set of codimension at most $n-1$. We expect a more precise statement for solutions to~\eqref{eq:mass-opt} but this would require a fine analysis of singularities in optimal configurations for our limit compliance problem.
\item \textit{Providing a full $\Gamma$-convergence result for general integrands $j$ (instead of just the quadratic one $j=\frac12|\cdot|^2$) and in arbitrary dimensions:} In~\cite{Bouchitte03} a sketch is given for a possible proof of the upper bound, using the primal problem in terms of the displacement instead of the stress. However, a uniform Korn-type inequality on sequences of domains, which would allow one to implement this strategy, seems to be missing. On the other hand, an adaptation of our algebraic and compensated compactness arguments might lead to a general lower bound inequality. 
\item \textit{Understanding diffuse concentrations:} In~\cite{Bouchitte03} a possible approach to the vanishing mass conjecture in terms of (generalized) Young measure generated by the sequence $\{\sigma_\e\mu_\e\}_{\e>0}$ is outlined. However, despite recent advances in the theory of Young measures for $\Acal$-free sequences~\cite{ArroyoRabasa19?,KristensenRaita19?}, diffuse concentrations (singular measures converging weakly* to an absolutely continuous one) remain only superficially understood. It is, however, precisely these diffuse concentrations that lie at the core of the general conjecture (as is already observed in~\cite{Bouchitte03}).
\end{itemize}

\subsection{Acknowledgements}

This project has received funding from the European Research Council (ERC) under the European Union's Horizon 2020 research and innovation programme, grant agreement No 757254 (SINGULARITY). The research of J.-F.~B.\ was supported by a public grant as part of the Investissement d'avenir project, reference ANR-11-LABX-0056-LMH, LabEx LMH. F.I.\ 
acknowledges partial support from the Mission pour les Initiatives Transverses et Interdisciplinaires (MITI) of the CNRS through the project ``CalVaMec''. The authors would like to thank Giovanni Alberti, Guido De~Philippis and Gilles Francfort for helpful discussions related to this paper. The authors are grateful to the anonymous referees for their valuable comments.

\section{Preliminaries}

\subsection{Notation}

We denote by $\mathbb M^{n \times n}$ the space of $n \times n$ matrices and by $\Msn$ and $\mathbb M^{n \times n}_{\rm skew}$ the subspaces of symmetric and skew-symmetric matrices, respectively. If $\xi, \zeta \in \mathbb M^{n \times n}$, then $\xi:\zeta:=\tr(\xi^T \zeta)$ is the Frobenius scalar product between $\xi$ and $\zeta$, and $|\xi|:=\sqrt{\xi:\xi}$ is the corresponding Frobenius norm. We recall that for \emph{symmetric} matrices $\xi \in \Msn$, one has the expression $|\xi|^2 = \xi_1^2 + \cdots + \xi_n^2$ with $\xi_1,\ldots,\xi_n$ the eigenvalues of $\xi$, which in the following we always order as singular values, i.e., $|\xi_1|\leq \cdots \leq |\xi_n|$. The tensor product and symmetric tensor product of $a,b \in \R^n$ are $a \otimes b := ab^T$ and $a \odot b := \frac12 (a \otimes b + b \otimes a)$, respectively.

The Lebesgue measure in $\R^n$ is written $\Lcal^n$ and $\HH^k$ stands for the $k$-dimensional Hausdorff (outer) measure. We denote by $\M(\R^n;\Msn)$ (respectively $\M(\R^n)$) the space of $\Msn$-valued (respectively real-valued) globally bounded Radon measures. The sets $\M^+(\R^n)$ and $\M^1(\R^n)$ contain all nonnegative bounded Radon measures and probability measures in $\R^n$. Finally, if $K \subset \R^n$ is compact, we denote by $\M(K;\Msn)$ the space of $\Msn$-valued bounded Radon measures in $\R^n$ with support contained in $K$. Corresponding definitions are used for $\M(K)$, $\M^+(K)$ and $\M^1(K)$. In the sequel, if $f\in \Lrm^1(\R^n)$, we will often identify $f$ with the absolutely continuous measure $f \LL^n$.

\subsection{Compensated compactness} \label{sc:CC}

In the theory of compensated compactness, the {\it wave cone} associated to the row-wise divergence ``$\dive$'' is defined as $\Lambda_{\dive} := \bigcup_{\lambda \neq 0} \ker \Abb(\lambda)$, where $\Abb(\lambda)M := (2\pi\ii) M\lambda$ ($M \in \Msn$, $\lambda \in \R^n$) is the (principal) Fourier symbol of $\dive$. It is given by the set of all singular symmetric matrices,
$$\Lambda_{\dive}=\setb{\sigma \in \Msn }{ \det \sigma=0 }.$$

Following~\cite{Tartar79}, a locally bounded Borel-measurable function $f \colon \Msn \to \R$ is said to be {\it $\Lambda_{\dive}$-convex} (respectively {\it $\Lambda_{\dive}$-affine}) if, for all $A \in \Msn$ and $\sigma \in \Lambda_{\dive}$, the function $t \in \R \mapsto f(A+t\sigma)$ is convex (respectively affine). Another relevant notion is that of (symmetric) $\dive$-quasiconvexity~\cite{FonsecaMuller99}, which turns out to be a necessary and sufficient condition for weak lower semicontinuity of integral functionals under a divergence-free constraint (see~\cite{FonsecaMuller99} or~\cite[Lemma~2.5]{ContiMullerOrtiz20}): We call a locally bounded Borel-measurable function $f \colon \Msn \to \R$  \emph{(symmetric) $\dive$-quasiconvex} if for all $\varphi \in \Crm^\infty_{\rm per}((0,1)^n;\Msn)$ such that $\dive \varphi=0$ in $\R^n$, it holds that
$$f\biggl(\int_{(0,1)^n} \varphi(x) \dd x\biggr) \leq \int_{(0,1)^n}f(\varphi(x)) \dd x.$$
Here, $\Crm^\infty_{\rm per}((0,1)^n;\Msn)$ is the space of all $\Msn$-valued functions that are smooth and periodic on the unit torus, i.e.,\ $(0,1)^n$ with opposite edges identified. It is well known that (symmetric) $\dive$-quasiconvexity implies $\Lambda_{\dive}$-convexity (see~\cite[Lemma~2.4]{ContiMullerOrtiz20} and also~\cite[Proposition~3.4]{FonsecaMuller99}) but that the converse implication is false in general (see~\cite{Palombaro10}). However, both (symmetric) $\dive$-quasiconvexity and $\Lambda_{\dive}$-convexity are equivalent for quadratic forms. Indeed, from~\cite[Theorem~11]{Tartar79} (or~\cite[Theorem~8.30]{Rindler18book}) we get that $\Lambda_{\dive}$-convexity implies weak* lower semicontinuity for quadratic forms, which yields the $\dive$-convexity by~\cite[Theorem~3.6]{FonsecaMuller99}.

In particular, for $n=2$, the determinant is $\Lambda_{\dive}$-affine. As a consequence, we can apply the standard theory of compensated compactness (see, e.g.,~\cite[Theorem~8.30]{Rindler18book} or~\cite[Corollary~13]{Tartar79}) which ensures that if $\{\sigma_k\}_{k\in \N}$ is a sequence in $\Lrm^2(\Omega;\Mstwo)$ and $\sigma \in \Lrm^2(\Omega;\Mstwo)$, then
\[
\left\{\begin{aligned}
\sigma_k &\wto \sigma &&\text{in }\Lrm^2(\Omega;\Mstwo),\\
\dive \sigma &\to \dive \sigma &&\text{in }\Hrm^{-1}(\Omega;\R^2)
\end{aligned}\right\}
\quad
\Longrightarrow
\quad 
\det \sigma_k \wsto \det \sigma \text{ in $\M(\Omega)$.}
\]
In fact, for $n=2$, the determinant is the only nonlinear $\Lambda_{\dive}$-affine function. This is a consequence of the fact that, if $n=2$, a matrix is singular if and only if it has rank one or zero. Therefore, the $\Lambda_{\dive}$-affine functions are precisely the rank-one affine functions, i.e., the determinant and affine functions (see~\cite[Example~5.21]{Dacorogna08book}).

In dimension $n=3$, it turns out that there are no non-zero $\Lambda_{\dive}$-affine quadratic forms on symmetric matrices as a consequence of the general result in~\cite{Murat81}. Indeed, a $\Lambda_{\dive}$-affine quadratic form $q$ must vanish on all singular matrices and, in particular, must be rank-one affine. So, it has to take the form $q(\sigma)=B:{\rm cof}(\sigma)$ for some matrix $B \in \mathbb M^{3 \times 3}$ and all $\sigma \in \mathbb M^{3 \times 3}_{\text{sym}}$ (see~\cite[Theorem~5.20]{Dacorogna08book}). Since ${\rm cof}(\sigma)$ is symmetric, we can also require $B$ to be symmetric. Using next that $q$ must vanish on rank-two matrices $\sigma$, it follows that $B_{ii}=0$ for all $1 \leq i \leq 3$ (taking $\sigma=e_i \odot e_i+e_j \odot e_j$ with $i \neq j$), and then that 
$B$ is skew-symmetric (taking $\sigma=e_i \odot e_j+e_i\odot e_k$ for all $i \neq j \neq k$), hence $q\equiv 0$.

In Lemma~\ref{lem:supR} below, we will instead introduce a class of $\Lambda_{\dive}$-convex quadratic forms that can be seen as a generalizations of Tartar's quadratic form $\tau \in \Msn \mapsto (n-1)|\tau|^2-({\rm tr\,}\tau)^2$.

\subsection{Convex analysis} \label{sc:convex}

We define the function
$$\bar j(\xi):=\sup_{\tau \in \Lambda_{\dive}}\left\{\xi : \tau - \frac12|\tau|^2 \right\}, \qquad \xi \in \Msn.$$
Clearly, $\bar j$ is positively $2$-homogeneous.  Following~\cite{Bouchitte03,BouchitteButtazzo01} we introduce the \emph{gauge function} of the convex set $\{\bar j \leq \frac12\}$, i.e., the convex, continuous, and positively one-homogeneous function $\rho \colon \Msn \to [0,\infty)$ defined by
\begin{equation}\label{eqrho}
\rho(\xi):=\inf\setBB{t>0 }{ \bar j\left(\frac{\xi}{t}\right) \leq \frac12 }, \qquad \xi \in \Msn.
\end{equation}
According to the two-homogeneity of $\bar j$, we have that
\[
  \bar j = \frac12 \rho^2.
\]
Let us further introduce the \emph{polar function} $\rho^\circ \colon \Msn\to [0,\infty)$ of $\rho$ defined by
$$\rho^\circ(\tau):=\sup_{\rho(\xi)\leq 1}\xi:\tau=\sup_{\bar j(\xi)\leq \frac12}\xi:\tau, \qquad \tau \in \Msn.$$
According to~\cite[Corollary~15.3.1]{Rockafellar70book}, for the convex conjugate $\bar j^* \colon \Msn \to [0,\infty)$ of $\bar j$, defined by
\[
  \bar j^*(\tau):=\sup_{\xi\in\Msn} \bigl\{\xi:\tau-\bar j(\xi)\bigr\}, \qquad \tau \in \Msn,
\]
we have
\[
  \bar j^*=\frac12 (\rho^\circ)^2.
\]

The following lemma collects some properties of $\bar j$ (as already established in~\cite{Bouchitte03}).
\begin{lem} \label{lem:jbar}
Let $\xi \in \Msn$ and let $\xi_1,\cdots,\xi_n$ be the eigenvalues of $\xi$, ordered as singular values, $|\xi_1|\leq\cdots\leq|\xi_n|$. Then,
\begin{equation} \label{eq:jbar_formula}
  \bar j(\xi)=\frac12 (|\xi|^2-\xi_1^2).
\end{equation}
Furthermore, for all $\tau \in \Lambda_{\dive}$,
\[
  \bar j^*(\tau) = j^*(\tau).
\]
\end{lem}

\begin{proof}
Let $\xi = P D P^T$ with $P \in \mathrm{SO}(n)$ and $D = \diag(\xi_1,\ldots,\xi_n)$. Then,
\begin{align*}
  \bar j(\xi)
  &= \sup_{\tau \in \Lambda_{\dive}}\left\{D : (P^T\tau P) - \frac12|P^T\tau P|^2 \right\} \\
  &= \sup_{\tau' \in \Lambda_{\dive}}\left\{D : \tau' - \frac12|\tau'|^2 \right\} \\
  &= \bar j(D).
\end{align*}
The expression in the brackets is maximized for $\tau' = \diag(0,\xi_2,\ldots,\xi_n)$ (this can be seen in an elementary way by adding a Lagrange multiplier for the constraint $\det(\tau') = 0$). Then,
\[
  \bar j(\xi) = \frac12 \bigl( \xi_2^2 + \cdots + \xi_n^2 \bigr)
  = \frac12 (|\xi|^2-\xi_1^2),
\]
so~\eqref{eq:jbar_formula} follows.

For the second assertion, we assume for $\tau \in \Lambda_{\dive}$ that it is diagonal (by a similar argument as above) and compute
\begin{align*}
  \bar j^*(\tau) &= \sup_{\xi \in \Msn} \left\{\xi : \tau - \bar j(\xi) \right\} \\
  &= \sup_{\xi \in \Lambda_{\dive}}\left\{\xi : \tau - j(\xi) \right\} \\
  &= \sup_{\xi \in \Msn}\left\{\xi : \tau - j(\xi) \right\} \\
  &= j^*(\tau),
\end{align*}
where we used that $\bar j(\xi) = j(\xi)$ for $\xi \in \Lambda_{\dive}$ by~\eqref{eq:jbar_formula}.
\end{proof}

From~\eqref{eq:jbar_formula} one also gets immediately that there is a $c > 0$ such that
\begin{equation} \label{eq:jbar_est}
  c^{-1} |\xi|^2 \leq \bar j(\xi) \leq c|\xi|^2,  \qquad
  c^{-1} |\xi| \leq \rho^\circ(\xi) \leq c|\xi|, \qquad \xi \in \Msn,
\end{equation}
since we may estimate $\bar j(\xi) \geq \frac12 |\xi_n|^2 \geq \frac{1}{2n}|\xi|^2$) and the upper bound is obvious. Also,
\[
  \rho(\xi) \leq 1     \qquad\Longleftrightarrow\qquad
  \bar j(\xi) \leq \frac12  \qquad\Longleftrightarrow\qquad
  \xi_2^2 +\cdots+\xi_n^2 \leq 1.
\]

In the physical cases $n=2$ or $3$, we can write $\bar j^*$ explicitly by combining the previous formula~\eqref{eq:jbar_formula} with the arguments of~\cite[p.~872]{AllaireKohn93}. In this way one obtains the expressions for $\bar j^*$ and $\rho^\circ$ in~\eqref{eq:jbs_elast} and~\eqref{eq:WMi}, respectively.

\begin{rem}\label{rem:qalpha}
In the two-dimensional case, it is easy to approximate $\bar j^*$ given by~\eqref{eq:jbs_elast} by means of quadratic forms. Indeed, for all $\alpha \in [-1,1]$, let us define the nonnegative (hence convex) quadratic form
$$q_\alpha(\tau):=\frac12|\tau|^2+\alpha \det \tau, \qquad \tau \in \Mstwo.$$
For all $\sigma \in \Mstwo$, we have
\begin{equation}\label{eq:maxalpha}
\bar j^*(\tau)=\frac12 (|\tau_1| + |\tau_2|)^2=\max_{\alpha\in \{-1,1\}} q_\alpha(\tau).
\end{equation}
\end{rem}

In the three-dimension case, a similar approximation of $\bar j^*$ is more involved. To this aim, the following algebraic lemma will be employed below to show that a certain class of quadratic forms is included in the set of the $\Lambda_{\dive}$-convex functions.
\begin{lem}\label{diagonalentries}
Let $\xi \in \Msn$ and let $\xi_1,\cdots,\xi_n$ be the eigenvalues of $\xi$, ordered as singular values, $|\xi_1|\leq\cdots\leq|\xi_n|$. Then,
\begin{equation}\label{23entries}
\xi_{22}^2+\cdots+\xi_{nn}^2\leq\xi_2^2+\cdots+\xi_n^2.
\end{equation}
\end{lem}
\begin{proof}
This is a general result of linear algebra which can be found in~\cite{DeOliveira71} and~\cite[Corollary~on~p.~90]{Sing76}. We also give a direct argument for $n=2$ or $3$, which correspond to the physical cases of interest here. 

For $n=2$, it is well known that the spectral radius of a matrix $\xi \in \Mstwo$ can be obtained by maximizing the Rayleigh quotient
$$|\xi_2|^2=\max_{|x|=1}(x^T\xi^T\xi x)=\max_{|x|=1}|\xi x|^2.$$
Thus, taking $x=\ee_2$, the second vector of the canonical basis of $\R^2$, we get that $\xi_{22}^2 \leq \xi_2^2$, which corresponds to~\eqref{23entries} in the case $n=2$.

For $n=3$, we will prove that 
$$\sup_{\substack{|x|=|y|=1\\ x\cdot y=0}} (|\xi x|^2+|\xi y|^2) \leq \xi_2^2+\xi_3^2,$$
then~\eqref{23entries} will follow by taking $x=\ee_2$ and $y=\ee_3$.

Up to a change of variables in the supremum, we can assume that $\xi$ is diagonal, $\xi=\text{diag}(\xi_1,\xi_2,\xi_3)$. Let $f \colon \R^3\times\R^3\to\R$ be defined by
$$f(x,y):=|\xi x|^2+|\xi y|^2.$$
We will show that for all $x,y\in \R^3$, with $|x|=|y|=1$ and $x\cdot y=0$, 
\begin{equation}\label{uppestim}
f(x,y)\leq \xi_2^2+\xi_3^2.
\end{equation}
We set
$$\tilde x:=\left(0,x_2,\sqrt{x_1^2+x_3^2}\right)\qquad \text{and} \qquad \tilde y:=\left(0,-\sqrt{x_1^2+x_3^2},x_2\right)$$
which satisfy $|\tilde x|=|\tilde y|=1$ and $\tilde x\cdot \tilde y=0$. Compute
\begin{align}
f(\tilde x,\tilde y)-f(x,y)&=\xi_1^2(\tilde x_1^2+\tilde y_1^2-x_1^2- y_1^2) \\
&\qquad +\xi_2^2(\tilde x_2^2+\tilde y_2^2-x_2^2- y_2^2) \\
&\qquad +\xi_3^2(\tilde x_3^2+\tilde y_3^2-x_3^2- y_3^2)\notag\\
&= (\xi_2^2-\xi_1^2)(\tilde x_2^2+\tilde y_2^2-x_2^2- y_2^2) \\
&\qquad +(\xi_3^2-\xi_1^2)(\tilde x_3^2 +\tilde y_3^2-x_3^2- y_3^2)\notag\\
&= (\xi_2^2-\xi_1^2)(x_1^2+x_3^2-y_2^2)+(\xi_3^2-\xi_1^2)(x_1^2+x_2^2-y_3^2)\notag\\
&= (\xi_2^2-\xi_1^2)(1-x_2^2-y_2^2)+(\xi_3^2-\xi_1^2)(1-x_3^2-y_3^2),\label{positive}
\end{align}
where in the second and in the fourth equalities we used $|x|=|y|=|\tilde x|=|\tilde y|=1$, and in the third equality we used the definition of $\tilde x$ and $\tilde y$. Since by hypothesis $\xi_2^2-\xi_1^2\geq0$ and  $\xi_3^2-\xi_1^2\geq0$, the expression in~\eqref{positive} is shown to be nonnegative as soon as we have checked that
\begin{equation}\label{componentwise}
x_2^2+y_2^2\leq 1, \qquad x_3^2+y_3^2\leq 1.
\end{equation}
If $x_1=0$ and $y_1 = 0$, the inequality~\eqref{componentwise} follows trivially since $y=\pm(0,-x_3,x_2)$. Then, without loss of generality, we may assume $x_1\neq0$. Writing 
$$y_1=-\frac{x_2y_2+x_3y_3}{x_1}$$
and inserting this expression into $|y|^2=1$, we find
\begin{equation}\label{y2y3}
(x_1^2+x_2^2)y_2^2+(2x_2x_3y_3)y_2+(x_1^2y_3^2+x_3^2y_3^2-x_1^2)=0.
\end{equation}
Since the second component $y_2$ of the given $y\in\R^3$ is a solution of the previous equation, the discriminant $\Delta$ of this equation (with respect to $y_2$) must be nonnegative, that is, recalling that $|x|=1$,
\[
  \Delta=4 x_1^2 (1-x_3^2-y_3^2)\geq0.
\]
This gives the second inequality of~\eqref{componentwise}. Writing~\eqref{y2y3} instead as a quadratic equation for $y_3$, we obtain by an analogous argument that also $1-x_2^2-y_2^2 \geq 0$, that is, the first inequality of~\eqref{componentwise}. Finally, we observe
\[
  f(\tilde x,\tilde y)=\xi_2^2+\xi_3^2.
\]
Then,~\eqref{positive} implies~\eqref{uppestim} and the proof of the lemma is finished.
\end{proof}

The next lemma provides a family of $\Lambda_{\dive}$-convex quadratic forms, which includes Tartar's $\Lambda_{\dive}$-convex function $\tau \in \Msn \mapsto (n-1)|\tau|^2-({\rm tr\,}\tau)^2$.

\begin{lem}\label{lem:supR}
For $\xi \in \Msn$ such that $\rho(\xi) \leq 1$ (which is equivalent to $\xi_2^2 +\cdots+\xi_n^2 \leq 1$) we define the quadratic form $Q_\xi \colon \Msn \to \R$ via
$$Q_\xi(\tau):=\frac12|\tau|^2 - \frac12 (\xi:\tau)^2, \qquad \tau \in \Msn.$$
Then $Q_\xi$ is $\Lambda_{\rm div}$-convex and 
$$\bar j^*(\tau)
= \frac12 \sup_{\rho(\xi) \leq 1} (\xi:\tau)^2
=\sup_{\rho(\xi) \leq 1} \left\{\frac12|\tau|^2 - Q_\xi(\tau)\right\}, \qquad \tau \in \Msn.$$
\end{lem}

\begin{proof}
Since for all $\tau \in \Msn$ we have $\bar j^*(\tau)=\frac12 \rho^\circ(\tau)^2$, we get
\[
 \bar j^*(\tau)
  = \frac12 \sup_{\rho(\xi) \leq 1} (\xi:\tau)^2
 = \sup_{\rho(\xi) \leq 1} \left\{\frac12|\tau|^2 - Q_\xi(\tau)\right\}, \qquad \tau \in \Msn.
\]
Here, we used $\rho(-\xi) = \rho(\xi)$ to exchange the square and the supremum.

Given $\xi \in \Msn$ such that $\rho(\xi)\leq1$, it remains to show that $Q_\xi$ is $\Lambda_{\rm div}$-convex. Let $\tau \in \Lambda_{\rm div}$, i.e., $\tau \in \Msn$ and ${\rm det }\,\tau=0$. By spectral decomposition, we can write $\tau=PDP^T$ for some $P \in \mathrm{SO}(n)$ and $D={\rm diag}(\tau_1,\ldots,\tau_n)$, where $\tau_1,\ldots,\tau_n$ are the eigenvalues of $\tau$, ordered as singular values, $|\tau_1|\leq\cdots\leq |\tau_n|$. Since $\tau$ is singular, $\tau_1=0$. Setting $A:=P^T \xi P$,
\begin{align*}
 Q_\xi(\tau) &= \frac12 |D|^2-\frac12 (A:D)^2 \\
 &=\frac12 (|\tau_2|^2+\cdots+|\tau_n|^2) - \frac12 (A_{22} \tau_2 +\cdots+ A_{nn}\tau_n)^2 \\
 &= \frac12 (Z \hat{\tau})\cdot\hat{\tau},
\end{align*}
where $\hat{\tau} := (\tau_2,\tau_3,\ldots,\tau_n)$ and
\[
  Z := \begin{pmatrix}
    1 - A_{22}^2   &  -A_{22}A_{33}  &  \cdots  &- A_{22}A_{nn} \\
    -A_{22}A_{33}  &  1 - A_{33}^2   &  \cdots  &  -A_{33}A_{nn} \\
    \vdots         &         \vdots        &  \ddots  & \vdots \\ 
    -A_{22}A_{nn}               &  -A_{33}A_{nn}              &    \hdots      & 1 - A_{nn}^2
  \end{pmatrix}
  = \Id - a \otimes a
\]
for
\[
a = (A_{22}, A_{33}, \ldots, A_{nn}).
\]
Then, $Q_\xi$ is nonnegative if and only if all eigenvalues of $Z$ are nonnegative. The scalar $\lambda \in \R$ is an eigenvalue of $Z$ if and only if
\[
  0 = \det( \lambda \Id - Z ) = (-1)^{n-1} \det( (1-\lambda)\Id - a \otimes a ),
\]
that is, $1-\lambda$ is an eigenvalue of $a \otimes a$. Since $a\otimes a$ has rank one, only one eigenvalue of $a \otimes a$ may be non-zero, so the nonnegativity of all $\lambda$'s is equivalent to $\tr (a \otimes a) \leq 1$.

By Lemma~\ref{diagonalentries}, using that $\xi$ and $A$ have the same eigenvalues and that $\rho(\xi) \leq 1$, we obtain that indeed
\[
  \tr (a \otimes a) = A_{22}^2+\cdots+A_{nn}^2 \leq  \xi_2^2+\cdots+\xi_n^2\leq1.
\]
We thus deduce that the quadratic form $Q_\xi$ is nonnegative on the wave cone $\Lambda_{\rm div}$, hence that $Q_\xi$ is $\Lambda_{\rm div}$-convex.
\end{proof}

\subsection{The dual of $\Hrm^1$ modulo rigid deformations}

The following result is in the spirit of the standard characterization of the dual of the Sobolev space $\Hrm^1(\Omega;\R^n)$, see, e.g.,~\cite[Theorem~4.3.2]{Ziemer89book}.

\begin{prop}\label{prop:W1pR*}
Let $\Omega\subset\R^n$ be a bounded Lipschitz domain. Let $g \in \Hrm^{-1}(\R^n;\R^n)$ with $\supp(g) \subset \overline\Omega$ and $\langle g,r\rangle=0$ for all $r \in \Rcal$. Then, there exists $G \in \Lrm^2(\Omega;\Msn)$ such that
\[
  \dprb{ g,v }=\int_\Omega G:e(v)\dd x, \qquad  v \in \Hrm^1(\R^n;\R^n)
\]
 and
\[
  \|G\|_{\Lrm^2(\Omega;\Msn)} = \|g\|_{\Hrm^{-1}(\R^n;\R^n)}.
\]
\end{prop}

\begin{proof}
Since $\Omega$ has Lipschitz boundary, it is an $\Hrm^1$-extension domain. Therefore, any $g \in \Hrm^{-1}(\R^n;\R^n)$ satisfying $\supp(g) \subset \overline\Omega$ and $\langle g,r\rangle=0$ for all $r \in \Rcal$ can be identified with an element of the dual space $[\Hrm^1(\Omega;\R^n)/\Rcal]^*$, where $\Rcal$ as before denotes the space of rigid deformations. Using Korn's inequality (see~\cite{Nitsche81,Horgan95,Temam84book}), we can endow $\Hrm^1(\Omega;\R^n)/\Rcal$ with the norm
$$\|v\|:=\|e(v)\|_{\Lrm^2(\Omega;\Msn)}, \qquad v \in \Hrm^1(\Omega;\R^n)/\Rcal,$$
which is equivalent to the canonical quotient norm.

Let us consider the map $P \colon \Hrm^1(\Omega;\R^n)/\Rcal \to \Lrm^2(\Omega;\Msn)$ defined by
$$P(v):=e(v), \qquad   v \in \Hrm^1(\Omega;\R^n)/\Rcal.$$
Clearly, $P$ defines an isometric isomorphism from  $\Hrm^1(\Omega;\R^n)/\Rcal$ (equipped with the above norm) to its range $E:={\rm Im}(P)$ in $\Lrm^2(\Omega;\Msn)$, which is therefore a closed subspace of $\Lrm^2(\Omega;\Msn)$. Seeing $g$ as an element of the space $[\Hrm^1(\Omega;\R^n)/\Rcal]^*$, we define
$$L:=g \circ P^{-1}$$
which is an element of the dual space $E^*$ of $E$ and which satisfies
$$\|L\|_{E^*}=\|g\|_{[\Hrm^1(\Omega;\R^n)/\Rcal]^*}.$$
According to the Hahn--Banach extension theorem, $L$ can be extended to $\tilde L \in [\Lrm^2(\Omega;\Msn)]^*$ with
$$\|\tilde L\|_{[\Lrm^2(\Omega;\Msn)]^*}=\|L\|_{E^*}.$$
Then, since $\Lrm^2(\Omega;\Msn)$ is isometrically isomorphic to its dual, there is $G \in \Lrm^2(\Omega;\Msn)$ such that
$$\tilde L(\xi)=\int_\Omega G:\xi\dd x, \qquad \xi \in \Lrm^2(\Omega;\Msn)$$
and
$$\|G\|_{\Lrm^2(\Omega;\Msn)}=\|\tilde L\|_{[\Lrm^2(\Omega;\Msn)]^*}.$$
As a consequence, we get for all $v \in  \Hrm^1(\Omega;\R^n)/\Rcal$ that
$$\dprb{ g,v }=\dprb{ g\circ P^{-1},e(v)}=\dprb{ L,e(v) }=\dprb{ \tilde L,e(v)}=\int_\Omega G:e(v)\dd x$$
and
\begin{align*}
\|G\|_{\Lrm^2(\Omega;\Msn)}&=\|\tilde L\|_{[\Lrm^2(\Omega;\Msn)]^*}\\
&=\|L\|_{E^*}\\
&=\|g\|_{[\Hrm^1(\Omega;\R^n)/\Rcal]^*}\\
&=\|g\|_{\Hrm^{-1}(\R^n;\R^n)},
\end{align*}
which completes the proof of the proposition.
\end{proof}

\subsection{Approximation of measures with compact support}\label{s:mollification}

In this section we will construct regularizations of a given measure supported in $\overline{\Omega}$ that remain supported in $\overline\Omega$. In contrast, classical mollification would produce regularizations with support in a set slightly larger than $\Omega$ and thus outside the domain of finiteness of our limit functional $\overline{\mathscr E}$. We adapt the mollification construction of~\cite{ErnGuermond} to the case of measures.

In all of the following we assume that $\Omega \subset \R^n$ is a bounded open set with $\Crm^2$-boundary. We first show the existence of a transversal field for $\Omega$. The main point, contrary to~\cite[Sections~2.1 and~2.2]{ErnGuermond} (see also~\cite{HofmannMitreaTaylor07}), is that the transversal field can be taken to be a gradient. This property will be instrumental in constructing an approximation of the identity (later denoted by $\theta_\delta$) whose gradient is a symmetric matrix field.

\begin{lem}
Let $\Omega \subset \R^n$ be a bounded open set with $\Crm^2$-boundary.  There exists a scalar function $k\in \Crm^2_c(\R^n)$ such that $\nabla k= \nu$ on $\partial \Omega$, where $\nu$ is the unit outer normal to $\partial \Omega$. 
\end{lem}

\begin{proof}
Let us consider the signed distance function to $\partial \Omega$ defined by
$$d(x):=\begin{cases}
-\dist(x,\partial \Omega) & \text{if }x \in \Omega,\\
\phantom{-}\dist(x,\partial \Omega) & \text{if }x \notin \Omega.
\end{cases}$$
It is well known (see, e.g.,~\cite[Lemma~14.16]{GilbargTrudinger98book}) that $d$ is of class $\Crm^2$ in an open neighborhood $U$ of $\partial \Omega$ and that $\nabla d=\nu$ on $\partial\Omega$. Let $\varphi \in \Crm^\infty_c(\R^n)$ be such that 
$0 \leq \varphi \leq 1$ in $\R^n$, $\supp(\varphi) \subset U$ and $\varphi \equiv 1$ in an open neighborhood of $\partial\Omega$. Then, $k:= \varphi d \in \Crm^2_c(\R^n)$ satisfies $\nabla k=\nu$ on $\partial\Omega$, as required.
\end{proof}

For all $\delta\in[0,1]$ let us define the {\it expansion map} $\theta_\delta\colon\R^n\to \R^n$ by
$$\theta_\delta(x):=x+3\delta \nabla k(x), \qquad x\in\R^n.$$
Clearly, we have the following properties for all $\delta > 0$:
\begin{enumerate}[(i)]
	\item $\theta_\delta\in \Crm^1(\R^n;\R^n)$;
	\item $\nabla \theta_\delta=\Id+3\delta D^2k$ is $\Msn$-valued;
	\item there exists $c>0$ (independent of $\delta$) such that
	$$ \max_{x \in \R^n} \bigl\{|\theta_\delta(x)-x|+|\nabla \theta_\delta(x)-\Id| \bigr\} \leq c\delta.$$ 
	\end{enumerate} 

For $\delta > 0$ small, the map $\theta_\delta$ is a $\Crm^1$-diffeomorphism from $\R^n$ onto its range. Indeed, let
\[
\delta_0' := \min \bigl\{ (3\|D^2 k\|_\infty)^{-1}, 1 \bigr\}.
\]
For any $\delta \in (0,\delta_0')$, we have that $\nabla \theta_\delta(x)=\Id+3\delta D^2k(x)$ is invertible for all $x \in \R^n$,  while
$$|x-x'|\leq 3 \delta \|D^2 k\|_{\infty}|x-x'|+|\theta_\delta(x)-\theta_\delta(x')|, \qquad x,x' \in \R^n,$$ 
and hence $\theta_\delta$ is injective in $\R^n$. The claim then follows from the inverse function theorem.

Moreover, we have that $\theta_\delta$ converges  uniformly in $\R^n$ to the identity as $\delta \to 0$, as well as $\nabla \theta_\delta$, $(\nabla \theta_\delta)^{-1} \to \Id$ and $\det\nabla \theta_\delta \to 1$ uniformly in $\R^n$. 

The next result shows that the map $\theta_\delta$ indeed expands $\Omega$ into a larger domain.

\begin{lem}\label{lem:2delta}
Let $\Omega \subset \R^n$ be a bounded open set with $\Crm^2$-boundary. There exists a $\delta_0'' > 0$ such that for all $\delta \in (0,\delta_0'')$ we have
$$\theta_\delta(\partial\Omega) + B_{2\delta}(0) \Subset \R^n \setminus \overline\Omega.$$
\end{lem}

\begin{proof}
Let $\delta_0'' > 0$ be so small that for all $z \in \partial \Omega$ it holds that $B_{3\delta}(z+3\delta\nu(z)) \subset \R^n \setminus \cl{\Omega}$ and $B_{3\delta}(z+3\delta\nu(z)) \subset \{k=d\}$ with $\nu(z)$ the unit outer normal to $\Omega$ at $z$. The first inclusion follows from the so-called uniform outer sphere condition, which is implied by the $\Crm^2$-regularity of $\partial \Omega$ and the ensuing fact that all scalar curvatures of $\partial \Omega$ are bounded (see~\cite[Section~14.6]{GilbargTrudinger98book}). The second inclusion is a consequence of the fact that $k=d$ in an open neighborhood of $\partial\Omega$.

Let us now fix $z \in \partial\Omega$. Since $k=d$ in an open neighborhood of $\partial\Omega$, we have $\theta_\delta(z)=z+3\delta\nabla d(z)$. Using that $\nabla d(z)=\nu(z)$, we get that
$$\theta_\delta(z)=z+3\delta \nu(z),$$ 
which shows that $\theta_\delta(z)+B_{2\delta}(0) \Subset \R^n\setminus \overline\Omega$. 
\end{proof}

Set $\delta_0 := \min\{\delta_0',\delta_0''\}$ and define our regularization as follows:

\begin{defn}
Let $\lambda \in \M(\overline\Omega;\Msn)$ and $\delta\in (0,\delta_0)$. For all $x\in \R^n$, we define
\[
\lambda^\delta(x):=\frac{\det(\nabla \theta_\delta(x))}{\delta^n}\int_{\R^n}\eta\Bigl(\frac{\theta_\delta(x)-y}{\delta}\Bigr)\dd\lambda(y)(\nabla \theta_\delta(x))^{-1},
\]
where $\eta \in \Crm^\infty_c(\R^n)$ is a standard mollifier (satisfying $\eta(-x)=\eta(x)$ for all $x \in \R^n$, $\supp(\eta)\subset B_1(0)$ and $\int_{\R^n}\eta(z)\dd z=1$).
\end{defn}

Clearly, this construction is linear in $\lambda$. We next collect the relevant properties of $\lambda^\delta$.

\begin{prop}\label{prop:properies_mollification2}
Let $\Omega \subset \R^n$ be a bounded open set with $\Crm^2$-boundary. The following properties hold:	
\begin{enumerate}[(i)]
	\item $\lambda^\delta\in \Crm_c(\Omega;\Msn)$ for $\delta \in (0,\delta_0)$;\label{e:m32}
	\item $\lambda^\delta\wsto \lambda$ in $\M(\R^n;\Msn)$ as $\delta \to 0$;\label{e:m52}
	\item $|\lambda^\delta|(\Omega) \to |\lambda|(\overline\Omega)$ as $\delta \to 0$;\label{e:m62}
	\item If $\dive\lambda \in \M(\R^n;\R^n)$ then $\dive \lambda^\delta \in \Crm_c(\Omega;\R^n)$ and $\dive\lambda^\delta \wsto \dive\lambda$ in $\M(\R^n;\R^n)$ as $\delta \to 0$. If further $\dive \lambda = 0$, then also $\dive \lambda^\delta = 0$.\label{e:m72}
	\item If $\lambda \in \Lrm^1(\Omega;\Msn)$, then $\|\lambda^\delta - \lambda\|_{\Lrm^1(\Omega;\Msn)} \to 0$ as $\delta \to 0$.\label{e:m82}
\end{enumerate}	
\end{prop} 

\begin{proof}
Since $\nabla \theta_\delta(x)=\Id+3\delta D^2k(x) \in \Msn$ for all $x \in \R^n$, we observe that $\lambda^\delta$ is $\Msn$-valued.

\emph{Step~1.} Concerning~\ref{e:m32}, since $\det(\nabla \theta_\delta) \in \Crm(\R^n)$ and also $(\nabla\theta_\delta)^{-1} \in \Crm(\R^n;\Msn)$, it is enough to check that the map
$$\tilde \lambda^\delta(x):=\frac{1}{\delta^n}\int_{\R^n}\eta\Bigl(\frac{\theta_\delta(x)-y}{\delta}\Bigr)\dd\lambda(y),  \qquad x \in \R^n,$$
belongs to $\Crm_c(\Omega;\Msn)$. We first show that $\tilde \lambda^\delta$ is Lipschitz continuous in $\R^n$. Indeed, for all $x$, $x' \in \R^n$, we have
$$\tilde \lambda^\delta(x)-\tilde \lambda^\delta(x')=\frac{1}{\delta^n}\int_{\R^n}\Bigl(\eta\Bigl(\frac{\theta_\delta(x)-y}{\delta}\Bigr)-\eta\Bigl(\frac{\theta_\delta(x')-y}{\delta}\Bigr)\Bigr)\dd\lambda(y).$$	
Since $\eta$ and $\theta_\delta$ are uniformly Lipschitz continuous in $\R^n$, we have
$$\left|\eta\Bigl(\frac{\theta_\delta(x)-y}{\delta}\Bigr)-\eta\Bigl(\frac{\theta_\delta(x')-y}{\delta}\Bigr)\right|\leq \frac{c}{\delta}|x-x'|,$$
for some constant $c>0$, hence,
$$|\tilde\lambda^\delta(x)-\tilde \lambda^\delta(x')| \leq \frac{c}{\delta^{n+1}}|x-x'|,$$
that is, $\tilde\lambda^\delta$ is Lipschitz continuous in $\R^n$ and then $ \tilde \lambda^\delta \in \Crm(\R^n;\Msn)$.

We next show that $\tilde \lambda^\delta$ has support in $\Omega$. Let $\e_\delta:=\delta/(1+3\delta c_k)$, where $c_k$ is the Lipschitz constant of $\nabla k$, and let $x\in\Omega$ be such that $\dist(x,\partial\Omega)<\e_\delta$. Arguing as in~\cite[Lemma~4.1]{ErnGuermond}, we have
$$\theta_\delta(x)+B_{\delta}(0)\subset \R^n \setminus \overline{\Omega}.$$
Indeed, $\partial\Omega$ being compact, there exists a point $z \in \partial\Omega$ such that $|x-z|=\dist(x,\partial\Omega)<\e_\delta$. Then,
\begin{align*}
\theta_\delta(x) +B_\delta(0) &= \theta_\delta(z) + B_\delta(0) +\theta_\delta(x)-\theta_\delta(z)\\
&= \theta_\delta(z) + B_\delta(0) + x-z +3\delta[ \nabla k(x)-\nabla k(z)]
\end{align*}
and, thanks to Lemma~\ref{lem:2delta}, this set is contained in
\[
  \theta_\delta(z) + B_\delta(0) +B_{\e_\delta(1+3\delta c_k)}(0)
  \subset \theta_\delta(z)+B_{2\delta}(0)
  \subset \R^n \setminus \overline{\Omega}.
\]
Now, if $y\in \R^n$ is such that 
$$\frac{y-\theta_\delta(x)}{\delta}\in B_1(0),$$
we have $y\in \theta_\delta(x)+B_{\delta}(0)\subset\R^n\setminus\overline{\Omega}$. We conclude by definition of $\tilde \lambda^\delta$, using that $\supp(\eta) \subset B_1(0)$ and $\supp(\lambda) \subset \overline\Omega$, that $\tilde \lambda^\delta(x)=0$ for all $x\in\Omega$ such that $\dist(x,\partial\Omega)<\e_\delta$, and hence $\tilde \lambda^\delta\in \Crm_c(\Omega;\Msn)$.

\emph{Step~2.} Concerning~\ref{e:m52}, for all $\varphi\in \Crm_c(\R^n;\Msn)$, by Fubini's theorem and a change of variables we have
\begin{align*}
&\int_{\Omega}\tilde \lambda^\delta(x):\varphi(x)\det(\nabla\theta_\delta(x)) \dd x \\
&\qquad = \int_{\R^n}\varphi(x):\biggl(\int_{\R^n}\frac{\det(\nabla \theta_\delta(x))}{\delta^n}\,\eta\Bigl(\frac{\theta_\delta(x)-y}{\delta}\Bigr)\dd\lambda(y)\biggr)\dd x\\
&\qquad = \int_{\R^n}\biggl(\int_{\R^n}\varphi(x)\,\frac{\det(\nabla \theta_\delta(x))}{\delta^n}\,\eta\Bigl(\frac{\theta_\delta(x)-y}{\delta}\Bigr)\dd x\biggr):\di\lambda(y)\\
&\qquad = \int_{\overline\Omega}\biggl(\int_{B_1(0)}\varphi\big(\theta_\delta^{-1}(y+\delta z)\big)\eta(z)\dd z\biggr): \di\lambda(y).
\end{align*}
As $\delta\to 0$ we have that $\theta_\delta^{-1}(y+\delta z)\to y$ uniformly with respect to $(y,z)\in \overline \Omega \times B_1(0)$, and hence we conclude that 
\begin{align*}
\int_{\Omega}\tilde \lambda^\delta(x):\varphi(x) \det(\nabla\theta_\delta(x)) \dd x
&\to \int_{\overline\Omega}\biggl(\int_{B_1(0)}\eta(z)\dd z\biggr)\varphi(y):\di\lambda(y)\\
&=\int_{\overline\Omega}\varphi(y):\di\lambda(y).
\end{align*}
This entails that $\det(\nabla\theta_\delta)\tilde\lambda^\delta \wsto \lambda$ in $\M(\R^n;\Msn)$. Then, since $(\nabla\theta_\delta)^{-1} \to {\rm Id}$ uniformly in $\R^n$, we deduce that $\lambda^\delta=\det(\nabla\theta_\delta)\tilde\lambda^\delta(\nabla\theta_\delta)^{-1}  \wsto \lambda$ in $\M(\R^n;\Msn)$.

\emph{Step~3.} To prove~\ref{e:m62}, we first note that, by lower semicontinuity of the total variation, we have
$$|\lambda|(\overline\Omega)\leq\liminf_{\delta\to 0}|\lambda^\delta|(\Omega).$$
To get the other inequality, we first observe that since $\nabla \theta_\delta=\Id+3\delta D^2 k$, for all $\delta \in (0,\delta_0)$ we have
$$(\nabla \theta_\delta)^{-1}=\sum_{l=0}^{\infty} [-3\delta D^2k]^l=\Id+R_\delta \quad \text{ in }\overline \Omega,$$
where $R_\delta \in \Crm(\R^n;\Msn)$ with $\sup_{\overline\Omega}|R_\delta| \leq M\delta$ for some $M>0$ (independent of $\delta$). Thus, for all $x \in \R^n$,
\begin{align*}
\lambda^\delta(x) &=\frac{\det(\nabla \theta_\delta(x))}{\delta^n}\int_{\R^n}\eta\Bigl(\frac{\theta_\delta(x)-y}{\delta}\Bigr)\dd\lambda(y) \\
&\qquad +\frac{\det(\nabla \theta_\delta(x))}{\delta^n}\int_{\R^n}\eta\Bigl(\frac{\theta_\delta(x)-y}{\delta}\Bigr)\dd\lambda(y)R_\delta(x).
\end{align*}
Taking the total variation on both sides and using Fubini's theorem as well as a change of variables,
\begin{align*}
  |\lambda^\delta|(\Omega) &\leq \frac{1+M\delta}{\delta^n}\int_{\R^n}\left(\int_{\R^n}\eta\Bigl(\frac{\theta_\delta(x)-y}{\delta}\Bigr) |\det(\nabla \theta_\delta(x))|\dd|\lambda|(y)\right)\dd x \\
  &=\frac{1+M\delta}{\delta^n}\int_{\R^n}\left(\int_{\R^n}\eta\Bigl(\frac{\theta_\delta(x)-y}{\delta}\Bigr) |\det(\nabla \theta_\delta(x))| \dd x\right)\dd|\lambda|(y)\\
&=(1+M\delta) \int_{\overline\Omega}\left(\int_{\R^n} \eta(z)\dd z\right)\dd |\lambda|(y)\\
&=(1+M\delta)|\lambda|(\overline \Omega).
\end{align*}
This implies that 
$$\limsup_{\delta\to 0}|\lambda^\delta|(\Omega)\leq |\lambda|(\overline\Omega),$$
which proves~\ref{e:m62}.

\emph{Step~4.} Concerning~\ref{e:m72}, we first observe that if $\lambda\in \Crm^\infty_c(\R^n;\Msn)$, then a change of variables implies that
$$\lambda^\delta(x)=\int_{\R^n} \eta(z) \det(\nabla \theta_\delta(x))\lambda(\theta_\delta(x)+\delta z)(\nabla \theta_\delta(x))^{-1}\dd z.$$
Using that $\nabla\theta_\delta$ is symmetric, it holds that
\[
  \dive(\det(\nabla \theta_\delta) \nabla \theta_\delta^{-1}) = \dive(\cof(\nabla \theta_\delta))=0
\]
by Cramer's rule and Piola's identity, the latter of which is proved, for instance, in~\cite[Lemma on p.~462]{Evans10book}. Hence, we get
\begin{align*}
\dive\lambda^\delta(x)&= \int_{\R^n} \eta(z) \det(\nabla \theta_\delta(x))[\dive\lambda](\theta_\delta(x)+\delta z)\dd z\\
 &= \frac{\det(\nabla \theta_\delta(x))}{\delta^n}\int_{\R^n}\eta\Bigl(\frac{\theta_\delta(x)-y}{\delta}\Bigr)[\dive\lambda](y)\dd y.
\end{align*}
Thus, using a standard approximation of measures (e.g., by usual convolution), we infer that if $\lambda \in \M(\overline\Omega;\Msn)$ with $\dive \lambda \in \M(\R^n;\R^n)$, then
$$\dive\lambda^\delta(x)=\frac{\det(\nabla \theta_\delta(x))}{\delta^n}\int_{\R^n}\eta\Bigl(\frac{\theta_\delta(x)-y}{\delta}\Bigr)\dd[\dive\lambda](y).$$
Moreover, arguing as in Step~1 shows that $\dive \lambda^\delta \in \Crm_c(\Omega;\R^n)$. Owing again to Fubini's theorem and a change of variables, we deduce, like in Step~2, that $\dive\lambda^\delta \wsto \dive\lambda$ in $\M(\R^n;\R^n)$. That $\dive \lambda^\delta = 0$ if $\dive \lambda = 0$ is obvious from the above formula, which concludes the proof of~\ref{e:m72}.

\emph{Step~5.} Finally, for~\ref{e:m82} we observe that if $\lambda \in \Lrm^1(\Omega;\Msn)$, then by a similar argument as in Step~3, $\|\lambda^\delta\|_{\Lrm^1(\Omega;\Msn)} \leq C \|\lambda\|_{\Lrm^1(\Omega;\Msn)}$ for all $\delta \in (0,1)$ and a constant $C > 0$. Let $\e > 0$ and take $g \in \Crm_c(\Omega;\Msn)$ with $\|\lambda-g\|_{\Lrm^1(\Omega;\Msn)} \leq \e$. Then, using the linearity of the regularization,
\begin{align*}
  \|\lambda^\delta-\lambda\|_{\Lrm^1(\Omega;\Msn)}
  &\leq \|(\lambda-g)^\delta\|_{\Lrm^1(\Omega;\Msn)} 
    + \|g^\delta-g\|_{\Lrm^1(\Omega;\Msn)} \\
  &\qquad + \|g-\lambda\|_{\Lrm^1(\Omega;\Msn)} \\
  &\leq  (C+1) \|\lambda-g\|_{\Lrm^1(\Omega;\Msn)} + \|g^\delta-g\|_{\Lrm^1(\Omega;\Msn)} \\
  &\leq  (C+1)\e + \|g^\delta-g\|_{\Lrm^1(\Omega;\Msn)}.
\end{align*}
The second term converges to zero as $\delta \todown 0$ since for $g \in \Crm_c(\Omega;\Msn)$ we have $g^\delta \to g$ uniformly (because $\nabla \theta_\delta$, $(\nabla \theta_\delta)^{-1} \to \Id$ and $\det\nabla \theta_\delta \to 1$ uniformly and also using standard arguments for mollifiers). As $\e > 0$ was arbitrary, we thus have shown that $\|\lambda^\delta-\lambda\|_{\Lrm^1(\Omega;\Msn)} \to 0$.
\end{proof}

\begin{rem}\label{rem:2.9}
Since $\supp(\lambda^\delta)\subset \Omega$, a further approximation of $\lambda^\delta$ by means of usual convolution would produce a smooth approximation of $\lambda$ in $\Crm^\infty_c(\Omega;\Msn)$  satisfying all the requirements of Proposition~\ref{prop:properies_mollification2}.
\end{rem}

\begin{rem}
If $\mu\in \M^1(\overline\Omega)$ and $\delta \in (0,\delta_0)$, we can define analogously
\[
\mu^\delta(x):=\frac{\det(\nabla \theta_\delta(x))}{\delta^n}\int_{\R^n}\eta\Bigl(\frac{\theta_\delta(x)-y}{\delta}\Bigr)\dd\mu(y), \qquad x \in \R^n.
\]
Then, we have similarly
\begin{enumerate}[(i)]
\item $\mu^\delta\in \Crm_c(\Omega)$;\label{e:m3}
\item $\mu^\delta(\Omega)=1$;\label{e:m4}
\item $\mu^\delta\wsto \mu$ in $\M(\R^n)$ as $\delta \to 0$.\label{e:m5}
\end{enumerate}
Indeed, properties (i) and (iii) follow from the same argument than in the proof of Proposition~\ref{prop:properies_mollification2}. Concerning (ii), using that $\supp(\mu^\delta)\subset\Omega$, Fubini's theorem and a change of variables, we get
\begin{align*}
\mu^\delta(\Omega) &=\int_{\R^n}\frac{\det(\nabla \theta_\delta(x))}{\delta^n}\biggl(\int_{\R^n}\eta\Bigl(\frac{\theta_\delta(x)-y}{\delta}\Bigr)\dd\mu(y)\biggr)\dd x \\
&=\int_{\R^n}\biggl(\int_{\R^n}\frac{\det(\nabla \theta_\delta(x))}{\delta^n}\eta\Bigl(\frac{\theta_\delta(x)-y}{\delta}\Bigr)\dd x\biggr)\dd\mu(y)\\
&= \int_{\R^n}\biggl(\int_{\R^n}\eta(z)\dd z\biggr)\dd\mu(y) \\
&=1
\end{align*}
because $\mu$ and $\eta\LL^n$ are probability measures. Moreover, as in Remark~\ref{rem:2.9}, we can then in turn mollify $\mu^\delta$ by usual convolution to get a smooth approximating sequence in $\Crm^\infty_c(\Omega)$.
\end{rem}

\section{Compactness and lower bound} \label{sc:lower}

We first identify the natural topology with respect to which minimizing sequences of~\eqref{eq:Ceps} are expected to converge. This follows from the following compactness result. Note that the boundedness assumption of the divergence term will usually be guaranteed by the fact that, in the definition~\eqref{eq:Ceps} of $\mathscr C_\e$, the divergence of the competitor stresses are prescribed in $\Hrm^{-1}(\R^n;\R^n)$. This property will also be instrumental in the proof of the lower bound in order to apply a compensated compactness argument on the rescaled stresses $\sqrt\e\sigma_\e\mu_\e$.

\begin{prop} \label{prop:compactness}
Assume that for all $\e > 0$ we are given
\[
  (\sigma_\e,\mu_\e) \in X_\e(\Omega)
\]
such that
\[
  \sup_{\e > 0} \int_{\R^n}|\sigma_\e|^2\dd\mu_\e < \infty  \qquad\text{and}\qquad
  \sup_{\e > 0} \, \norm{\dive (\sigma_\e\mu_\e)}_{\Hrm^{-1}(\R^n;\R^n)} < \infty.
\]
Then, there exist a sequence $\{\e_k\}_{k\in \N}$ with $\e_k \todown 0$ and $(\sigma,\mu) \in X(\Omega)$ with $\dive(\sigma\mu) \in \Hrm^{-1}(\R^n;\R^n)$ such that
\[
 \begin{cases}
\mu_{\e_k} \wsto \mu &\text{in $\M(\R^n)$}, \\
\sigma_{\e_k} \mu_{\e_k} \wsto \sigma \mu  &\text{in $\M(\R^n;\Msn)$},\\
\dive(\sigma_{\e_k}\mu_{\e_k})\wto\dive(\sigma \mu) &\text{in }\Hrm^{-1}(\R^n;\R^n),
\end{cases}
\]
and
\[
\begin{cases}
  \sqrt{\e_k} \, \sigma_{\e_k} \mu_{\e_k} \wto 0  & \text{in $\Lrm^2(\R^n;\Msn)$},\\
  \dive( \sqrt{\e_k} \, \sigma_{\e_k} \mu_{\e_k}) \to 0  &\text{in $\Hrm^{-1}(\R^n;\R^n)$.}
  \end{cases}
\]
\end{prop}

\begin{proof}
Let $(\sigma_\e,\mu_\e) \in X_\e(\Omega)$ be such that
$$\int_\Omega \frac12 |\sigma_\e|^2  \dd \mu_\e \leq M, \qquad \|\dive(\sigma_\e \mu_\e)\|_{\Hrm^{-1}(\R^n;\R^n)}\leq M$$
for some $M>0$. By the definition of $X_\e(\Omega)$, there exists a set $\omega_\e \in \mathcal A_\e$ such that $\mu_\e=\frac{1}{\e}\LL^n\res\omega_\e \in \M^1(\overline\Omega)$. Moreover, the family of measures $\{\sigma_\e\mu_\e\}_{\e>0}$ is uniformly bounded in $\M(\R^n;\Msn)$ according to the Cauchy--Schwarz inequality since
$$\int_\Omega |\sigma_\e| \dd \mu_\e \leq \left(\int_\Omega |\sigma_\e|^2 \dd \mu_\e\right)^{1/2}\leq (2M)^{1/2}.$$
Thus, there is a subsequence $\{\e_k\}_{k \in \N}$ with $\e_k \todown 0$ and two measures $\mu \in \M^1(\overline\Omega)$ and $\lambda\in\M(\overline\Omega;\Msn)$ such that $\mu_{\e_k}\wsto\mu$ in $\M(\R^n)$, $\sigma_{\e_k} \mu_{\e_k} \wsto \lambda$ in $\M(\R^n;\Msn)$  and  $\dive(\sigma_{\e_k} \mu_{\e_k}) \wto \dive \lambda$ in $\Hrm^{-1}(\R^n;\R^n)$. 

We will next show that $\lambda$ is absolutely continuous with respect to $\mu$. Indeed, for all $\varphi \in \Crm_c(\R^n;\Msn)$, according to the Cauchy--Schwarz inequality, we infer that
\begin{align*}
\left|\int_\Omega \sigma_{\e_k}:\varphi \dd \mu_{\e_k}\right|
&\leq  \left(\int_\Omega |\sigma_{\e_k}|^2 \dd \mu_{\e_k}\right)^{1/2} \left(\int_\Omega|\varphi|^2 \dd \mu_{\e_k}\right)^{1/2} \\
&\leq  (2M)^{1/2}\left(\int_\Omega|\varphi|^2 \dd \mu_{\e_k}\right)^{1/2}.
\end{align*}
Thus, passing to the limit in the previous inequality yields
$$\left|\int_{\overline \Omega} \varphi : \dd \lambda\right|\leq (2M)^{1/2}\left(\int_{\overline \Omega}|\varphi|^2 \dd \mu\right)^{1/2}.$$
This last inequality actually ensures that $\lambda$ is absolutely continuous with respect to $\mu$. The Radon--Nikod\'ym and Riesz representation theorems give $\sigma \in \Lrm^2(\R^n,\mu;\Msn)$ such that $\lambda=\sigma\mu$. We have thus established the existence of $(\sigma,\mu) \in X(\Omega)$ with $\dive(\sigma\mu) \in \Hrm^{-1}(\R^n;\R^n)$ such that $\mu_{\e_k}\wsto\mu$ in $\M(\R^n)$, $\sigma_{\e_k}\mu_{\e_k}\wsto\sigma\mu$ in $\M(\R^n;\Msn)$ and $\dive(\sigma_{\e_k} \mu_{\e_k}) \wto \dive (\sigma\mu)$ in $\Hrm^{-1}(\R^n;\R^n)$.

Let us define the rescaled stress $\tau_{\e_k}:=\sqrt{\e_k} \sigma_{\e_k} \mu_{\e_k}$. Since the sequence $\{\sigma_{\e_k}\mu_{\e_k}\}_{k \in \N}$ is bounded in $\Lrm^1(\Omega;\Msn)$, it follows that
\[
  \tau_{\e_k} \to 0  \qquad\text{in $\Lrm^1(\R^n;\Msn)$.}
\]
On the other hand, since
\[
\int_\Omega|\tau_{\e_k}|^2 \dd x = \e_k \int_{\Omega} |\sigma_{\e_k}|^2 \frac{\chi_{\omega_{\e_k}}}{\e_k^2} \dd x = \int_\Omega |\sigma_{\e_k}|^2  \dd \mu_{\e_k} \leq 2M,
\]
thus, $\{\tau_{\e_k}\}_{k\in\N}$ is bounded in $\Lrm^2(\R^n;\Msn)$ and $\tau_{\e_k}\wto 0$ in $\Lrm^2(\R^n;\Msn)$. Finally, using that $\{\dive(\sigma_{\e_k}\mu_{\e_k})\}_{k \in \N}$ is bounded in $\Hrm^{-1}(\R^n;\R^n)$, we have that $\dive \tau_{\e_k}=\sqrt{\e_k} \dive(\sigma_{\e_k}\mu_{\e_k}) \to 0$ in $\Hrm^{-1}(\R^n;\R^n)$.
\end{proof}

We now derive a lower bound estimate using a similar technique as the one in~\cite{BabadjianIurlanoRindler19?}. The main argument consists of harnessing Lemma~\ref{lem:supR} by splitting the original energy density $\tau \mapsto \frac12|\tau|^2$ as the sum of a $\Lambda_{\dive}$-convex quadratic form $Q_\xi$, where $\xi \in \Msn$ with $\rho(\xi) \leq 1$, and a function which, after optimization with respect to all such $\xi$, precisely gives the right limit energy density $\bar j^*$. The $\Lambda_{\dive}$-convex part of this splitting acts on the rescaled stress $\tau_\e:=\sqrt\e\sigma_\e\mu_\e$ which converges weakly to $0$ in $\Lrm^2(\R^n;\Msn)$ and $\dive\tau_\e\to 0$ strongly in $\Hrm^{-1}(\R^n;\R^n)$. Thanks to a compensated compactness argument, it can thus be estimated from below by zero in the limit.

\begin{prop}\label{prop:liminf}
Let $\mu \in \mathcal M^1(\overline\Omega)$ and $\{\mu_\e\}_{\e>0}$ be a family in $\M^1(\overline\Omega)$ such that $\mu_\e \wsto \mu$ in $\mathcal M(\R^n)$. Then,
$$\overline{\mathscr C}(\mu) \leq \liminf_{\e \todown 0}\mathscr C_\e(\mu_\e).$$
\end{prop}

\begin{proof}
If $\liminf_{\e \todown 0} \mathscr C_\e(\mu_\e)=\infty$, then there is nothing to prove. Otherwise, we can extract a subsequence such that
\begin{equation}\label{eq:extraction}
\lim_{k \to \infty}\mathscr C_{\e_k}(\mu_{\e_k})=\liminf_{\e \todown 0}\mathscr C_\e(\mu_\e)<\infty.
\end{equation}
To simplify notation, we write from now on $\mu_k := \mu_{\e_k}$. By definition of $\mathscr C_{\e_k}(\mu_k)$ (and the Direct Method), there exists $\sigma_k \in \Lrm^2(\R^n,\mu_k;\Msn)$ such that $-\dive(\sigma_k\mu_k)=f$ in $\Dcal'(\R^n;\R^n)$ and
\begin{equation}\label{eq:1403}
\mathscr C_{\e_k}(\mu_k)=\int_{\R^n} \frac12 |\sigma_k|^2\dd\mu_k.
\end{equation}
Thanks to Proposition~\ref{prop:compactness} there exists a map $\sigma \in \Lrm^2(\R^n,\mu;\Msn)$ with $-\dive(\sigma\mu) =f$ in $\mathcal D'(\R^n;\R^n)$ such that, up to a subsequence,
\[
  \sigma_k\mu_k\wsto \sigma \mu  \qquad\text{in $\M(\R^n;\Msn)$}
\]
and
\begin{equation}\label{eq:prop1.1}
\begin{cases}
  \sqrt{\e_k} \, \sigma_k \mu_k \wto 0 &\text{in $\Lrm^2(\R^n;\Msn)$},\\
  \dive( \sqrt{\e_k} \, \sigma_k \mu_k)\to 0 &\text{in $\Hrm^{-1}(\R^n;\R^n)$.}
  \end{cases}
\end{equation}
We define the positive measure
\[
\gamma_k(A):=\int_A \frac12 |\sigma_k|^2 \dd \mu_k,  \qquad\text{$A \subset \Omega$ Borel.}
\]
Up to a subsequence, we can assume that $\gamma_k \wsto \gamma$ in $\M(\R^n)$ for some positive measure $\gamma \in \M^+(\overline\Omega)$. Let $\Omega_0$ be the set of all points $x_0 \in \overline \Omega$ that are $\Lrm^2$-Lebesgue points of $\sigma$ with respect to $\mu$, i.e.,
\begin{equation}\label{eq:lebpt}
\lim_{\rho \to 0}\frac{1}{\mu(B_\varrho(x_0))}\int_{B_\varrho(x_0)}|\sigma-\sigma(x_0)|^2 \dd \mu=0,
\end{equation}
and such that the Radon--Nikod\'ym derivative 
\begin{equation}\label{eq:lebpt2}
\frac{\di \gamma}{\di\mu}(x_0)=\lim_{\varrho\to 0}\frac{\gamma(B_\varrho(x_0))}{\mu(B_\varrho(x_0))}
\end{equation}
exists (as the above limit) and is finite. From the Besicovitch differentiation theorem, these properties are satisfied for $\mu$-almost every point $x_0$ in $\overline \Omega$ so that $\mu(\overline \Omega \setminus \Omega_0)=0$. Let us fix $x_0 \in \Omega_0$ and let us consider a sequence of radii $\{\varrho_j\}_{j\in \N}$ such that $\varrho_j \todown 0$ and $\gamma(\partial B_{\varrho_j}(x_0))=0$ for all $j \in \N$. Thus, we have 
\[
\gamma(B_{\varrho_j}(x_0))=\lim_{k \to\infty}\gamma_k(B_{\varrho_j}(x_0)).
\]

Denote by $\sigma_1(x_0),\ldots,\sigma_n(x_0)$ the eigenvalues of the symmetric matrix $\sigma(x_0)$ ordered as singular values, $|\sigma_1(x_0)|\leq \cdots \leq |\sigma_n(x_0)|$. For every $\xi \in \Msn$ with $\rho(\xi) \leq 1$ we consider the $\Lambda_{\dive}$-convex quadratic form $Q_\xi \colon \Msn\to \R$ introduced in Lemma~\ref{lem:supR}, namely
$$Q_\xi(\tau):=\frac12 |\tau|^2 - \frac12 (\xi:\tau)^2, \qquad \tau \in \Msn.$$
Since, by~\eqref{eq:prop1.1}, $\tau_k:=\sqrt{\e_k}\sigma_k\mu_k \wto 0$ weakly in $\Lrm^2(\R^n;\Msn)$ and $\dive\tau_k \to 0$ strongly in $\Hrm^{-1}(\R^n;\R^n)$, it follows from the theory of compensated compactness, see, e.g.,~\cite[Theorem~8.30]{Rindler18book}, that for all $j \in \N$ and all $\varphi \in \Crm^\infty_c(B_{\varrho_j}(x_0))$ with $0 \leq \varphi \leq 1$,
$$\liminf_{k \to \infty} \int_{B_{\varrho_j}(x_0)} \varphi^2 Q_\xi(\tau_k) \dd x \geq 0.$$
Hence,
\begin{align*}
\gamma(B_{\varrho_j}(x_0)) &\geq \liminf_{k \to \infty} \int_{B_{\varrho_j}(x_0)} \frac{\varphi^2}{2}|\sigma_k|^2 \dd \mu_k \\
&=\liminf_{k \to \infty} \int_{B_{\varrho_j}(x_0)} \frac{\varphi^2}{2} |\tau_k|^2 \dd x\\ 
&\geq \liminf_{k \to \infty} \int_{B_{\varrho_j}(x_0)}\varphi^2\left(\frac12 |\tau_k|^2-Q_\xi(\tau_k)\right) \dd x \\
&= \liminf_{k \to \infty} \int_{B_{\varrho_j}(x_0)}\varphi^2 \left(\frac12 |\sigma_k|^2-Q_\xi(\sigma_k)\right) \dd \mu_k\\
&= \liminf_{k \to \infty} \int_{B_{\varrho_j}(x_0)} \frac{\varphi^2}{2}(\xi : \sigma_k)^2 \dd \mu_k.
\end{align*}

Let $g \in \Crm_c(B_{\varrho_j}(x_0))$ be a nonnegative function such that $\|g\|_{\Lrm^2(B_{\varrho_j}(x_0),\mu)}=1$. Using the Cauchy--Schwarz inequality, we get that
$$\int_{B_{\varrho_j}(x_0)} \varphi^2(\xi : \sigma_k)^2 \dd \mu_k \geq \frac{1}{\int_{B_{\varrho_j}(x_0)} |g|^2 \dd \mu_k}\left(\int_{B_{\varrho_j}(x_0)} g\varphi|\xi : \sigma_k| \dd \mu_k \right)^2.$$
Since $\int_{B_{\varrho_j}(x_0)} |g|^2 \dd \mu_k \to \int_{B_{\varrho_j}(x_0)} |g|^2 \dd \mu=1$ as $k\to\infty$, we obtain that 
\[
\gamma(B_{\varrho_j}(x_0)) \geq \liminf_{k \to \infty}\frac12 \left(\int_{B_{\varrho_j}(x_0)} g\varphi|\xi : \sigma_k| \dd \mu_k \right)^2.
\]
The function $W \colon (x,\tau) \in  B_{\varrho_j}(x_0) \times \Msn \mapsto g(x)\varphi(x) |\xi:\tau|\in [0,\infty)$ is continuous, positively $1$-homogeneous and convex in its second variable. Then, setting $\lambda_k:=\sigma_k\mu_k \wsto \lambda = \sigma \mu$ (see the proof of Proposition~\ref{prop:compactness}), we deduce from Reshetnyak's lower semicontinuity theorem (see, e.g.,~\cite[Theorem~2.38]{AmbrosioFuscoPallara00book}) that
\begin{align*}
\liminf_{k \to \infty}\int_{B_{\varrho_j}(x_0)} g\varphi|\xi : \sigma_k| \dd \mu_k &=\liminf_{k \to \infty}\int_{B_{\varrho_j}(x_0)} W(x,\sigma_k) \dd\mu_k\\
&=\liminf_{k \to \infty}\int_{B_{\varrho_j}(x_0)} W\left(x,\frac{\di\lambda_k}{\di\mu_k}\right) \dd\mu_k \\
&=\liminf_{k \to \infty}\int_{B_{\varrho_j}(x_0)} W\left(x,\frac{\di\lambda_k}{\di|\lambda_k|}\right) \dd|\lambda_k| \\
&\geq  \int_{B_{\varrho_j}(x_0)} W\left(x,\frac{\di\lambda}{\di|\lambda|}\right) \dd|\lambda|.
\end{align*}
Furthermore,
\begin{align*}
\int_{B_{\varrho_j}(x_0)} W\left(x,\frac{\di\lambda}{\di|\lambda|}\right) \dd|\lambda|
&=\int_{B_{\varrho_j}(x_0)} W\left(x,\frac{\di\lambda}{\di\mu}\right) \dd\mu\\
&=\int_{B_{\varrho_j}(x_0)} W(x,\sigma) \dd \mu \\
&=\int_{B_{\varrho_j}(x_0)} g\varphi|\xi : \sigma| \dd \mu,
\end{align*}
and thus
$$\gamma(B_{\varrho_j}(x_0)) \geq \frac12\left(\int_{B_{\varrho_j}(x_0)} g \varphi |\xi : \sigma| \dd \mu\right)^2.$$
Passing first to the supremum with respect to all $g \in \Crm_c(B_{\varrho_j}(x_0))$ with $\|g\|_{\Lrm^2(B_{\varrho_j}(x_0),\mu)}=1$
in the right-hand side of the previous inequality yields
$$\gamma(B_{\varrho_j}(x_0)) \geq \int_{B_{\varrho_j}(x_0)} \frac{\varphi^2}{2} (\xi : \sigma)^2 \dd \mu$$
and then over all $\varphi \in \Crm^\infty_c(B_{\varrho_j}(x_0))$ with $0 \leq \varphi \leq 1$ leads to
$$\gamma(B_{\varrho_j}(x_0)) \geq \int_{B_{\varrho_j}(x_0)} \frac12  (\xi : \sigma)^2 \dd \mu$$
for all $j\in\N$ and all $\xi \in \Msn$ with $\rho(\xi) \leq 1$. Dividing the previous inequality by $\mu(B_{\varrho_j}(x_0))$, letting $j \to \infty$ and using~\eqref{eq:lebpt}--\eqref{eq:lebpt2} leads to
$$\frac{\di \gamma}{\di\mu}(x_0) \geq  \frac12  (\xi : \sigma(x_0))^2.$$
Hence, by Lemma~\ref{lem:supR}, taking the supremum over all such $\xi$, we obtain
$$\frac{\di \gamma}{\di\mu}(x_0) \geq  \bar j^*(\sigma(x_0)).$$
Integrating the previous inequality with respect to $\mu$ over $\R^n$ leads to
\begin{equation}\label{eq:lb}
\lim_{k\to \infty} \int_{\overline\Omega} \frac12 |\sigma_k|^2 \dd \mu_k=\lim_{k \to \infty}\gamma_k(\R^n)\geq \gamma(\R^n)\geq \int_{\overline \Omega} \bar j^*(\sigma) \dd \mu,
\end{equation}
where we used that $\mu$ is supported in $\overline \Omega$. Thus, combining~\eqref{eq:extraction},~\eqref{eq:1403} and~\eqref{eq:lb},
$$\liminf_{\e \todown 0} \mathscr C_\e(\mu_\e)\geq \overline{\mathscr E}(\sigma,\mu)\geq \overline{\mathscr C}(\mu),$$
as required.
\end{proof}

\begin{rem}
In the two-dimensional case, there is an alternative way to proceed by replacing the quadratic approximation $Q_\xi$ by the more elementary one
\[
  \tau \in \Mstwo \mapsto q_\alpha(\tau):=\frac12|\tau|^2 + \alpha \det\tau,  \qquad \alpha \in \{-1,+1\},
\]
introduced in Remark~\ref{rem:qalpha}. Indeed, in that case, since $\tau_k\wto 0$ weakly in $\Lrm^2(B_{\varrho_j}(x_0);\Mstwo)$ and $\dive\tau_k \to 0$ strongly in $\Hrm^{-1}(B_{\varrho_j}(x_0);\R^2)$, the theory of compensated compactness ensures that $\det \tau_k \wsto 0$ in $\Dcal'(B_{\varrho_j}(x_0))$, hence
\begin{align*}
\gamma(B_{\varrho_j}(x_0)) &\geq \lim_{k \to \infty} \int_{B_{\varrho_j}(x_0)} \frac{\varphi^2}{2} \left[|\tau_k|^2 +2\alpha\det\tau_k\right]  \dd x\\
 &= \lim_{k \to \infty} \int_{B_{\varrho_j}(x_0)} \frac{\varphi^2}{2} \left[|\sigma_k|^2 +2\alpha\det\sigma_k\right]  \dd\mu_k.
\end{align*}
The remaining parts of the proof are identical, applying this time Reshetnyak's lower semicontinuity theorem to the function $(x,\tau) \in  B_{\varrho_j}(x_0) \times \Mstwo \mapsto g(x)\varphi(x) \sqrt{|\tau|^2+2\alpha \det \tau}\in [0,\infty)$, which is continuous, positively $1$-homogeneous, and convex in its second variable. We end up with the inequality
$$\frac{\di \gamma}{\di\mu}(x_0) \geq \frac12|\sigma(x_0)|^2 + \alpha\det\sigma(x_0) \qquad \text{for }\alpha=\pm1,$$
hence, maximizing over $\alpha=\pm 1$ and using~\eqref{eq:maxalpha},
$$\frac{\di \gamma}{\di\mu}(x_0) \geq \bar j^*(\sigma(x_0)).$$
\end{rem}

\begin{rem}\label{rem:liminf}
More generally, assume that for some measurable sets $A_\e \subset \Omega$, $\{\mu_\e=\frac{1}{\e}\LL^n \res A_\e\}_{\e>0}$ is a family in $\M^+(\R^n)$ such that $\mu_\e \wsto \mu$ in $\M(\R^n)$ for some $\mu \in \M^+(\R^n)$ ($A_\e$ does not need to belong to the class $\mathcal A_\e$, and $\mu_\e,\mu$ do not need to be probability measures). Also assume that $\sigma_\e \in \Lrm^2(\R^n,\mu_\e;\Msn)$ is such that $\sigma_\e\mu_\e \wsto \sigma\mu$ in $\M(\R^n;\Msn)$ and $\dive(\sigma_\e\mu_\e) \wto \dive(\sigma\mu)$ in $\Hrm^{-1}(\R^n;\R^n)$ for some $\sigma\in \Lrm^2(\R^n,\mu;\Msn)$. Then, the arguments in this section may be adapted to show that still we have the lower bound
$$\liminf_{\e \todown 0}\int_{\R^n} \frac12 |\sigma_\e|^2\dd \mu_\e \geq \int_{\R^n}\bar j^*(\sigma)\dd\mu.$$
\end{rem}

\section{Relaxation of the Kohn--Strang functional} \label{sc:relax}

The objective of this section is to clarify several results about the relaxation of the Kohn--Strang functional, which do not seem to be easily accessible.

Let $\alpha>0$, $\beta>0$ and let $h \colon \Msn \to \R$ be Kohn--Strang function defined by
$$h(\tau):=
\begin{cases}
\alpha|\tau|^2+\beta & \text{if }\tau \neq 0,\\
0 & \text{if } \tau=0,$$
\end{cases}$$
and let $Q_{\dive}h$ be its (symmetric) $\dive$-quasiconvexification, defined for $\tau \in \Msn$ by
\begin{align}
Q_{\dive} h(\tau) := \inf \, \setBB{\int_{(0,1)^n} h(\varphi)\dd x }{ &\varphi \in \Crm^\infty_{\rm per}((0,1)^n;\Msn), \notag\\
 &\int_{(0,1)^n}\varphi \dd x=\tau,\; \dive\varphi=0 \text{ in }\R^n }.  \label{eq:Qh}
\end{align}

It is remarkable that an explicit expression for $Q_{\dive}$ can be computed (see~\cite{KohnStrang86_1,KohnStrang86_2,KohnStrang86_3,Allaire02book,AllaireKohn93,AllaireBonnetierFrancfortJouve97}) in terms of the eigenvalues $\tau_1,\ldots,\tau_n$ of a matrix $\tau \in \Msn$, ordered by size as singular values, i.e., $|\tau_1|\leq\cdots\leq|\tau_n|$. Indeed,~\cite[Sections~4,7]{AllaireKohn93} established that with
\[
  \hat{\rho}(\tau) := \sqrt{\frac{\alpha}{\beta}} \, \rho^\circ(\tau),  \qquad \tau \in \Msn,
\]
where $\rho^\circ$ is defined in~\eqref{eq:WMi}, it holds that
\[
  Q_{\dive} h(\tau) = \begin{cases}
    \alpha |\tau|^2 + \beta  &\text{if } \hat{\rho}(\tau) \geq 1, \\
    \alpha |\tau|^2 + \beta \hat{\rho}(\tau) (2-\hat{\rho}(\tau))  &\text{if } \hat{\rho}(\tau) \leq 1.
  \end{cases}
\]
Then, one may compute, for $n=2$,
$$Q_{\dive} h(\tau)=
\begin{cases}
\alpha|\tau|^2+\beta & \text{if }\rho^\circ(\tau)\geq\sqrt{\beta/\alpha},\\
2\sqrt{\alpha\beta}\rho^\circ(\tau)-2\alpha|\tau_1\tau_2| & \text{if }\rho^\circ(\tau)<\sqrt{\beta/\alpha},
\end{cases}
$$
and, for $n=3$,
$$Q_{\dive}h(\tau)=
\begin{cases}
\displaystyle \alpha|\tau|^2+\beta&  \text{if } \rho^\circ(\tau) \geq \displaystyle \sqrt{\beta/\alpha},\\
\displaystyle 2\sqrt{\alpha\beta} \rho^\circ(\tau)-2 \alpha |\tau_1\tau_2|&  \text{if } 
\begin{cases}
\displaystyle\rho^\circ(\tau)<\sqrt{\beta/\alpha},\\
|\tau_1|+|\tau_2|\leq |\tau_3|,
\end{cases}\\
\displaystyle \begin{aligned} &2\sqrt{\alpha\beta}  \rho^\circ(\tau) \\[-2pt] &\; +\alpha\left({\textstyle\frac12}|\tau|^2-|\tau_1\tau_2|-|\tau_1\tau_3|-|\tau_2\tau_3|\right) \end{aligned} & \text{if }
\begin{cases}
\displaystyle\rho^\circ(\tau)<\sqrt{\beta/\alpha},\\
|\tau_1|+|\tau_2|> |\tau_3|.
\end{cases}
\end{cases}$$

We first prove that $Q_{\dive}h$ is indeed (symmetric) $\dive$-quasiconvex. This is not immediate from standard results like~\cite[Proposition~3.4]{FonsecaMuller99} since those require the function to be at least upper semicontinuous, which our $h$ is not.

\begin{lem} \label{lem:Qdivh}
The function  $Q_{\dive}h$ is (symmetric) $\dive$-quasiconvex.
\end{lem}

\begin{proof}
We introduce, for all $M>0$, the function $h_M \colon \Msn \to \R$ defined by
$$h_M(\tau):=\min\bigl\{\alpha|\tau|^2+\beta,M|\tau|^2\bigr\}, \qquad \tau \in \Msn.$$
Note that $h_M$ is continuous and that it pointwise increases to $h$ as $M \to \infty$.  According to~\cite[Proposition~3.4]{FonsecaMuller99}, it follows that the (symmetric) $\dive$-quasiconvexification of $h_M$ given, for all $\tau \in \Msn$, by an analogous formula to~\eqref{eq:Qh}, is (symmetric) $\dive$-quasiconvex. Moreover, from the proof of~\cite[Theorem~3.1]{AllaireBonnetierFrancfortJouve97} (see pages 38--39 in \textit{loc.\ cit.}) we know that $Q_{\dive}h_M \toup Q_{\dive}h$ as $M \to \infty$. As a consequence, $Q_{\dive}h$ is the supremum of (symmetric) $\dive$-quasiconvex functions and hence it is (symmetric) $\dive$-quasiconvex as well.
\end{proof}

The following proposition is a relaxation result for the Kohn--Strang functional under a soft divergence-constraint. This is not easily found in the literature and does not follow from the standard relaxation results in the context of $\Acal$-quasiconvexity (see~\cite{FonsecaMuller99,BraidesFonsecaLeoni00}) since the integrand $h$ is not continuous. We give a complete proof.

\begin{prop}\label{prop:KS}
Let $\Omega\subset \R^n$ be a bounded open set with Lipschitz boundary. For all $\sigma \in \Lrm^2(\Omega;\Msn)$, the following assertions hold:
\begin{enumerate}[(i)]
\item For all sequences $\{\sigma_k\}_{k\in \N} \subset \Lrm^2(\Omega;\Msn)$ such that $\sigma_k\wto\sigma$ in $\Lrm^2(\Omega;\Msn)$ and $\dive\sigma_k\to\dive\sigma$ in $[\Hrm^1(\Omega;\R^n)]^*$,
$$\liminf_{k\to\infty}\int_\Omega h(\sigma_k)\dd x\geq\int_\Omega Q_{\dive}h(\sigma)\dd x.$$
\item There exists a sequence $\{\bar \sigma_k\}_{k\in\N} \subset \Lrm^2(\Omega;\Msn)$ such that $\bar \sigma_k\wto\sigma$ in $\Lrm^2(\Omega;\Msn)$, $\dive\bar\sigma_k\to\dive\sigma$ in $[\Hrm^1(\Omega;\R^n)]^*$, and
$$\lim_{k\to\infty}\int_\Omega h(\bar\sigma_k)\dd x=\int_\Omega Q_{\dive}h(\sigma)\dd x.$$
\end{enumerate}
\end{prop}

\begin{proof}
Let $B$ be a large ball in $\R^n$ such that $\overline\Omega\subset B$. Since $\Omega$ is an $\Hrm^1$-extension domain, it follows that any element of $[\Hrm^1(\Omega;\R^n)]^*$ can be identified with an element of $\Hrm^{-1}(B;\R^n)$, the dual space of $\Hrm^1_0(B;\R^n)$. 

Let $\Ocal(B)$ be the family of all open subsets of $B$. Let us introduce the functional $H \colon \Lrm^2(B;\Msn)\times \Ocal(B) \to [0,\infty)$, defined by 
$$H(\sigma;A):=\int_A h(\sigma)\dd x,$$
and its relaxation $\overline H \colon \Lrm^2(B;\Msn) \times \Ocal(B)\to [0,\infty)$, given by
\begin{align*}
\overline H(\sigma;A):=\inf \, \setBB{\liminf_{k \to\infty} H(\sigma_k;A) }{ &\sigma_k \to \sigma \text{ in }\Hrm^{-1}(A;\Msn), \\
&\dive\sigma_k \to \dive\sigma \text{ in }\Hrm^{-1}(A;\R^n) }.
\end{align*}
According to the integral representation result of~\cite[Theorem~2.3 \& Remark~2.4]{AnsiniDalMasoZeppieri14} (this result is actually stated for non-symmetric matrices, however, a careful inspection of the proof shows that it remains valid for symmetric matrices), there exists a Carath\'eodory function $\bar h \colon B \times \Msn \to \R$, which is (symmetric) $\dive$-quasiconvex in the second variable, such that for all $A \in \Ocal(B)$,
$$\overline H(\sigma;A)=\int_A \bar h(x,\sigma)\dd x, \qquad \sigma\in \Lrm^2(B;\Msn).$$

Let us first show that $\bar h$ is actually independent of the spatial variable $x$. Indeed, let $x_0$, $y_0 \in B$ and $\varrho>0$ be such that $B_\varrho(x_0) \cup B_\varrho(y_0) \subset B$ and let $\tau\in \Msn$ be a fixed matrix. For every sequence $\{\sigma_k\}_{k \in \N}$ in $\Lrm^2(B_\varrho(x_0);\Msn)$ such that $\sigma_k \to \tau$ in $\Hrm^{-1}(B_\varrho(x_0);\Msn)$ and $\dive\sigma_k \to \dive \tau=0$ in $\Hrm^{-1}(B_\varrho(x_0);\R^n)$, we define
\[
  \tilde \sigma_k(y):=\sigma_k(\frarg -y_0+x_0) \in \Lrm^2(B_\varrho(y_0);\Msn),
\]
which satisfies
\[
  \tilde \sigma_k \to \tau \quad\text{in $\Hrm^{-1}(B_\varrho(y_0);\Msn)$,} \qquad
  \dive\tilde\sigma_k \to 0 \quad\text{in $\Hrm^{-1}(B_\varrho(y_0);\R^n)$.}
\]
Using that the original Kohn--Strang integrand $h$ is spatially independent and the change of variables $y=x-x_0+y_0$, we get that
$$\overline H(\tau;B_\varrho(y_0)) \leq \liminf_{k \to\infty}\int_{B_\varrho(y_0)} h(\tilde\sigma_k(y))\dd y=\liminf_{k \to\infty}\int_{B_\varrho(x_0)} h(\sigma_k(x))\dd x.$$
Passing to the infimum among all minimizing sequences $\{\sigma_k\}_{k \in \N}$ as above, we get
$$\overline H(\tau;B_\varrho(y_0)) \leq \overline H(\tau;B_\varrho(x_0)).$$
A similar argument shows that the opposite inequality holds, so that
$$\int_{B_\varrho(y_0)} \bar h(x,\tau)\dd x=\overline H(\tau;B_\varrho(y_0)) = \overline H(\tau;B_\varrho(x_0))=\int_{B_\varrho(x_0)} \bar h(x,\tau)\dd x.$$
Dividing the previous inequality by $\varrho^n$ and passing to the upper limit as $\varrho\to 0$ (cf.\ formula (2.9) in~\cite{AnsiniDalMasoZeppieri14}) leads to 
$$\bar h(x_0,\tau)=\bar h(y_0,\tau),$$
proving that $\bar h$ is spatially independent.

We next show that 
\begin{equation}\label{eqbarhdiv}
\bar h=Q_{\dive}h.
\end{equation} 
Note that the compact embedding of $\Lrm^2(B;\Msn)$ into $\Hrm^{-1}(B;\Msn)$ and the coercivity property $H(\frarg;B) \geq c\|\frarg\|^2_{\Lrm^2(B;\Msn)}$ for some constant $c > 0$, ensure that
\begin{align*}
\overline H(\sigma;B)=\inf \, \setBB{\liminf_{k \to\infty} H(\sigma_k;B) }{ &\sigma_k \wto \sigma \text{ in }\Lrm^2(B;\Msn),\\
&\dive\sigma_k \to \dive\sigma \text{ in }\Hrm^{-1}(B;\R^n) }.
 \end{align*}
First of all, we observe that since $\overline H \leq H$, then $\bar h \leq h$. Using that $\bar h$ is (symmetric) $\dive$-quasiconvex, the first inequality $\bar h \leq Q_{\dive} h$ follows directly from~\eqref{eq:Qh}. To prove the converse inequality, we recall that $Q_{\dive}h$ is (symmetric) $\dive$-quasiconvex by Lemma~\ref{lem:Qdivh} and, by the explicit expression above, $Q_{\dive}h$ is also continuous. Then,~\cite[Theorem~3.7]{FonsecaMuller99} ensures that
$$\liminf_{k\to\infty}\int_B h(\sigma_k)\dd x\geq\liminf_{k\to\infty}\int_B Q_{\dive}h(\sigma_k)\dd x\geq\int_B Q_{\dive}h(\sigma)\dd x$$
for any sequence  $\{\sigma_k\}_{k\in \N}$ in $\Lrm^2(B;\Msn)$ such that $\sigma_k\wto\sigma$ in $\Lrm^{2}(B;\Msn)$ and $\dive\sigma_k\to\dive\sigma$ in $\Hrm^{-1}(B;\R^n)$. This shows that 
$$\int_B Q_{\dive}h(\sigma)\dd x \leq \int_B \bar h(\sigma)\dd x, \qquad \sigma\in \Lrm^2(B;\Msn),$$
and thus $Q_{\dive}h\leq \bar h$.

To complete the proof of the lower bound inequality, let us consider an arbitrary sequence $\{\sigma_k\}_{k\in \N}$ in $\Lrm^2(\Omega;\Msn)$ such that $\sigma_k\wto\sigma$ in $\Lrm^2(\Omega;\Msn)$ and $\dive\sigma_k\to\dive\sigma$ in $[\Hrm^1(\Omega;\R^n)]^*$. Then, $\sigma_k\chi_\Omega \wto\sigma\chi_\Omega$ in $\Lrm^2(B;\Msn)$ and $\dive(\sigma_k\chi_\Omega)\to\dive(\sigma\chi_\Omega)$ in $\Hrm^{-1}(B;\R^n)$, hence~\eqref{eqbarhdiv} shows that
$$\liminf_{k\to\infty}\int_B h(\sigma_k\chi_\Omega)\dd x\geq\int_B Q_{\dive}h(\sigma\chi_\Omega)\dd x.$$
Using that $h(0)=Q_{\dive}h(0)=0$ yields the lower bound
\begin{equation} \label{eq:Qh_lower}
  \liminf_{k\to\infty}\int_\Omega h(\sigma_k)\dd x\geq\int_\Omega Q_{\dive}h(\sigma)\dd x.   
\end{equation}

For the upper bound,~\eqref{eqbarhdiv} ensures the existence of a recovery sequence $\{\bar \sigma_k\}_{k\in\N}$ in $\Lrm^2(B;\Msn)$ such that $\bar \sigma_k\wto\sigma\chi_\Omega$ in $\Lrm^2(B;\Msn)$, $\dive\bar\sigma_k\to\dive(\sigma\chi_\Omega)$ in $\Hrm^{-1}(B;\R^n)$, and
$$\lim_{k\to\infty}\int_B h(\bar\sigma_k)\dd x=\int_B Q_{\dive}h(\sigma\chi_\Omega)\dd x.$$
In particular, $\bar \sigma_k\wto\sigma$ in $\Lrm^2(\Omega;\Msn)$, $\dive\bar\sigma_k\to\dive\sigma$ in $[\Hrm^1(\Omega;\R^n)]^*$, and, using again that $Q_{\dive}h(0)=0$,
$$\limsup_{k\to\infty}\int_\Omega h(\bar\sigma_k)\dd x\leq \int_\Omega Q_{\dive}h(\sigma)\dd x.$$
When combined with~\eqref{eq:Qh_lower}, we have thus shown the existence of a recovery sequence.
\end{proof}

\begin{rem} \label{rem:terrible}
Proving an analogous relaxation result with the hard constraint of a prescribed divergence, which was seemingly used implicitly in~\cite{Olbermann20}, in place of the soft constraint of a strongly $[\Hrm^1(\Omega;\R^n)]^*$-converging divergence, is not straightforward because of the lack of continuity of $h$. Indeed, the available results in that direction (see, e.g.,~\cite[Theorem~1.1]{BraidesFonsecaLeoni00},~\cite[Theorem~2.2]{AnsiniGarroni07} or~\cite[Theorem~3.3]{AnsiniDalMasoZeppieri14}) all seem to require the continuity of the integrand.

When the continuity of $h$ is missing, the lower bound remains valid, but there is an issue in proving the existence of a recovery sequence satisfying the divergence constraint \emph{exactly}. Setting $f:=-\dive \sigma$, the usual argument is to correct the recovery sequence $\{\bar\sigma_k\}_{k\in \N}$ obtained in Proposition~\ref{prop:KS} by considering the unique solution $v_k \in \Hrm^1(\Omega;\R^n)/\Rcal$ of $-\dive e(v_k) = f + \dive \bar\sigma_k$, i.e.,
$$\int_\Omega e(v_k):e(w)\dd x = \dprb{ f+\dive\bar\sigma_k,w }, \qquad w \in \Hrm^1(\Omega;\R^n)/\Rcal.$$
Since $f+\dive\bar\sigma_k \to 0$ in $[\Hrm^1(\Omega;\R^n)]^*$, we have that $e(v_k) \to 0$ in $\Lrm^2(\Omega;\Msn)$. Considering next $\hat \sigma_k :=\bar\sigma_k+e(v_k) \wto \sigma$ in $\Lrm^2(\Omega;\Msn)$, we have $-\dive\hat\sigma_k=f$ in $[\Hrm^1(\Omega;\R^n)]^*$. The problem now is to ensure that
$$\int_\Omega h(\bar \sigma_k)\dd x \quad \text{and }\quad \int_\Omega h(\hat \sigma_k)\dd x$$
have the same limit. This is the place where it seems some form of continuity property of $h$ is required, which is, however, missing in our situation.

One could also try to approximate $h$ by the continuous functions $h_M$ as in the proof of Proposition~\ref{prop:KS}. For $h_M$, we do have a relaxation result with prescribed divergence-constraint because $h_M$ is continuous (see, e.g.,~\cite[Theorem~1.1]{BraidesFonsecaLeoni00} when $f=0$). The new difficulty now is that there is a double limit to be taken as $k \to \infty$ and $M \to \infty$ and it is not clear whether the limits commute. According to~\cite{Allaire02book,AllaireBonnetierFrancfortJouve97,AllaireKohn93}, this is expected to be true, but we could not locate any reference proving precisely this result.
\end{rem}

The lack of general relaxation result for the Kohn--Strang functional with prescribed divergence necessitates an alternative approach of the limsup-inequality in~\cite{Olbermann20}. We consider the particular Kohn--Strang function
$$h_\e(\tau):=
\begin{cases}
\frac{\e}{2}|\tau|^2+\frac{1}{2\e} & \text{if }\tau \neq 0,\\
0 & \text{if } \tau=0,$$
\end{cases}$$
corresponding to $\alpha=\e/2$ and $\beta=1/(2\e)$. Note that, in that case, we have that \[
   Q_{\dive}h_\e \to \rho^\circ \quad\text{pointwise as $\e \todown 0$,}
\]
where $\rho^\circ$ is given by~\eqref{eq:WMi}.

\begin{prop}\label{prop:Olbermann}
Given a bounded $\Crm^2$-domain  $\Omega$ in $\R^n$ ($n=2,3$), let $f \in \M(\R^n;\R^n) \cap \Hrm^{-1}(\R^n;\R^n)$ be such that $\supp(f)\subset\overline\Omega$ and $\langle f,r\rangle=0$ for all $r \in \mathcal R$. For every $\lambda \in \M(\overline\Omega;\Msn)$ satisfying $-\dive\lambda=f$ in $\Dcal'(\R^n;\R^n)$, there exists a sequence $\{\lambda_\e\}_{\e>0}$ in $\Lrm^2(\Omega;\Msn)$ such that $\lambda_\e \wsto \lambda$ in $\M(\R^n;\Msn)$, $-\dive\lambda_\e \to f$ in $\Hrm^{-1}(\R^n;\R^n)$, and
$$\lim_{\e \todown 0} \int_\Omega h_\e(\lambda_\e)\dd x = \int_{\R^n} \rho^\circ \left(\frac{\di\lambda}{\di|\lambda|}\right)\dd|\lambda|.$$
\end{prop}

\begin{proof}
By Proposition~\ref{prop:W1pR*} there is an $F \in \Lrm^2(\Omega;\Msn)$ with $-\dive F = f$ in $\Dcal'(\R^n;\R^n)$, i.e.,
\[
  \dprb{ f, v } = \int_\Omega F : e(v) \dd x , \qquad  v \in \Hrm^1(\R^n;\R^n)
\]
and
\[
  \|F\|_{\Lrm^2(\Omega;\Msn)} = \|f\|_{\Hrm^{-1}(\R^n;\R^n)}.
\]
Let $\bar\lambda := \lambda - F \in \M(\overline\Omega;\Msn)$ and let $\bar\lambda^\delta\in \Crm_c(\Omega;\Msn)$ be the approximation of $\bar\lambda$ given by Proposition~\ref{prop:properies_mollification2}, for which $-\dive \bar\lambda^\delta = 0$ in $\Dcal'(\R^n;\R^n)$. Define
\[
  \lambda_\delta := \bar\lambda^\delta + F \in \Lrm^2(\Omega;\Msn),
\]
for which $-\dive \lambda_\delta = f$ in $\mathcal D'(\R^n;\R^n)$. Moreover, $\supp(\lambda_\delta) \subset \overline \Omega$, $\lambda_\delta \wsto \lambda$ in $\M(\R^n;\Msn)$ and $|\lambda_\delta|(\Omega)\to |\lambda|(\overline\Omega)$. Indeed, $|\lambda|(\overline\Omega) \leq \liminf_{\delta \todown 0} |\lambda_\delta|(\Omega)$ follows from the weak* lower semicontinuity of the total variation and
\begin{align*}
  \limsup_{\delta \todown 0} \, |\lambda_\delta|(\Omega)
  &= \limsup_{\delta \todown 0} \, |\bar{\lambda}^\delta + F\LL^n|(\Omega) \\
  &= \limsup_{\delta \todown 0} \, |\lambda^\delta - (F\LL^n)^\delta + F\LL^n|(\Omega) \\
  &\leq  \lim_{\delta \todown 0} \, |\lambda^\delta|(\Omega) 
    + \lim_{\delta \todown 0} \, \|F - F^\delta\|_{\Lrm^1(\Omega;\Msn)} \\
  &= |\lambda|(\overline\Omega)
\end{align*}
since the regularization of Proposition~\ref{prop:properies_mollification2} is linear and $\|F - F^\delta\|_{\Lrm^1(\Omega;\Msn)} \to 0$. 

As $\rho^\circ$ is positively $1$-homogeneous, the Reshetnyak continuity theorem (see~\cite[Theorem~10.3]{Rindler18book} or~\cite[Theorem~2.39]{AmbrosioFuscoPallara00book}) thus ensures that
\begin{equation}\label{eq:delta}
\lim_{\delta\to 0}\int_{\Omega} \rho^\circ(\lambda_\delta)\dd x=\int_{\overline\Omega}\rho^\circ\left(\frac{\di\lambda}{\di|\lambda|}\right)\dd|\lambda|.
\end{equation}
On the other hand, since $Q_{\dive}h_\e \to \rho^\circ$ pointwise as $\e \todown 0$, by the dominated convergence theorem we have
\begin{equation}\label{eq:epsilon}
\lim_{\e \todown 0}\int_{\Omega} Q_{\dive}h_\e(\lambda_{\delta})\dd x=\int_{\Omega} \rho^\circ(\lambda_\delta)\dd x.
\end{equation}
Gathering~\eqref{eq:delta},~\eqref{eq:epsilon} and using a diagonalization argument, we can find $\delta(\e) \todown 0$ such that $\lambda_{\delta(\e)}\wsto\lambda$ in $\M(\R^n;\Msn)$, $-\dive \lambda_{\delta(\e)} =f$ in $\Dcal'(\R^n;\R^n)$ and
$$\lim_{\e \todown 0}\int_{\Omega} Q_{\dive}h_\e(\lambda_{\delta(\e)})\dd x =\int_{\overline\Omega}\rho^\circ\left(\frac{\di\lambda}{\di|\lambda|}\right)\dd|\lambda|.$$

Applying now the relaxation result of Proposition~\ref{prop:KS} to the map $\lambda_{\delta(\e)} \in \Lrm^2(\Omega;\Msn)$ for fixed $\e>0$, we obtain a sequence $\{\lambda_\e^k\}_{k\in \N}$ in $\Lrm^2(\Omega;\Msn)$ such that
\[
\begin{cases}
    \lambda_\e^k \wto \lambda_{\delta(\e)} &\text{in $\Lrm^2(\Omega;\Msn)$ (thus also weakly* in $\M(\R^n;\Msn)$),} \\
    -\dive \lambda_\e^k \to f &\text{in $\Hrm^{-1}(\R^n;\R^n)$,}
  \end{cases}
\]
as $k \to \infty$, and
$$\lim_{k \to \infty} \int_\Omega h_\e(\lambda_\e^k)\dd x= \int_{\Omega} Q_{\dive}h_\e(\lambda_{\delta(\e)})\dd x.$$
Using again a diagonalization argument, we obtain a sequence $k(\e) \toup \infty$ such that with $\lambda_\e:=\lambda_\e^{k(\e)}$ it holds that
\[
\begin{cases}
    \lambda_\e \wsto \lambda &\text{in $\M(\R^n;\Msn)$,} \\
    - \dive \lambda_\e \to f &\text{in $\Hrm^{-1}(\R^n;\R^n)$,}
  \end{cases}
\]
and
$$\lim_{\e \todown 0} \int_\Omega h_\e(\lambda_\e)\dd x = \int_{\R^n} \rho^\circ \left(\frac{\di\lambda}{\di|\lambda|}\right)\dd|\lambda|,$$
which completes the proof of the proposition.
\end{proof}

\section{Proof of the main result} \label{sc:upper}

In this section, we prove the main result, Theorem~\ref{thm:conv-min}. Using the characterization of a minimizer $\mu^*$ of $\overline{\mathscr C}$ obtained in~\cite{BouchitteButtazzo01}, we construct a recovery sequence associated to $\mu^*$. The key idea is to consider a recovery sequence of the Lagrange-multiplier formulation associated to the measure $\lambda^*=\sigma^*\mu^*$ (where $\sigma^*$ is a solution of~\eqref{eq:barC}), taking as Lagrange multiplier $\kappa=\rho^\circ(\sigma^*)$. The crucial fact that makes our strategy work is that $\kappa$ turns out to be a \emph{constant} function (for a minimizer $\mu^*$) thanks to the results of~\cite{BouchitteButtazzo01}.

\begin{proof}[Proof of Theorem~\ref{thm:conv-min}]
Let $f\in \Lrm^2(\partial\Omega;\R^n)$ be such that $\int_{\partial\Omega}f\cdot r\dd\HH^{n-1}=0$ for all $r \in \Rcal$. We recall that $f$ is identified with an element of $\Hrm^{-1}(\R^n;\R^n) \cap \M(\R^n;\R^n)$. Without loss of generality, we can assume that $f  \not\equiv 0$; otherwise $\mathscr C_\e (\e^{-1}\LL^n\res \omega)=0$ for all $\omega \in \A_\e$ and 
$\overline{\mathscr C}(\mu)=0$ for all $\mu \in \M^1(\overline\Omega)$.

\emph{Step~1: Construction of a recovery sequence associated to minimizers of the limit compliance.}  Let us consider the following auxiliary minimization problem:
\begin{equation}\label{eq:Olb}
\kappa :=\inf_\lambda \setBB{\int_{\R^n} \rho^\circ\left(\frac{\di\lambda}{\di|\lambda|}\right)\dd|\lambda| }{ \lambda \in \M(\overline\Omega;\Msn),\; -\dive\lambda=f },
\end{equation}
where here and in the following all divergence constraints are understood in $\Dcal'(\R^n;\R^n)$.

According to Proposition~\ref{prop:W1pR*}, the set of competitors $\lambda$ in~\eqref{eq:Olb} is not empty. Applying the Direct Method of the Calculus of Variations together with Reshetnyak's lower semicontinuity theorem (see, for instance,~\cite[Theorem~2.38]{AmbrosioFuscoPallara00book}), we find that~\eqref{eq:Olb} is well-posed and that $\kappa<\infty$. Note also that $\kappa>0$, since otherwise every minimizer $\lambda^*$ of~\eqref{eq:Olb} would be identically zero, contradicting the divergence-constraint $-\dive\lambda^*=f$ since we are in the situation where $f \not\equiv 0$.

Since $\kappa<\infty$, according to~\cite[Theorem~2.3]{BouchitteButtazzo01} (applied to $\bar j$), the mass optimization problem~\eqref{eq:mass-opt}
has	a solution $\mu^*\in \M^1(\overline\Omega)$, i.e.,
\begin{equation} \label{eq:mustar_min}
  \overline{\mathscr C}(\mu^*) = \min_{\M^1(\overline\Omega)} \overline{\mathscr C},
\end{equation}
where
\begin{equation}  \label{eq:compliance}
  \overline{\mathscr C}(\mu) = \min_{\sigma} \setBB{\int_{\R^n} \frac12 \rho^\circ(\sigma)^2\dd\mu }{\sigma \in \Lrm^2(\R^n,\mu;\Msn),\; -\dive(\sigma\mu)=f }
\end{equation}
and
$$\overline{\mathscr C}(\mu^*) = \frac{\kappa^2}{2}.$$
Moreover, 
\begin{equation}\label{eq:minkappa}
\kappa=\min_{\sigma} \setBB{\int_{\R^n} \rho^\circ(\sigma)\dd\mu^* }{ \sigma \in \Lrm^1(\R^n,\mu^*;\Msn),\; -\dive(\sigma\mu^*)=f }
\end{equation}
and every solution $\sigma^*\in \Lrm^1(\R^n,\mu^*;\Msn)$ of~\eqref{eq:minkappa} satisfies
\begin{equation}\label{eq:kappa}
\rho^\circ(\sigma^*)=\kappa \quad\mu^*\text{-a.e.\ in }\R^n.
\end{equation}
As a consequence, $\sigma^*\in \Lrm^\infty(\R^n,\mu^*;\Msn)$ (see~\eqref{eq:jbar_est}) is also a minimizer of~\eqref{eq:compliance} for $\mu^*$ and $\lambda^*:=\sigma^*\mu^*$ is a solution of~\eqref{eq:Olb}.

\emph{Step~1a.} We first consider the Lagrange multiplier formulation. For each $\e>0$ and mass $m\in (0,1)$, we define the Kohn--Strang function $h^m_\e\colon \Msn \to \R$ by
$$h^m_\e(\tau):=
\begin{cases}
\dfrac{\e}{2}|\tau|^2+\dfrac{\kappa^2}{2m\e}& \text{if } \tau\neq 0,\\
0&\text{if } \tau=0.
\end{cases}$$
Its (symmetric) $\dive$-quasiconvex envelope $Q_{\dive}h^m_\e$ is explicitly given (see the beginning of Section~\ref{sc:relax}) for $n=2$ by
$$Q_{\dive}h^m_\e(\tau)=
\begin{cases}
\displaystyle  \frac{\e}{2}|\tau|^2+\frac{\kappa^2}{2m\e} &  \text{if } \rho^\circ(\tau) \geq \kappa/(\e\sqrt{m}),\\
\displaystyle  \frac{\kappa}{\sqrt{m}}\rho^\circ(\tau)-\e |\tau_1\tau_2| &  \text{if } \rho^\circ(\tau) < \kappa/(\e\sqrt{m}),
\end{cases}$$
and for $n=3$ by
$$Q_{\dive}h^m_\e(\tau)=
\begin{cases}
\displaystyle\frac{\e}{2}|\tau|^2+\frac{\kappa^2}{2m\e} &  \text{if } \rho^\circ(\tau) \geq \kappa/(\e\sqrt{m}),\\
\displaystyle\frac{\kappa}{\sqrt{m}}\rho^\circ(\tau)- \e|\tau_1\tau_2|&  \text{if } 
\begin{cases}
\rho^\circ(\tau)<\kappa/(\e\sqrt{m}),\\
|\tau_1|+|\tau_2|\leq |\tau_3|,
\end{cases}\\
\begin{aligned} &{\textstyle \frac{\kappa}{\sqrt{m}}}\rho^\circ(\tau) \\[-2pt] &\; + {\textstyle \frac{\e}{2}} \left({\textstyle \frac12|} \tau|^2-|\tau_1\tau_2|-|\tau_1\tau_3|-|\tau_2\tau_3|\right) \end{aligned} \hspace{-4pt} & \text{if }
\begin{cases}
\rho^\circ(\tau)<\kappa/(\e\sqrt{m}),\\
|\tau_1|+|\tau_2|> |\tau_3|.
\end{cases}
\end{cases}$$
According to Proposition~\ref{prop:Olbermann} (with $\e' := \e\sqrt{m}/\kappa$), there exists a family of maps $\{\lambda_\e^{m}\}_{\e>0} \subset \Lrm^2(\Omega;\Msn)$ such that $\lambda^m_\e \wsto \lambda^*$ in $\M(\R^n;\Msn)$ as $\e \todown 0$, 
\begin{equation}\label{eq:diveps}
\lim_{\e \todown 0}\|-\dive\lambda^m_\e-f\|_{\Hrm^{-1}(\R^n;\R^n)}=0,
\end{equation}
and
\begin{equation} \label{eq:hme_conv}
\lim_{\e \todown 0} \int_{\Omega} h^m_\e(\lambda^m_\e)\dd x=\frac{\kappa}{\sqrt{m}}\int_{\overline\Omega}\rho^\circ\left(\frac{\di\lambda^*}{\di|\lambda^*|}\right)\dd|\lambda^*|=\frac{\kappa}{\sqrt{m}}\int_{\overline\Omega}\rho^\circ(\sigma^*)\dd\mu^*,
\end{equation}
where we used the positive one-homogeneity of $\rho^\circ$.

\emph{Step~1b.} Next, we construct a measure $\tilde \mu^m_\e\in \M^+(\overline\Omega)$ and a map $\tilde\sigma^m_\e \in \Lrm^2(\R^n,\tilde\mu^m_\e;\Msn)$ which satisfy the right upper bound inequality, but fail to fulfil the mass and divergence constraints; this will be rectified in the next step. To this aim, let us define
\begin{align*}
  \tilde\omega_\e^{m} &:=\setb{x \in \Omega }{ \lambda^m_\e(x) \neq 0 }, \\
  \tilde\mu_\e^{m} &:=\frac{1}{\e}\LL^n\res\tilde\omega^m_\e\in \M^+(\R^n), \\
  \tilde \sigma^{m}_\e &:=\e \lambda^m_\e \in \Lrm^2(\R^n,\tilde \mu_\e;\Msn),
\end{align*} 
which satisfy
\begin{equation}\label{eq:1}
\int_\Omega \frac12  \left(|\tilde \sigma^m_\e|^2 + \frac{\kappa^2}{m}\right)\dd\tilde\mu^m_\e= \int_\Omega h^m_\e(\lambda^m_\e)\dd x \leq C.
\end{equation}
Clearly,
\begin{equation} \label{eq:mume_conv}
  \tilde\sigma^m_\e\tilde\mu^m_\e=\lambda^m_\e \wsto \lambda^* \quad\text{ in $\M(\R^n;\Msn)$}
\end{equation}
as $\e\to 0$. Moreover, using that $\kappa^2/m>0$, the previous energy bound additionally ensures that the family $\{\tilde \mu^m_\e\}_{\e>0}$ is bounded in $\M(\R^n)$. Up to a subsequence, we have 
\begin{equation}\label{eqconvm}
\tilde\mu^m_\e\wsto \tilde\mu^m \text{ in } \M(\R^n) 
\end{equation}
as $\e \todown 0$ for some $\tilde\mu^m \in \M^+(\overline\Omega)$. Arguing as in the proof of compactness in Proposition~\ref{prop:compactness}, we obtain that $\lambda^*$ is absolutely continuous with respect to $\tilde \mu^m$ and
\[
  \lambda^*= \tilde \sigma^m\tilde\mu^m
\]
for some $\tilde\sigma^m \in \Lrm^2(\R^n,\tilde\mu^m;\Msn)$. The measure $\tilde \mu^m$ can be decomposed as
\begin{equation} \label{eq:tildemu_decomp}
  \tilde \mu^m=\theta^m\mu^* + \tilde\nu^m,
\end{equation}
where $\theta^m\in \Lrm^1(\R^n,\mu^*)$ and $\tilde\nu^m \in \M^+(\R^n)$ is singular with respect to $\mu^*$. Thus,
$$\sigma^*\mu^*=\lambda^*=\tilde\sigma^m\tilde\mu^m= \tilde\sigma^m \theta^m\mu^* + \tilde\sigma^m\tilde \nu^m.$$
Since $\tilde\sigma^m\tilde\nu^m$ is singular with respect to $\mu^*$, we conclude $\tilde\sigma^m\tilde \nu^m=0$. As a consequence, 
\begin{equation}\label{eq:tildesigma}
\sigma^*=\tilde\sigma^m\theta^m \quad\mu^*\text{-a.e.\ in }\R^n.
\end{equation}

Passing to the upper limit as $\e \todown 0$ in~\eqref{eq:1} using~\eqref{eq:hme_conv},~\eqref{eq:tildemu_decomp},
\begin{align}
&\frac{\kappa^2}{2m}\tilde\nu^m(\R^n) +\frac{\kappa^2}{2m}\int_{\R^n}\theta^m\dd
\mu^*+ \limsup_{\e \todown 0} \int_\Omega \frac12  |\tilde\sigma^m_\e|^2\dd\tilde\mu^m_\e \notag\\
&\qquad \leq \frac{\kappa}{\sqrt{m}} \int_{\R^n} \rho^\circ(\sigma^*) \dd\mu^*. \label{eq:2}
\end{align}
On the other hand, we can pass to the lower bound via Proposition~\ref{prop:liminf} and Remark~\ref{rem:liminf}, whose assumptions are satisfied by~\eqref{eq:diveps},~\eqref{eq:mume_conv}, and ~\eqref{eqconvm}. This gives
\begin{align}
\liminf_{\e \todown 0}\int_\Omega \frac12  |\tilde\sigma^m_\e|^2\dd\tilde\mu^m_\e &\geq \int_{\R^n} \frac12  \rho^\circ(\tilde\sigma^m)^2 \dd\tilde\mu^m \notag \\
&\geq\int_{\{\theta^m>0\}} \frac{\rho^\circ(\sigma^*)^2}{2\theta^m}\dd\mu^*,  \label{eq:3}
\end{align}
where for the second inequality we employed~\eqref{eq:tildemu_decomp},~\eqref{eq:tildesigma}. Combining~\eqref{eq:2} and~\eqref{eq:3}, we arrive at
$$\frac{\kappa^2}{2m}\tilde\nu^m(\R^n) +\int_{\{\theta^m>0\}}\frac{1}{2\theta^m} \left(\frac{\kappa\theta^m}{\sqrt{m}} -\rho^\circ(\sigma^*)\right)^2\dd\mu^* \leq 0.$$
Thus, $\tilde\nu^m=0$ and, remembering~\eqref{eq:kappa}, $\theta^m=\sqrt{m}$ $\mu^*$-a.e.\ in $\{\theta^m>0\}$. Using~\eqref{eq:tildesigma} together with~\eqref{eq:kappa} again, leads to $\rho^\circ(\tilde\sigma^m)\theta^m=\rho^\circ(\sigma^*)=\kappa>0$ $\mu^*$-a.e.\ in $\R^n$, so that $\mu^*(\{\theta^m=0\})=0$. Summing up, we have shown
$$  \theta^m=\sqrt{m} \quad\text{$\mu^*$-a.e.,} \qquad
  \tilde \mu^m=\sqrt{m} \mu^*,  \qquad \tilde\nu^m=0.$$
Inserting this into~\eqref{eq:2} and using~\eqref{eq:kappa} once more, leads to
\begin{equation}\label{eq:me}
\limsup_{\e \todown 0}\int_\Omega \frac12  |\tilde\sigma^m_\e|^2\dd\tilde\mu^m_\e \leq \int_{\R^n} \frac{1}{2\sqrt{m}} \rho^\circ(\sigma^*)^2 \dd\mu^*.
\end{equation}

Since by the uniqueness of the weak* limit, there is no need of extracting a subsequence, we have thus proved that $\tilde\mu^m_\e \wsto \sqrt{m}\mu^*$ in $\M(\R^n)$ as $\e \todown 0$ and that~\eqref{eq:me} holds. Moreover, since all these measures are positive and have uniformly bounded support, we get $\tilde\mu^m_\e(\R^n)\to \sqrt{m}=\sqrt{m}\mu^*(\R^n)$.

\emph{Step~1c.}
We now modify $\tilde\omega^m_\e$ to make it admissible in the class $\A_\e$ defined in~\eqref{eq:Ae}. We first observe that since $m \in (0,1)$, there exists $\e_m>0$ such that for $\e<\e_m$,
$$\frac{\LL^n(\tilde\omega^m_\e)}{\e}\leq \frac{\sqrt{m}+1}{2}<1.$$
Let
$$E^m_\e:=\tilde\omega^m_\e \cup \partial\Omega,$$
which is a Lebesgue-measurable set satisfying $\LL^n(E^m_\e)<\e$.

Let $\mathcal B$ denote the collection of open balls in $\R^n$. It is well-known that $\mathcal B$ generates the Borel $\sigma$-algebra on $\R^n$. We may find a countable collection of open balls $B_k \in \mathcal B$, $k \in \N$, such that
\begin{equation} \label{eq:Q1}
E^m_\e \subset  \bigcup_{k=1}^\infty B_k
\end{equation}
and
\begin{equation} \label{eq:Q2}
  \LL^n \biggl( \bigcup_{k=1}^\infty B_k \biggr) < \e,   \qquad
  \LL^n \biggl( \bigcup_{k=1}^\infty B_k \setminus E^m_\e \biggr) < \frac{\e^2}{2}.
\end{equation}
Then choose $N = N(m,\e)$ large enough such that for
\[
  U^m_\e :=  \bigcup_{k=1}^N  B_k
\]
the following properties hold:
\begin{enumerate}[a)]
\item $\partial\Omega \subset U_\e^m$ (this is possible since $\{B_k\}_{k \in \mathbb N}$ is in particular an open covering of the compact set $\partial\Omega$, from which we can extract a finite subcovering);
  \item $\LL^n(U^m_\e) < \e$ (by~\eqref{eq:Q2});
  \item $\LL^n(E^m_\e \Delta U^m_\e) < \e^2$ and $\LL^n\left(\bigcup_k B_k \setminus U_\e^m \right)<\e^2$ (by~\eqref{eq:Q1},~\eqref{eq:Q2});
  \item $\displaystyle \normb{\lambda^m_\e \restrict U^m_\e - \lambda^m_\e}_{\Lrm^2(\R^n;\Msn)} < \e$ (since $\lambda^m_\e$ is concentrated in $E^m_\e$ and, by the dominated convergence theorem, $\int_{\bigcup_{k> N}  B_k} \abs{\lambda^m_\e}^2 \dd x \to 0$ as $N \to \infty$);
  \item $\displaystyle \absBB{\int_{U^m_\e} \abs{\tilde\sigma^m_\e}^2 \; \frac{\di x}{\e} - \int_{\R^n}\abs{\tilde\sigma^m_\e}^2 \dd \tilde\mu^m_\e} < \e$ (by a similar argument).
\end{enumerate}
Note that $U^m_\e$ might fail to have a  Lipschitz boundary because tangential balls have an intersection point generating a cusp.  However, there are only finitely many such singular points, so that we can add to the $\{B_k\}_{1 \leq k \leq N}$ finitely many small balls centered at these points with arbitrarily small measure. In that way, we can further assume without loss of generality that the set $U_\e^m$ has a Lipschitz boundary.

Next, we find an open set $V^m_\e \subset \Omega$ such that
\[
  \omega^m_\e:=(\Omega \cap U^m_\e) \cup V^m_\e
\]
lies in $\A_\e$, that is, $\omega^m_\e$ is a connected open set with Lipschitz boundary and
\[
  \partial\Omega \subset\partial\omega^m_\e,  \qquad
  \LL^n(\omega^m_\e)=\e.
\]
Indeed, we can construct $V^m_\e$ as the union of finitely many cylindrical ``struts'' to make the set connected and with Lipschitz boundary. This construction can be achieved with arbitrarily small added volume and by~b) there is a gap between $\LL^n(U^m_\e)$ and the target volume $\e$. The requirement that $\partial\Omega \subset\partial\omega^m_\e$ holds by construction since $\partial\Omega \subset U^m_\e$. We can add additional mass to ensure the final condition $\LL^n(\omega^m_\e)=\e$. Hence, $\omega^m_\e \in \A_\e$ is indeed an admissible shape.

Define
\[
  \mu^m_\e:=\frac{1}{\e}\LL^n\res\omega^m_\e, \qquad \sigma^m_\e:=\tilde\sigma^m_\e \restrict U^m_\e =\e \lambda_\e^m \restrict U_\e^m.
\]
Since $\{\mu^m_\e\}_{\e>0}$ is a family of probability measures supported in $\overline \Omega$, up to a subsequence, $\mu^m_\e \wsto \mu^m$ in $\M(\R^n)$ as $\e\to 0$ for some $\mu^m \in \M^1(\overline\Omega)$. We observe that for all $\phi \in \Crm_0(\R^n)$ with $\phi \geq 0$ it holds that
\begin{align*}
  &\int_{\R^n} \phi \dd (\mu^m_\e - \tilde\mu^m_\e) \\
  &\qquad = \int_{U_\e^m}\phi \dd (\mu^m_\e - \tilde\mu^m_\e) \\
  &\qquad\qquad + \int_{\bigcup_k B_k \setminus U_\e^m} \phi \dd (\mu^m_\e - \tilde\mu^m_\e) + \int_{\R^n \setminus \bigcup_k B_k} \phi \dd (\mu^m_\e - \tilde\mu^m_\e)\\
  &\qquad \geq  \frac{1}{\e}\int_{U_\e^m} (1-\chi_{E_\e^m})\phi\dd x -\frac{2}{\e} \norm{\phi}_{\Lrm^\infty(\R^n)}\LL^n\left(\bigcup_k B_k  \setminus U_\e^m\right) \\
  &\qquad\qquad + \int_{\omega_\e^m \setminus \bigcup_k B_k} \frac1\e\phi \dd x \\
  &\qquad \geq - 2\e \norm{\phi}_{\Lrm^\infty(\R^n)}
\end{align*}
by~c). Thus, since $\tilde\mu^m_\e \wsto \sqrt{m}\mu^*$, we have $\mu^m\geq \sqrt{m}\mu^*$.

Next, up to a subsequence, $\mu^m \wsto \hat \mu$ in $\M(\R^n)$ as $m\to 1$ for some $\hat \mu \in \M^1(\overline\Omega)$ with $\hat\mu\geq\mu^*$. Since $\mu^*$ and $\hat\mu$ are both probability measures, we deduce that $\hat\mu=\mu^*$. We further claim that
\begin{align}
& \displaystyle \lim_{m \to 1}\lim_{\e \todown 0}\int_{\R^n}\varphi\dd\mu^m_\e =\int_{\R^n}\varphi\dd\mu^* \quad\text{for all } \varphi \in \Crm_0(\R^n),  \label{eq:R1} \\
  & \displaystyle\lim_{m \to 1}\lim_{\e \todown 0}\|-\dive(\sigma^m_\e\mu^m_\e) - f\|_{\Hrm^{-1}(\R^n;\R^n)}=0, \label{eq:R2} \\
& \displaystyle\limsup_{m \to 1}\limsup_{\e \todown 0}\int_\Omega \frac12  |\sigma^{m}_\e|^2\dd\mu^{m}_\e  \leq \int_{\R^n} \frac12 \rho^\circ(\sigma^*)^2\dd\mu^*.  \label{eq:R3}
\end{align}
The first condition,~\eqref{eq:R1}, follows directly from the convergences above. For~\eqref{eq:R2}, we estimate for any $\psi \in \Hrm^1(\R^n;\R^n)$ with $\norm{\psi}_{\Hrm^1} \leq 1$ as follows:
\begin{eqnarray*}
  \dprb{-\dive (\sigma^m_\e\mu_\e^m) - f, \psi} & = & \dprb{-\dive (\lambda^m_\e \restrict U^m_\e) - f, \psi}\\
& = & \int_\Omega \lambda^m_\e : \nabla \psi \dd x - \dprb{f,\psi} - \int_{\Omega \setminus U^m_\e} \lambda^m_\e : \nabla \psi \dd x.
\end{eqnarray*}
Thus,
\begin{align*}
  &\sup_{\norm{\psi}_{\Hrm^1} \leq 1} \dprb{-\dive (\sigma^m_\e \mu^m_\e) - f, \psi} \\
  &\qquad \leq \sup_{\norm{\psi}_{\Hrm^1} \leq 1} \dprb{-\dive \lambda^m_\e - f, \psi} + \sup_{\norm{\psi}_{\Hrm^1} \leq 1} \absBB{\int_{\Omega \setminus U^m_\e} \lambda^m_\e \cdot \nabla \psi \dd x} \\
  &\qquad \leq \normb{-\dive \lambda^m_\e - f}_{\Hrm^{-1}(\R^n;\R^n)} + \normb{\lambda^m_\e \restrict U^m_\e - \lambda^m_\e}_{\Lrm^2(\R^n;\Msn)}
\end{align*}
and this converges to zero as $\e \todown 0$ by~\eqref{eq:diveps} and~d). Finally, for~\eqref{eq:R3}, we observe via~e) and~\eqref{eq:me},
\begin{align*}
  \limsup_{m \to 1}\limsup_{\e \todown 0} \int_\Omega \frac12 |\sigma^{m}_\e|^2\dd\mu^{m}_\e
  &= \limsup_{m \to 1}\limsup_{\e \todown 0}\int_{U^m_\e} \frac12  |\tilde\sigma^{m}_\e|^2\dd\tilde\mu^{m}_\e \\
  &\leq \limsup_{m \to 1}\limsup_{\e \todown 0}\int_\Omega \frac12  |\tilde\sigma^{m}_\e|^2\dd\tilde\mu^{m}_\e \\
  &\leq \int_{\R^n} \frac12 \rho^\circ(\sigma^*)^2\dd\mu^*.
\end{align*}

Since $\M(\R^n)$ is the dual of the separable space $\Crm_0(\R^n)$, we can apply a diagonalization argument to show the existence of $m(\e) \toup 1$ such that for
\begin{align*}
  \omega^*_\e &:=\omega_\e^{m(\e)} \in \A_\e, \\
  \mu^*_\e &:=\mu_\e^{m(\e)} =\frac{1}{\e}\LL^n\res\omega^*_\e \in \M^+(\R^n), \\
  \hat \sigma_\e &:=\sigma_\e^{m(\e)} \in \Lrm^2(\R^n,\mu^*_\e;\R^n),
\end{align*}
we have
\[
\begin{cases}
  \mu^*_\e \wsto \mu^*  &\text{in } \M(\R^n), \\
  -\dive(\hat \sigma_\e\mu^*_\e) \to f  &\text{in }\Hrm^{-1}(\R^n;\R^n),
  \end{cases}
\]
and
\begin{equation}\label{eq:C2}
\limsup_{\e \todown 0}\int_\Omega \frac12  |\hat \sigma_\e|^2\dd\mu^*_\e \leq \int_{\R^n} \frac12  \rho^\circ(\sigma^*)^2 \dd\mu^*=\overline{\mathscr C}(\mu^*).
\end{equation}

\emph{Step~1d.} It finally remains to modify the stress $\hat \sigma_\e$ in order to satisfy the hard divergence-constraint. We first notice that $f+\dive(\hat \sigma_\e\mu^*_\e) \in \Hrm^{-1}(\R^n;\R^n)$ has support in $\overline{\omega^*_\e}$ (because $\supp(f)\subset \partial\Omega \subset\partial\omega^*_\e$) and $\langle f+\dive(\hat\sigma_\e\mu^*_\e),r\rangle=0$ for all $r \in \Rcal$. Since $\omega^*_\e$ is a bounded Lipschitz domain, Proposition~\ref{prop:W1pR*} ensures the existence of $F_\e \in \Lrm^2(\omega^*_\e;\Msn)$ such that 
$$\dprb{ f+\dive(\hat \sigma_\e\mu^*_\e),z }=\int_{\omega^*_\e} F_\e:e(z)\dd x, \qquad z \in \Hrm^1(\R^n;\R^n)$$
and
$$\|F_\e\|_{\Lrm^2(\omega^*_\e;\Msn)}=\|f+\dive(\hat\sigma_\e\mu^*_\e)\|_{\Hrm^{-1}(\R^n;\R^n)}.$$
We now set
\[
  \sigma^*_\e:=\hat\sigma_\e+ \e F_\e\chi_{\omega^*_\e} \in \Lrm^2(\R^n,\mu^*_\e;\Msn).
\]
By construction we have 
$$-\dive(\sigma^*_\e\mu^*_\e)=-\dive(\hat\sigma_\e\mu^*_\e)-\dive(F_\e\chi_{\omega^*_\e})=f \qquad \text{in }\Dcal'(\R^n;\R^n).$$
Also,
$$\int_{\R^n} |\e F_\e \chi_{\omega^*_\e}|^2\dd\mu^*_\e=\e \int_{\omega^*_\e} |F_\e|^2\dd x \to 0$$
and hence, using~\eqref{eq:C2},
\begin{equation} \label{eq:C3}
  \limsup_{\e \todown 0} \int_\Omega \frac12 |\sigma^*_\e|^2\dd\mu^*_\e = \limsup_{\e \todown 0}\int_\Omega \frac12 |\hat \sigma_\e|^2\dd\mu^*_\e \leq\overline{\mathscr C}(\mu^*).
\end{equation}

\emph{Step~1e.} We are now in position to conclude the upper bound. Indeed, we have
\[
  \inf_{\M^1(\overline\Omega)} \mathscr C_\e \leq \mathscr C_\e(\mu^*_\e),
\]
and passing to the limit as $\e \todown 0$ using~\eqref{eq:C3}, we get
\begin{equation} \label{eq:Ce_lower}
  \limsup_{\e \todown 0}\inf_{\M^1(\overline\Omega)} \mathscr C_\e \leq \limsup_{\e \todown 0}\mathscr C_\e(\mu^*_\e) \leq \overline{\mathscr C}(\mu^*).
\end{equation}
According to the lower bound established in Proposition~\ref{prop:liminf}, we also have that
$$\overline{\mathscr C}(\mu^*) \leq \liminf_{\e \todown 0}\mathscr C_\e(\mu^*_\e),$$
hence, recalling~\eqref{eq:mustar_min},
\begin{equation}\label{eq:recovery}
\lim_{\e \todown 0}\mathscr C_\e(\mu^*_\e)=\overline{\mathscr C}(\mu^*)= \min_{\M^1(\overline\Omega)}\overline{\mathscr C}.
\end{equation}

\emph{Step~2: Compactness and lower bound for the compliance.} 
We first notice that, by~\eqref{eq:Ce_lower} and~\eqref{eq:recovery}, the infimal value
$$\inf_{\M^1(\overline\Omega)} \mathscr C_\e$$
is finite, and uniformly bounded with respect to $\e$. 
Let $\{\alpha_\e\}_{\e>0}$ be such that $\alpha_\e \todown 0$ and for $\e>0$ assume we are given $\omega_\e \in \A_\e$ with
\[
  \mathscr C_\e \biggl( \frac{\LL^n\res \omega_\e}{\e} \biggr) \leq \inf_{\omega\in\A_\e} \mathscr C_\e \biggl( \frac{\LL^n\res \omega}{\e} \biggr)  + \alpha_\e.
\]
Let us define the probability measures $\bar \mu_\e :=\frac{1}{\e}\LL^n\res\omega_\e$, which thus satisfy
\[
\mathscr C_\e(\bar\mu_\e) \leq \inf_{\M^1(\overline\Omega)} \mathscr C_\e + \alpha_\e \leq \mathscr C_\e(\mu_\e^*) + \alpha_\e.
\]
Extract a subsequence $\{\e_k\}_{k\in \N}$ with $\e_k \todown 0$ such that $\bar\mu_{\e_k}\wsto \bar\mu$ in $\M(\R^n)$ for some $\bar\mu \in \M^1(\overline\Omega)$ and
$$\lim_{k\to \infty}\mathscr C_{\e_k}(\bar\mu_{\e_k})=\liminf_{\e \todown 0}\mathscr C_\e(\bar\mu_\e).$$
Using  the lower bound inequality established in Proposition~\ref{prop:liminf}, we infer that
\begin{equation}\label{eq:Ce_lower2}
\lim_{\e \todown 0}\mathscr C_\e(\mu^*_\e)\geq \liminf_{\e \todown 0}\inf_{\M^1(\overline\Omega)}{\mathscr C}_\e \geq\liminf_{\e \todown 0}\mathscr C_\e(\bar\mu_\e)=\lim_{k\to \infty}\mathscr C_{\e_k}(\bar\mu_{\e_k}) \geq  \overline{\mathscr C}(\bar\mu).
\end{equation}

\emph{Step~3: Proof of Theorem~\ref{thm:conv-min}.}
Combining the information of~\eqref{eq:recovery} with~\eqref{eq:Ce_lower2}  yields
\[
  \overline{\mathscr C}(\bar\mu)\leq\overline{\mathscr C}(\mu^*)
  = \min_{\mu \in \M^1(\overline\Omega)} \overline{\mathscr C}(\mu),
\]
which shows that $\bar\mu$ is a minimizer of~\eqref{eq:mass-opt}. Next,~\eqref{eq:Ce_lower},~\eqref{eq:recovery} and~\eqref{eq:Ce_lower2} 
together show that
$$\min_{\M^1(\overline\Omega)} \overline{\mathscr C} =\overline{\mathscr C}(\bar\mu)=\overline{\mathscr C}(\mu^*)=\lim_{\e \todown 0} \mathscr C_\e(\mu_\e^*)=\lim_{k\to \infty}\mathscr C_{\e_k}(\bar\mu_{\e_k})=\lim_{\e \todown 0}\inf_{\M^1(\overline\Omega)} \mathscr C_\e,$$
which completes the proof of Theorem~\ref{thm:conv-min}.
\end{proof}

\appendix

\section{Laminations leading to the infinitesimal-mass integrand}\label{ax:ex}

For the purpose of illustration, we now give a typical shape that is optimal for the elastic compliance in the vanishing-mass limit. This construction corresponds to the Allaire--Kohn laminate of order $2$ that was given in~\cite{AllaireKohn93} as the optimal microstructure for the Hashin--Shtrikman lower bound. It does not rely on a separation of scales, but rather on the linear superposition of single scale laminates with different proportions; see also~\cite{BourdinKohn08} for a construction in the high-porosity regime. This construction in particular motivates the precise form of $\bar j^*$ in~\eqref{eq:jbs_elast}.

We start with the two-dimensional case.

\begin{example}
Set $\Omega:=Q=(0,1)^2$ and define for $\alpha_1,\alpha_2 \in \R$ with $|\alpha_1|\leq|\alpha_2|$,
\[
\mu:=\LL^2\res Q, \qquad
\sigma:=
\begin{pmatrix}
	\alpha_1 & 0 \\
	0 & \alpha_2
\end{pmatrix}.
\]
In particular, $\dive(\sigma\mu)=0$ in $Q$.

Let $\gamma\in(0,1)$, to be chosen later depending on $\alpha_1$ and $\alpha_2$. Let $k(\e) \in \N$ be such that
\[
  \lim_{\e \to 0} \frac{\e}{k(\e)} = 0.
\]
For $k \in \N$ we define the sets
\begin{align*}
\displaystyle D_\e^1 &:= \bigcup_{i=0}^{k(\e)-1} \Bigl( \frac{i}{k(\e)}, \frac{i + \gamma \e}{k(\e)} \Bigr) \times (0,1),  \\
\displaystyle D_\e^2 &:=(0,1)\times\bigcup_{i=0}^{k(\e)-1} \Bigl( \frac{i}{k(\e)}, \frac{i + (1-\gamma) \e}{k(\e)} \Bigr),\\
D_\e &:=D_\e^1\cup D_\e^2,	
\end{align*}
so that 
\[
  \LL^2(D_\e^1)=\gamma\e,\qquad \LL^2(D_\e^2)=(1-\gamma)\e,\qquad \LL^2(D_\e)=\e - \gamma(1-\gamma)\e^2.
\]
We set
\[
  \mu_\e:=\frac{1}{\e}\LL^2\res D_\e,\qquad
\sigma_\e:=
\begin{pmatrix}\displaystyle
	\frac{\alpha_1}{1-\gamma}\chi_{D^2_\e} & 0 \\
	0 & \dfrac{\alpha_2}{\gamma}\chi_{D_\e^1} 
\end{pmatrix}.
\]
Then, since $\chi_{D_\e^1}$ depends only on $x_1$ and $\chi_{D_\e^2}$ depends only on $x_2$, we have that
\[
  \dive(\sigma_\e\mu_\e)=0.
\]
The mass constraint is satisfied up to a perturbation $\SmallO(\e)$ because $\mu_\e(Q)=\LL^2(D_\e)=\e+\SmallO(\e)$. It is nevertheless possible to adjust the definitions of the sets $D_\e^1$ and $D_\e^2$ in order to have exactly $\LL^2(D_\e)=\e$ in place of $\LL^2(D_\e)=\e+\SmallO(\e)$, and also $\partial Q\subset D_\e$. However, we prefer to keep the current definitions to avoid additional technicalities.

It is easy to see that
\[
\mu_\e\wsto\mu\quad\text{ in }\M(\R^2),\qquad \sigma_\e\mu_\e \wsto \sigma\mu \quad \text{ in }\M(\R^2;\Mstwo)
\]
and that the energy can be computed as
\begin{align*}
\mathscr E_\e(\sigma_\e,\mu_\e) &= \int_{D_\e}\frac{|\sigma_\e|^2}{2\e} \dd x \\
&= \frac{1}{2\e} \biggl( \frac{\alpha_{1}^2}{(1-\gamma)^2} \LL^2(D_\e^2) +
\frac{\alpha_2^2}{\gamma^2} \LL^2(D_\e^1) \biggr) \\
&= \frac12\left(\frac{\alpha_1^2}{1-\gamma}+\frac{\alpha_2^2}{\gamma}\right).
\end{align*}
Minimizing this expression with respect to $\gamma\in (0,1)$, we find that the optimal $\gamma$ is
$$\gamma=\frac{|\alpha_2|}{|\alpha_1|+|\alpha_2|}$$
if $|\alpha_1|\neq 0$, while $\gamma:=1-\e$ if $|\alpha_1|= 0$,
and then, for this choice of $\gamma$,
$$\lim_{\e\to0} \mathscr E_\e(\sigma_\e,\mu_\e)=\frac12(|\alpha_1|+|\alpha_2|)^2 = \bar j^*(\sigma) = \int_Q \bar j^*(\sigma) \dd \mu = \overline{\mathscr E}(\sigma,\mu).$$
This shows that $(\sigma_\e,\mu_\e)$ is a recovery sequence for $(\sigma,\mu)$ and indeed the integrand $j^*(\tau) = \frac12 |\tau|^2$ relaxes to $\bar j^*(\tau)$ in the vanishing-mass limit, at least in the present situation of a constant stress $\sigma$.
\end{example}

In the three-dimensional case, we need to distinguish two cases.

\begin{example}
Set $\Omega:=Q=(0,1)^3$ and define for $\alpha_1,\alpha_2,\alpha_3 \in \R$ with $0 < |\alpha_1|\leq|\alpha_2|\leq|\alpha_3|$ (if $\alpha_1 = 0$, we are again in the two-dimensional case),
\[
\mu:=\LL^3\res Q, \qquad
\sigma:= \begin{pmatrix}
    \alpha_1 & 0 & 0\\
    0 & \alpha_2 & 0\\
    0 & 0 & \alpha_3
  \end{pmatrix}.
\]
In particular, $\dive(\sigma\mu)=0$ in $Q$. 

\medskip

\textit{Case I.} Assume that
\[
  \abs{\alpha_3} \geq \abs{\alpha_1} + \abs{\alpha_2}.
\]
Let
\begin{equation} \label{eq:gamma}
  \gamma := \frac{\abs{\alpha_2}}{\abs{\alpha_1} + \abs{\alpha_2}}  \in (0,1).
\end{equation}
For $\e > 0$ small and $k \in \N$ we define, with $k(\e)$ defined as in the previous example,
\begin{align*}
  D^1_\e &:=  \bigcup_{i=0}^{k(\e)-1} \Bigl( \frac{i}{k(\e)}, \frac{i + \gamma \e}{k(\e)} \Bigr) \times (0,1)^2 , \\
  D^2_\e &:=   \bigcup_{i=0}^{k(\e)-1} \; (0,1) \times \Bigl( \frac{i}{k(\e)}, \frac{i + (1-\gamma) \e}{k(\e)} \Bigr) \times (0,1),\\
  D_\e &:=D_\e^1\cup D_\e^2.
\end{align*}
Then,
\[
  \LL^3(D_\e^1)=\gamma\e,\qquad \LL^3(D_\e^2)=(1-\gamma)\e,\qquad \LL^3(D_\e)=\e-\gamma(1-\gamma)\e^2.
\]
We set
\[
  \mu_\e := \frac{1}{\e}\LL^3\res D_\e, \qquad
  \sigma_\e := \begin{pmatrix}
    \dfrac{\alpha_1}{1-\gamma} \chi_{D^2_\e} & 0 & 0\\
   0 & \dfrac{\alpha_2}{\gamma} \chi_{D^1_\e} & 0\\
    0 & 0 & \alpha_3 \chi_{D^1_\e \cup D^2_\e}
  \end{pmatrix},
\]
which satisfy
\[
\mu_\e\wsto\mu\quad\text{ in }\M(\R^3),\qquad \sigma_\e\mu_\e \wsto \sigma\mu \quad \text{ in }\M(\R^3;\mathbb M^{3 \times 3}_{\rm sym})
\]
and
\[
  \dive(\sigma_{\e}\mu_{\e})=0.
\]
Now calculate for the energy
\begin{align*}
  \mathscr E_\e(\sigma_\e,\mu_\e) &=\int_{D_\e} \frac12 \abs{\sigma_\e}^2 \; \frac{\di x}{\e}\\
  &= \frac{1}{2\e} \biggl( \frac{\alpha_1^2}{(1-\gamma)^2}  \LL^3(D^2_\e) + \frac{\alpha_2^2}{\gamma^2}\LL^3(D^1_\e)  + \alpha_3^2  \LL^3(D^1_\e \cup D^2_\e) \biggr) \\
  &=\frac12 \biggl( \frac{\alpha_1^2}{1-\gamma} + \frac{\alpha_2^2}{\gamma} + \alpha_3^2 \biggr) + \BigO(\e)  \\
  &=\frac12 \bigl( (\abs{\alpha_1} + \abs{\alpha_2})^2 + \abs{\alpha_3}^2 \bigr) + \BigO(\e)  \\
  &= \bar{j}^*(\sigma) + \BigO(\e) \\
  &= \overline{\mathscr E}(\sigma,\mu) + \BigO(\e).
\end{align*}
This justifies the form of $\bar{j}^*$ in this case.

\medskip

\textit{Case II.} We now assume
\begin{equation} \label{eq:ex3cond}
  \abs{\alpha_3} < \abs{\alpha_1} + \abs{\alpha_2}.
\end{equation}
Let
\begin{align*}
  \gamma_1 &:= \frac{\abs{\alpha_2}+\abs{\alpha_3}-\abs{\alpha_1}}{\abs{\alpha_1} + \abs{\alpha_2} + \abs{\alpha_3}}  \in (0,1), \\
  \gamma_2 &:= \frac{\abs{\alpha_1}+\abs{\alpha_3}-\abs{\alpha_2}}{\abs{\alpha_1} + \abs{\alpha_2} + \abs{\alpha_3}}  \in (0,1), \\
  \gamma_3 &:= 1-\gamma_1-\gamma_2=\frac{\abs{\alpha_1}+\abs{\alpha_2}-\abs{\alpha_3}}{\abs{\alpha_1} + \abs{\alpha_2} + \abs{\alpha_3}}  \in (0,1)
\end{align*}
and
\begin{align*}
  D^1_\e &:= \bigcup_{i=0}^{k(\e)-1}\Bigl( \frac{i}{k(\e)}, \frac{i + \gamma_1 \e}{k(\e)} \Bigr) \times (0,1)^2, \\
  D^2_\e &:= \bigcup_{i=0}^{k(\e)-1}\; (0,1) \times \Bigl( \frac{i}{k(\e)}, \frac{i + \gamma_2 \e}{k(\e)} \Bigr) \times (0,1) , \\
  D^3_\e &:= \bigcup_{i=0}^{k(\e)-1}\; (0,1)^2 \times \Bigl( \frac{i}{k(\e)}, \frac{i + \gamma_3 \e}{k(\e)} \Bigr), \\
  D_\e  &:= D^1_\e \cup D^2_\e \cup D^3_\e.
\end{align*}
One computes, for $1 \leq i \neq j \leq 3$,
$$\LL^3(D^i_\e) = \gamma_i \e, \quad
  \LL^3(D^i_\e \cap D^j_\e) = \gamma_i \gamma_j \e^2, \quad
  \LL^3(D^1_\e \cap D^2_\e \cap D^3_\e) = \gamma_1\gamma_2 \gamma_3 \e^3,
 $$ 
 and thus
\begin{align*}
& \LL^3(D^i_\e \cup D^j_\e)  =  (\gamma_i + \gamma_j) \e - \gamma_i \gamma_j \e^2, \\
& \LL^3(D_\e)  =  \e -\big(\gamma_1\gamma_2+\gamma_2\gamma_3+\gamma_1\gamma_3\big)\e^2 + 2\gamma_1\gamma_2\gamma_3\e^3.
\end{align*}
Set
\[
  \mu_\e := \frac{1}{\e}\LL^3\res D_\e
\]
and
\[
  \sigma_\e := \begin{pmatrix}
    \dfrac{\alpha_1}{1-\gamma_1} \chi_{D^2_\e \cup D^3_\e} & 0 & 0 \\
   0 & \dfrac{\alpha_2}{1-\gamma_2} \chi_{D^1_\e \cup D^3_\e} & 0 \\
    0 & 0 & \dfrac{\alpha_3}{1-\gamma_3} \chi_{D^1_\e \cup D^2_\e}
  \end{pmatrix},
\]
which satisfy
\[
\mu_\e\wsto\mu\quad\text{ in }\M(\R^3),\qquad \sigma_\e\mu_\e \wsto \sigma\mu \quad \text{ in }\M(\R^3;\mathbb M^{3 \times 3}_{\rm sym})
\]
and
\[
\dive(\sigma_{\e}\mu_{\e})=0.
\]

Then, the energy is now given by
\begin{align*}
  \mathscr E_\e(\sigma_\e,\mu_\e) &= \int_{D_\e} \frac12 \abs{\sigma_\e}^2 \; \frac{\di x}{\e} \\
  &= \frac{1}{2\e} \biggl( \frac{\alpha_1^2}{(1-\gamma_1)^2}  \LL^3(D^2_\e \cup D^3_\e) + \frac{\alpha_2^2}{(1-\gamma_2)^2}  \LL^3(D^1_\e \cup D^3_\e) \\
  &\qquad\qquad + \frac{\alpha_3^2}{(1-\gamma_3)^2} \LL^3(D^1_\e \cup D^2_\e) \biggr) \\
  &= \frac12 \biggl( \frac{\alpha_1^2}{1-\gamma_1} + \frac{\alpha_2^2}{1-\gamma_2} + \frac{\alpha_3^2}{1-\gamma_3} \biggr)  + \BigO(\e)  \\
  &= \frac{1}{4} \bigl(\abs{\alpha_1} + \abs{\alpha_2} + \abs{\alpha_3}\bigr)^2  + \BigO(\e) \\
  &=\bar{j}^*(\sigma) + \BigO(\e) \\
  &= \overline{\mathscr E}(\sigma,\mu) + \BigO(\e).
\end{align*}
so again we have established the form of $\bar{j}^*$ in the present case. We finally note that if equality holds in~\eqref{eq:ex3cond}, then the constructions of Case~II is the same as the one in Case~I; correspondingly, in this situation also the two cases in the expression~\eqref{eq:jbs_elast} for $\bar j^*(\sigma)$ agree.
\end{example}


\begin{thebibliography}{10}

\bibitem{Allaire02book}
{\sc G.~Allaire}, {\em {Shape optimization by the homogenization method}},
  vol.~146 of Applied Mathematical Sciences, Springer, 2002.

\bibitem{AllaireBonnetierFrancfortJouve97}
{\sc G.~Allaire, E.~Bonnetier, G.~Francfort, and F.~Jouve}, {\em Shape
  optimization by the homogenization method}, Numer. Math., 76 (1997),
  pp.~27--68.

\bibitem{AllaireKohn93}
{\sc G.~Allaire and R.~V. Kohn}, {\em Optimal design for minimum weight and
  compliance in plane stress using extremal microstructures}, European J. Mech.
  A Solids, 12 (1993), pp.~839--878.

\bibitem{AmbrosioFuscoPallara00book}
{\sc L.~Ambrosio, N.~Fusco, and D.~Pallara}, {\em {Functions of bounded
  variation and free-discontinuity problems}}, Oxford Mathematical Monographs,
  Oxford University Press, 2000.

\bibitem{AnsiniDalMasoZeppieri14}
{\sc N.~Ansini, G.~Dal~Maso, and C.~I. Zeppieri}, {\em New results on
  {$\Gamma$}-limits of integral functionals}, Ann. Inst. H. Poincar\'{e} Anal.
  Non Lin\'{e}aire, 31 (2014), pp.~185--202.

\bibitem{AnsiniGarroni07}
{\sc N.~Ansini and A.~Garroni}, {\em {$\Gamma$}-convergence of functionals on
  divergence-free fields}, ESAIM Control Optim. Calc. Var., 13 (2007),
  pp.~809--828.

\bibitem{ArroyoRabasa19?}
{\sc A.~Arroyo-Rabasa}, {\em Characterization of generalized young measures
  generated by {$\mathcal A$}-free measures}, Arch. Ration. Mech. Anal., 242
  (2021), pp.~235--325.

\bibitem{ArroyoRabasaDePhilippisHirschRindler19}
{\sc A.~{Arroyo-Rabasa}, G.~{De Philippis}, J.~Hirsch, and F.~Rindler}, {\em
  Dimensional estimates and rectifiability for measures satisfying linear {PDE}
  constraints}, Geom. Funct. Anal., 29 (2019), pp.~639--658.

\bibitem{BabadjianIurlanoRindler19?}
{\sc J.-F. Babadjian, F.~Iurlano, and F.~Rindler}, {\em Concentration versus
  oscillation effects in brittle damage}, Comm. Pure Appl. Math., 74 (2021),
  pp.~1803--1854.

\bibitem{Bouchitte03}
{\sc G.~Bouchitt\'{e}}, {\em Optimization of light structures: the vanishing
  mass conjecture}, in Homogenization, 2001 ({N}aples), vol.~18 of GAKUTO
  Internat. Ser. Math. Sci. Appl., Gakk\={o}tosho, Tokyo, 2003, pp.~131--145.
\newblock available as arXiv:2001.02022.

\bibitem{BouchitteButtazzo01}
{\sc G.~Bouchitt\'{e} and G.~Buttazzo}, {\em Characterization of optimal shapes
  and masses through {M}onge-{K}antorovich equation}, J. Eur. Math. Soc.
  (JEMS), 3 (2001), pp.~139--168.

\bibitem{BouchitteFragalaSeppecher11}
{\sc G.~Bouchitt\'{e}, I.~Fragal\`a, and P.~Seppecher}, {\em Structural
  optimization of thin elastic plates: the three dimensional approach}, Arch.
  Ration. Mech. Anal., 202 (2011), pp.~829--874.

\bibitem{BouchitteGangboSeppecher08}
{\sc G.~Bouchitt\'{e}, W.~Gangbo, and P.~Seppecher}, {\em Michell trusses and
  lines of principal action}, Math. Models Methods Appl. Sci., 18 (2008),
  pp.~1571--1603.

\bibitem{BourdinKohn08}
{\sc B.~Bourdin and R.~V. Kohn}, {\em Optimization of structural topology in
  the high-porosity regime}, J. Mech. Phys. Solids, 56 (2008), pp.~1043--1064.

\bibitem{BraidesFonsecaLeoni00}
{\sc A.~Braides, I.~Fonseca, and G.~Leoni}, {\em {$\mathcal A$}-quasiconvexity:
  relaxation and homogenization}, ESAIM Control Optim. Calc. Var., 5 (2000),
  pp.~539--577.

\bibitem{ContiMullerOrtiz20}
{\sc S.~Conti, S.~M\"{u}ller, and M.~Ortiz}, {\em Symmetric div-quasiconvexity
  and the relaxation of static problems}, Arch. Ration. Mech. Anal., 235
  (2020), pp.~841--880.

\bibitem{Dacorogna08book}
{\sc B.~Dacorogna}, {\em {Direct methods in the calculus of variations}},
  vol.~78 of Applied Mathematical Sciences, Springer, 2nd~ed., 2008.

\bibitem{DalMaso93book}
{\sc G.~Dal~Maso}, {\em {An introduction to $\varGamma$-convergence}}, vol.~8
  of Progress in Nonlinear Differential Equations and their Applications,
  Birkh\"{a}user Boston, Inc., Boston, MA, 1993.

\bibitem{DeOliveira71}
{\sc G.~N. de~Oliveira}, {\em Matrices with prescribed principal elements and
  singular values}, Canad. Math. Bull., 14 (1971), pp.~247--249.

\bibitem{DePhilippisRindler16}
{\sc G.~{De Philippis} and F.~Rindler}, {\em On the structure of
  $\mathcal{A}$-free measures and applications}, Ann. of Math., 184 (2016),
  pp.~1017--1039.

\bibitem{DePhilippisRindler17}
\leavevmode\vrule height 2pt depth -1.6pt width 23pt, {\em Characterization of
  generalized {Young} measures generated by symmetric gradients}, Arch. Ration.
  Mech. Anal., 224 (2017), pp.~1087--1125.

\bibitem{EkelandTemam76book}
{\sc I.~Ekeland and R.~Temam}, {\em {Convex analysis and variational
  problems}}, North-Holland, 1976.

\bibitem{ErnGuermond}
{\sc A.~Ern and J.-L. Guermond}, {\em Mollification in strongly {L}ipschitz
  domains with application to continuous and discrete de {R}ham complexes},
  Comput. Methods Appl. Math., 16 (2016), pp.~51--75.

\bibitem{Evans10book}
{\sc L.~C. Evans}, {\em {Partial differential equations}}, vol.~19 of Graduate
  Studies in Mathematics, American Mathematical Society, 2nd~ed., 2010.

\bibitem{FonsecaMuller99}
{\sc I.~Fonseca and S.~M{\"u}ller}, {\em {$\mathcal{A}$}-quasiconvexity, lower
  semicontinuity, and {Y}oung measures}, SIAM J. Math. Anal., 30 (1999),
  pp.~1355--1390.

\bibitem{GilbargTrudinger98book}
{\sc D.~Gilbarg and N.~S. Trudinger}, {\em {Elliptic partial differential
  equations of second order}}, vol.~224 of Grundlehren der mathematischen
  Wissenschaften, Springer, 1998.

\bibitem{HofmannMitreaTaylor07}
{\sc S.~Hofmann, M.~Mitrea, and M.~Taylor}, {\em Geometric and transformational
  properties of {L}ipschitz domains, {S}emmes-{K}enig-{T}oro domains, and other
  classes of finite perimeter domains}, J. Geom. Anal., 17 (2007),
  pp.~593--647.

\bibitem{Horgan95}
{\sc C.~O. Horgan}, {\em Korn's inequalities and their applications in
  continuum mechanics}, SIAM Rev., 37 (1995), pp.~491--511.

\bibitem{KohnStrang86_1}
{\sc R.~V. Kohn and G.~Strang}, {\em Optimal design and relaxation of
  variational problems. {I}}, Comm. Pure Appl. Math., 39 (1986), pp.~113--137.

\bibitem{KohnStrang86_2}
\leavevmode\vrule height 2pt depth -1.6pt width 23pt, {\em Optimal design and
  relaxation of variational problems. {II}}, Comm. Pure Appl. Math., 39 (1986),
  pp.~139--182.

\bibitem{KohnStrang86_3}
\leavevmode\vrule height 2pt depth -1.6pt width 23pt, {\em Optimal design and
  relaxation of variational problems. {III}}, Comm. Pure Appl. Math., 39
  (1986), pp.~353--377.

\bibitem{KristensenRaita19?}
{\sc J.~Kristensen and B.~Raita}, {\em Oscillation and concentration in
  sequences of {PDE} constrained measures}.
\newblock arXiv:1912.09190, 2019.

\bibitem{KristensenRindler10}
{\sc J.~Kristensen and F.~Rindler}, {\em Characterization of generalized
  gradient {Young} measures generated by sequences in {W$\sp{1,1}$} and {BV}},
  Arch. Ration. Mech. Anal., 197 (2010), pp.~539--598.
\newblock Erratum: Vol. 203 (2012), 693-700.

\bibitem{LewickaMuller11}
{\sc M.~Lewicka and S.~M\"{u}ller}, {\em The uniform {K}orn-{P}oincar\'{e}
  inequality in thin domains}, Ann. Inst. H. Poincar\'{e} Anal. Non
  Lin\'{e}aire, 28 (2011), pp.~443--469.

\bibitem{LewickaMuller16}
\leavevmode\vrule height 2pt depth -1.6pt width 23pt, {\em On the optimal
  constants in {K}orn's and geometric rigidity estimates, in bounded and
  unbounded domains, under {N}eumann boundary conditions}, Indiana Univ. Math.
  J., 65 (2016), pp.~377--397.

\bibitem{Michell04}
{\sc A.~Michell}, {\em The limits of economy of material in frame-structures},
  Phil. Mag., 8 (1904), pp.~589--597.

\bibitem{Murat78}
{\sc F.~Murat}, {\em Compacit\'e par compensation}, Ann. Sc. Norm. Super. Pisa
  Cl. Sci., 5 (1978), pp.~489--507.

\bibitem{Murat79}
\leavevmode\vrule height 2pt depth -1.6pt width 23pt, {\em Compacit\'e par
  compensation. {II}}, in Proceedings of the {I}nternational {M}eeting on
  {R}ecent {M}ethods in {N}onlinear {A}nalysis ({R}ome, 1978), Pitagora
  Editrice Bologna, 1979, pp.~245--256.

\bibitem{Murat81}
\leavevmode\vrule height 2pt depth -1.6pt width 23pt, {\em Compacit\'e par
  compensation: condition n\'ecessaire et suffisante de continuit\'e faible
  sous une hypoth\`ese de rang constant}, Ann. Sc. Norm. Super. Pisa Cl. Sci.,
  8 (1981), pp.~69--102.

\bibitem{Nitsche81}
{\sc J.~A. Nitsche}, {\em On {K}orn's second inequality}, RAIRO Anal.
  Num\'{e}r., 15 (1981), pp.~237--248.

\bibitem{Olbermann17}
{\sc H.~Olbermann}, {\em Michell trusses in two dimensions as a {$\Gamma
  $}-limit of optimal design problems in linear elasticity}, Calc. Var. Partial
  Differential Equations, 56 (2017), pp.~Paper No. 166, 40.

\bibitem{Olbermann20}
\leavevmode\vrule height 2pt depth -1.6pt width 23pt, {\em Michell truss type
  theories as a {$\Gamma$}-limit of optimal design in linear elasticity}, Adv.
  Calc. Var., to appear (2020).

\bibitem{Palombaro10}
{\sc M.~Palombaro}, {\em {$\text{Rank-}(n-1)$} convexity and quasiconvexity for
  divergence free fields}, Adv. Calc. Var., 3 (2010), pp.~279--285.

\bibitem{Rindler18book}
{\sc F.~Rindler}, {\em {Calculus of Variations}}, Universitext, Springer, 2018.

\bibitem{Rockafellar70book}
{\sc R.~T. Rockafellar}, {\em {Convex Analysis}}, vol.~28 of Princeton
  Mathematical Series, Princeton University Press, 1970.

\bibitem{Sing76}
{\sc F.~Sing}, {\em Some results on matrices with prescribed diagonal elements
  and singular values}, Canad. Math. Bull., 19 (1976), pp.~89--92.

\bibitem{Tartar79}
{\sc L.~Tartar}, {\em Compensated compactness and applications to partial
  differential equations}, in Nonlinear analysis and mechanics: {H}eriot-{W}att
  {S}ymposium, {V}ol. {IV}, vol.~39 of Res. Notes in Math., Pitman, 1979,
  pp.~136--212.

\bibitem{Temam84book}
{\sc R.~Temam}, {\em {Navier-Stokes equations}}, vol.~2 of Studies in
  Mathematics and its Applications, North-Holland, 3rd~ed., 1984.

\bibitem{Ziemer89book}
{\sc W.~P. Ziemer}, {\em {Weakly differentiable functions}}, vol.~120 of
  Graduate Texts in Mathematics, Springer, 1989.

\end{thebibliography}
\end{document}